\documentclass{article}
\usepackage{amsfonts,amsmath,amssymb,amsthm,bbm}
\usepackage{float,mathtools}
\usepackage[shortlabels]{enumitem}
\usepackage[ruled]{algorithm}
\usepackage{mathrsfs}  
\usepackage{algpseudocode}
 \usepackage[numbers]{natbib}
 \usepackage{chngcntr}
\usepackage{apptools}

\newcommand{\cK}{\mathcal{K}}
\newcommand{\cN}{\mathcal{N}}

\usepackage{multirow}
\usepackage{array}
\newcolumntype{C}[1]{>{\centering\arraybackslash}p{#1}}

\usepackage{color}

\newtheorem{example}{Example}
\newtheorem{theorem}{Theorem}[section]
\newtheorem*{theorem*}{Theorem}
\newtheorem{lemma}{Lemma}[section]
\newtheorem{proposition}{Proposition}[section]

\newtheorem{assumption}{Assumption}

\providecommand{\keywords}[1]{\textbf{\textit{Keywords---}} #1}

\newtheorem{manualassumptioninner}{Assumption}
\newenvironment{manualassumption}[1]{%
  \renewcommand\themanualassumptioninner{#1}%
  \manualassumptioninner
}{\endmanualassumptioninner}

\newtheorem{remark}{Remark}

\DeclareMathOperator*{\argmin}{argmin}

\usepackage{graphicx}
\usepackage{authblk}

\title{Conformal Prediction for Dyadic Regression Under Complex Missingness}
\author[1]{Robert Lunde}
\author[1]{Minjie Yang}
\author[2]{Elizaveta Levina}
\author[2]{Ji Zhu}

\affil[1]{Department of Statistics and Data Science, Washington University in Saint Louis}
\affil[2]{Department of Statistics, University of Michigan}

\date{}

\begin{document}

\maketitle

\begin{abstract}
We develop a framework for conformal prediction in dyadic 
regression problems under complex missingness mechanisms. 
At the theoretical level, we develop general technical tools for establishing finite-sample validity of conformal prediction under distributional invariance conditions weaker than exchangeability. A key result handles the 
case where the sample itself is a random subset of the 
index set, a setting not covered by existing theory. In addition, we propose conformal prediction procedures for jointly exchangeable arrays, 
including full conformal, split conformal, a row-column 
approach exploiting similarities within rows and columns, and a selective conformal procedure achieving 
mask-conditional validity. For missing elements, we 
establish asymptotic validity of a weighted 
conformal procedure under a nonparametric graphon model for the missingness mechanism. We further establish conditional
validity results for both continuous and discrete
responses; to the best of our knowledge, this is the first formal proof of asymptotic conditional validity for weighted conformal prediction under a missing-not-at-random assumption. The proposed methods are illustrated 
on synthetic and real-world network data.
\end{abstract}


\section{Introduction}
A fundamental problem in network analysis is dyadic regression, where the goal is to predict the value of entries in a matrix $Y = (Y_{ij})_{i,j \in [n]}$, which summarizes relations (edges) between $n$ objects (nodes).  In general, this task can be accomplished using both node and edge covariates $(X_i,X_j, E_{ij})_{i,j \in [n]}$, and node and edge summary statistics $\mathcal{Z} = (\mathcal{Z}_i,\mathcal{Z}_j, \mathcal{Z}_{ij})_{i,j \in [n]}$ computed from $V =(Y_{ij},X_i,X_j, E_{ij})_{i,j \in [n]}$. One important special case is network link prediction, where $Y$ is binary. Link prediction methods have been used to predict interactions between proteins \citep{proteinlinkprediction}, model interactions between drugs \citep{drug-link-prediction}, and predict associations in criminal networks \citep{10.1371/journal.pone.0154244}, to name a few applications.  Dyadic regression has also been applied to modeling trade networks between countries in economics \citep{GRAHAM202023}. 

While many approaches have been developed to generate predictions in dyadic regression problems, uncertainty quantification for such predictions has proven to be challenging for several reasons.  As in all  problems involving network data, elements of $Y$ typically exhibit a complex dependence structure.  In addition, with dyadic regression, missing data is often the norm rather than the exception, and the missingness mechanism itself can be highly complex.  For example, in link prediction problems, missingness may depend on the underlying probability of edge formation \citep{doi:10.1080/10618600.2017.1286243}.  In these cases, even the choice of the pair of nodes to predict the link for may depend on whether the observation was missing, which leads to additional complications for uncertainty quantification.  

For a large class of network-assisted regression problems where the response is a node-level quantity, \citet{lunde2023conformal} showed that conformal prediction is finite-sample valid under a joint exchangeability assumption and mild conditions on network summary statistics.  Joint exchangeability is a natural invariance assumption for network data, which, at a high level, asserts that the distribution is invariant under node  relabeling.  This assumption is also natural in the dyadic regression setting.

In this paper, we propose a conformal prediction algorithm for attaching uncertainty to predictions made in dyadic regression, including link prediction, and study its properties under a joint exchangeability assumption on the array $(V,M)$, which includes $M$, the binary matrix with $M_{kl} = 1$ indicating that 
the pair $(k,l)$ is observed and $M_{kl} = 0$ indicating that it is not.               
We show that appropriate network variants of conformal prediction have the following finite-sample guarantee: for any $(k,l)$ with $k \neq l$, 
\begin{align*}
P\left(Y_{kl} \in \widehat{C}_{kl} \ \bigr\rvert \ M_{kl} = 1 \right) \geq 1-\alpha , 
\end{align*}
where the prediction set $\hat{C}$ is constructed using nonconformity scores that were also sampled according to the matrix $M$.  When the nonconformity scores themselves have an array structure, there are several possibilities for which scores to use for calibration.  

For instance, one may consider using all of the observed upper triangular elements for symmetric arrays.  On the other hand, elements in the same row or column share a node and therefore, these elements may be more similar to each other compared to arbitrary array elements.  In the jointly exchangeable array setting, we propose three ways of constructing conformal prediction sets that, in one form or another, make use of all of the sampled elements and establish finite-sample validity for all methods.  Along the way, we develop general theoretical tools for conformal prediction under distributional invariance 
conditions weaker than exchangeability. These results, presented in Section \ref{sec-beyond-exchangeability}, are of substantial independent 
interest beyond the array prediction setting and establish superuniformity under conditions not previously considered in the literature. To the best of our knowledge, these are also the first results to establish finite-sample validity 
of conformal prediction under general symmetry conditions imposed jointly on the data and the missingness mechanism.  The key challenge is that existing proof strategies fail when observations can leave the sample under group transformations.      

While prediction sets for sampled elements of the matrix that are valid for any effective sample size are valuable, one often wants to predict array elements that are known to be missing.  If the missingness mechanism is not known, one cannot expect to construct finite-sample valid prediction intervals in this setting.  However, we show that, under a mild graphon missingness assumption, a variant of weighted conformal prediction in which weights are estimated using graphon estimation yields asymptotically valid prediction sets for missing observations if the number of missing entries is not too large.  Moreover, with a proper choice of nonconformity score, under appropriate conditions, we will be able to achieve an asymptotic coverage guarantee conditional on a predicted value of $Y_{kl}$, denoted $\hat{Y}_{kl}$, along with the event that observation $(k,l)$ is missing:  
\begin{align*}
\liminf_{n \rightarrow \infty} P\left(Y_{kl} \in \widehat{C}_{kl}^{\mathrm{weighted}} \ \bigr\rvert \ M_{kl} = 0, \hat{Y}_{kl} \geq c \right) \geq 1-\alpha.
\end{align*} 
Here, $\widehat{C}_{kl}^{\mathrm{weighted}}$ is a set constructed based on weighted conformal prediction to be defined later, and $\hat{Y}_{kl}$ is a prediction of $Y_{kl}$ based on some model, which does not need to be specified correctly.  As we discuss in Section \ref{sec-weighted-conformal}, such a guarantee allows one to choose elements with ``interesting" predicted values as test points, essentially providing asymptotically valid post-selection inference.

Importantly, this case falls into the 
missing not at random (MNAR) category, since the probability 
of missingness can depend on unobserved random variables. 
Standard approaches to missing data that assume 
missingness at random (MAR) are therefore not  
applicable, and the joint exchangeability structure 
 plays a crucial role in enabling valid 
inference in this setting.

\section{Related Work}
\label{sec:related-work}

\subsection{Conformal Prediction and Predictive Inference}
Conformal prediction, pioneered by Vovk and colleagues in the late 1990s, has become a widely used tool for uncertainty quantification in statistics and machine learning; see, for example, \citet{vovk-algorithmic-world} and \citet{angelopoulos2021gentle} for overviews. In particular, conformal prediction has recently been considered in various network settings.  \citet{huang2023uncertainty} and \citet{Zargarbashi-gnn} consider node-level response variables, where the test set is randomly selected from a potential pool of calibration and test points. Since they do not assume exchangeability, their guarantees differ from the usual finite-sample marginal coverage guarantees available under exchangeability, and do not apply to a fixed pre-specified test point; these arguments also do not directly extend to situations where test points are chosen adaptively, which we consider in Section \ref{sec-weighted-conformal}. \citet{lunde2023conformal} consider a form of exchangeability for such problems, which is satisfied for a wide range of data-generating mechanisms, and establish guarantees that are analogous to those for i.i.d.\ data. \citet{lunde2023validityconformalpredictionnetwork} further extend these results to various settings involving non-uniform network sampling, including snowball sampling and random walk sampling.

There has also been some work on conformalized link prediction and matrix prediction. In Section \ref{sec-weighted-conformal}, we consider a weighted conformal prediction procedure that is similar to the one studied by \citet{gui2023conformalized}. These authors consider matrix completion in a setting where the underlying matrix is deterministic and entries of the matrix are revealed via independent Bernoulli sampling. Since no form of exchangeability is imposed on the underlying matrix, the authors weigh the nonconformity score for the test point using the worst-case probability over all entries, which is likely overly conservative when additional structure is available. Moreover, their coverage guarantee requires controlling the 
maximum entrywise estimation error of the sampling probabilities, which is a more demanding condition than our Frobenius-norm 
requirement (condition (c) of Theorem 
\ref{theorem-graphon-missigness-unconditional}).  \citet{Liang28042026} further develop this approach, establishing simultaneous group-level guarantees under the assumption of a fixed underlying matrix with known missingness probabilities. Their weighted conformal procedure achieves finite-sample validity but requires exact knowledge of the missingness structure, and because the underlying data are treated as fixed, the analysis does not use the symmetry structure available under joint exchangeability. In contrast, our framework handles unknown missingness mechanisms through the joint exchangeability structure and establishes validity for each individual test point, which is crucial for the adaptively chosen test points we consider in Section \ref{sec-weighted-conformal}.

\citet{shao2023distributionfreematrixpredictionarbitrary} consider conformal prediction with a jointly exchangeable matrix $A$ and a masking matrix $M$. They allow $M$ to be arbitrary and even adversarial,  but they require the
test point to be fixed and independent of $A$, while we instead make mild structural assumptions on $M$ but allow it to depend on $A$, covering cases where the test point is chosen adaptively.      
\citet{du2025conformallinkpredictionfalse} further consider conformalized link prediction under graphon models with the goal of FDR control for multiple missing links. Their guarantees rely on $M$ being independent of $A$ over the entire matrix,
which is a stronger independence requirement than the single-index condition in
\citet{shao2023distributionfreematrixpredictionarbitrary}, and this work is not applicable to all of the settings we consider, including the data-driven test point choice setting.  

\citet{marandon-2024} consider conformalized link prediction under a missingness mechanism that is a special case of the jointly exchangeable structure we study when the response is exchangeable.   The exchangeability they require is only achievable  when the data are fixed and the test points are sampled at random. 
\citet{pmlr-v204-luo23a} consider detection of anomalous edges under the assumption of edge exchangeability. Edge exchangeability can be a natural assumption in some problems, but joint exchangeability is often a more natural notion of symmetry for array prediction tasks that subsumes many commonly used models; see Example \ref{ex-hyperedge-sampling} for further discussion. \citet{pmlr-v202-zaffran23a} consider conformal prediction sets with missing covariates. Although their problem differs from ours, they discuss a notion of mask-conditional validity, which requires conformal prediction sets to be valid given $M$. The authors establish finite-sample validity in cases where the response does not depend on $M$. In contrast, we study cases where the response can depend on $M$, but we do not condition on $M$; in general, this is too much to ask for with jointly exchangeable arrays, which have less structure than exchangeable sequences.

Our work is also related to recent work on predictive inference 
under symmetry assumptions. Our results imply that one can 
directly use conformal prediction in various settings with group 
structure studied in \citet{dobriban2025symmpi}, 
who propose SymmPI, a framework based on randomization tests to construct prediction sets. As 
noted by \citet{ritzwoller2025randomizationinferencetheoryapplications}, 
when group invariance holds, randomization tests offer finite-sample 
validity. Their approach requires recomputing the test statistic at each transformed 
configuration, which typically is not possible when observations 
leave the sample under group transformations, as in the setting of 
Theorem \ref{theorem-super-unform-be}. Their treatment of 
network-structured data further requires the random variables 
associated with the network to be invariant under the automorphism 
group of a fixed network. This condition does not 
account for the randomness in how the network itself was 
generated, unlike joint exchangeability, which is satisfied 
naturally by a broad class of network models. As 
the authors note, determining the automorphism group of a graph is, 
in general, a hard computational problem; even when it can be 
computed efficiently, it is frequently trivial for randomly 
generated networks with little or no graph symmetry, which limits 
the applicability of their approach to the problems we consider. \citet{conditional_symmpi_2026} extend the SymmPI framework to 
near-conditional coverage under group symmetries using a 
multi-accuracy formulation; their framework does not address the 
missing-element setting we consider.

\citet{paul2026probability} develop probability inequalities for order statistics under violations of the i.i.d.\ assumption. While their framework handles approximate violations without structural assumptions, our results exploit the joint exchangeability structure to obtain exact finite-sample guarantees in the array prediction setting. \citet{unifying-conformal-2025} develop a unified framework for conformal prediction methods based on partial information revelation controlled by some random variable, recovering and extending several existing results (e.g., standard conformal prediction  when order statistics of scores are the information revealed). While their framework provides a conceptual unification of methods including weighted and nonexchangeable conformal prediction,  jointly exchangeable arrays and random sampling mechanisms are outside their scope. \citet{huang2026dataaugmentedbootstrapunifying} proposes a bootstrap framework 
for confidence interval construction from approximately 
invariant transformations, recovering conformal 
prediction under exact cyclic group invariance as a 
special case; their approximate invariance results 
rely on Gaussian universality conditions that do not 
apply to conformal prediction, and their framework 
does not address the random sampling mechanisms or 
missing element setting we consider.

\subsection{Dyadic Regression and Link Prediction}
Dyadic regression problems frequently arise in the social sciences, where these models have been used to model trade flows \citep{https://doi.org/10.1002/tie.5060050113}, international relations \citep{Oneal_Russett_1999}, and cooperation in hunter-gatherer groups \citep{apicella-hunter-gatherer}, among others; see \citet{GRAHAM202023} for an overview. In recent years, there has been work on theoretical properties of kernel regression for nonparametric dyadic regression problems \citep{GRAHAM2024105336}, and on extending classical regression inference to this setting \citep{MarrsFosdick2023}.

One special case of dyadic regression is link prediction, which has been studied in multiple disciplines; see, for example, \citet{10.1145/3012704} for a review.  Classical approaches to link prediction often involve constructing measures of node similarity and ranking the likelihood of an edge based on the similarity of its two incident nodes. Many existing methods do not directly take into account the fact that observed entries may not be representative, though there is some work in the matrix completion and prediction literature on  estimation under this non-representative sampling \citep{9762725, GAUCHER2021299}. 

\section{Problem Setup}
\label{sec-problem-setup}
\subsection{Jointly Exchangeable Arrays}
Our key symmetry assumption is made jointly on $V =(Y_{ij}, X_i,X_j,E_{ij})_{i,j \in [n]}$ and $M$. For an array $B = (B_{ij})_{i,j \in [n]}$ and a permutation function $\sigma: [n] \mapsto [n]$, let $B^{\sigma} = (B_{\sigma(i), \sigma(j)})_{i,j \in [n]}$.  We make the following assumption: 

\begin{assumption}[Joint Exchangeability]
\label{assumption-je1}
For any permutation $\sigma:[n] \mapsto [n]$, 
\begin{align*}
(V^\sigma, M^{\sigma}) \stackrel{d}{= } (V,M).
\end{align*}
\end{assumption}
\noindent It may not be immediately obvious what this assumption implies about the masking matrix.  The following examples suggest that this is a mild assumption that covers many natural sampling mechanisms.
\begin{example}[Sampling promising interactions]
\label{ex-promising-interactions}
Suppose that $V$ is jointly exchangeable and:
\begin{align*}
M_{ij} = \mathbbm{1} \{ f(X_i, X_j, \hat{Z}_i, \hat{Z}_j, \hat{Z}_{ij}, E_{ij}) \geq c \}
\end{align*}
for some $f$ and network summary statistics $\hat{Z}_1, \ldots, \hat{Z}_n, (\hat{Z}_{ij})_{1 \leq i, j \leq n}$ satisfying a mild regularity condition (Assumption \ref{assumption-network-covariates}). Under Assumption \ref{assumption-network-covariates}, 
joint exchangeability of $(V,M)$ follows immediately 
from that of $V$ via Proposition 
\ref{theorem-exchangeability-transformations}. 

In this example, $f$ can be thought of as a function estimated on an independent dataset that determines whether a new interaction is interesting or not. 

\end{example}

\begin{example}[Graphon missingness]
Suppose $Y_{ij} = g(\xi_i,\xi_j) + \epsilon_{ij}$, where  $\xi_1, \ldots, \xi_n$ are unobserved, i.i.d.\ latent variables, the variables $(\epsilon_{ij})_{1 \leq i,j \leq n}$ are jointly exchangeable, and  
\begin{align*}
M_{ij} \sim \mathrm{Bernoulli}(w(\xi_i,\xi_j)),
\end{align*}
for a graphon $w:[0,1]^2 \mapsto [0,1]$.  See Section \ref{sec-weighted-conformal} for an example of this setting.  Joint exchangeability of $(V,M)$ may be immediately proven via Proposition 
\ref{theorem-exchangeability-transformations}.       
\end{example}

\begin{example}[Sampling via hyperedge sequences]
\label{ex-hyperedge-sampling}  An observed matrix sometimes arises from multi-way interactions, which can be modeled as hyperedges in a hypergraph.  Let $\mathcal{V} \subseteq \mathbb{N}$ be a finite node set and let $\mathscr{H}_i  \subseteq \mathcal{V}, i \in [m]$ denote a random hyperedge sequence. Let $h_i =  \{v_1, \ldots, v_{k}\}$ denote a potential realization of $\mathscr{H}_i$, 
and consider $h_i^\sigma = \{ \sigma(v_1), \ldots, \sigma(v_{k})\}$ 
for a permutation $\sigma: \mathcal{V} \mapsto \mathcal{V}$.  Suppose that 
\begin{align}
\label{eq-hyper-permutation}
P(\{\mathscr{H}_1 = h_1\} \cap \cdots \cap \{\mathscr{H}_m = h_m\}) = P(\{\mathscr{H}_1 = h_1^\sigma\} \cap \cdots \cap \{\mathscr{H}_m = h_m^\sigma\})  
\end{align}
for all permutations $\sigma$ and any $h_1, \ldots, h_m \subseteq \mathcal{V}$, and  
$$M_{jk} = \mathbbm{1} ( \{j,k\} \subseteq \mathscr{H}_i \text{ for some } i). $$
Then $(V,M)$ is jointly exchangeable (See Proposition \ref{prop-hyperedge-exchangeability} in the Appendix).   
\end{example}
 One intuitive case where assumption \eqref{eq-hyper-permutation} is reasonable is when there are i.i.d.\ latent random variables associated with each node and the probability of a hyperedge involving a given subset of nodes is a function of these latent positions. 
This example makes no assumption about the dependence between $\mathscr{H}_1, \ldots, \mathscr{H}_m$ under the symmetry assumption \eqref{eq-hyper-permutation}.  In contrast, exchangeable (hyper)edge models of \citet{doi:10.1080/01621459.2017.1341413} assume that the sequence $\mathscr{H}_1, \ldots, \mathscr{H}_m$ is exchangeable, but do not place symmetry conditions on the vertices. 

\subsection{Network Covariates}
We consider both nodal and edge network summary statistics, captured by the array $\mathcal{Z} = (\mathcal{Z}_i,\mathcal{Z}_j, \mathcal{Z}_{ij})_{i,j \in [n]}$. Let $\mathcal{X} = (X_i, X_j, E_{ij}, M_{ij})_{1 \leq i,j \leq n}$ so that $\mathcal{Z} = \zeta(\mathcal{X})$ for some measurable function $\zeta$.  
 We make the following assumption:
 \begin{assumption}[Network Covariates]
 \label{assumption-network-covariates}
 For any permutation $\sigma$,
 %
$\mathcal{Z}^{\sigma} = \zeta( \mathcal{X}^\sigma).$
 \end{assumption}

This condition is closely related to a condition imposed on network covariates in \citet{lunde2023conformal}. In the dyadic regression setting, $(\hat{Z}_{ij})_{1\leq i,j \leq n}$ may, for example, be predicted or imputed via matrix completion with or without covariate information. The only requirement for these network statistics is that if we permute the rows and columns of the underlying data array, the output would be permuted accordingly.

 \subsection{Prediction Model}
We also require that the fitted model is permutation invariant, meaning that the fitted values get permuted accordingly and are otherwise unchanged when the input data is relabeled.  This assumption is not needed for network-assisted regression studied in \citet{lunde2023conformal}, but becomes necessary when general exchangeable sampling mechanisms are permitted.

Let $\mathcal{H} = \{ (i,j) \ | \ M_{ij} =1 \}$ denote the sampled indices, and let
\begin{align*}
& \mathcal{T} = (V_{ij}, \mathcal{Z}_{ij})_{(i,j) \in \mathcal{H}},  \quad  \mathcal{T}^\sigma = (V_{ij}^\sigma, \mathcal{Z}_{ij}^\sigma)_{(i,j) \in \mathcal{H}^\sigma},  \\
& \mathcal{I} =(\mathcal{X}_{ij}, \mathcal{Z}_{ij}, M_{ij})_{(i,j) \in [n]^2},  
\quad  \mathcal{I}^\sigma =(\mathcal{X}_{ij}^\sigma, \mathcal{Z}_{ij}^\sigma, M_{ij}^\sigma)_{(i,j) \in [n]^2} , 
\end{align*}
where $\mathcal{H}^{\sigma} = \{(i,j) \ | \ M_{ij}^\sigma = 1\}$. 
Let $\boldsymbol{\hat{\mu}}_{ij} = \hat{\mu}(\mathcal{I}_{ij};\mathcal{T})$ denote the model output for $(i,j)$ and $\boldsymbol{\hat{\mu}} = (\boldsymbol{\hat{\mu}}_{ij})_{1 \leq i,j \leq n}$ denote the corresponding matrix of fitted values. In most cases, we will be interested in the behavior of fitted values for $(i,j) \in \mathcal{H}$.  In fact, $\hat{\mu}( \mathcal{I}_{ij} \ ; \mathcal{T})$ may output $\infty$ when $M_{ij} = 0$. 

\begin{assumption}[Permutation Invariant Model]
\label{assumption-permutation-invariant}
For any permutation $\sigma:[n] \mapsto [n]$,
$
\boldsymbol{\hat{\mu}}^\sigma =\left( \hat{\mu}(\mathcal{I}_{ij}^\sigma ; \mathcal{T}^\sigma )\right)_{(i,j) \in [n]^2}.
$
\end{assumption}
An analogous condition will be imposed for split conformal prediction; see Section \ref{sec:split-conformal}.  The permutation invariant model assumption is most natural when the fitted model is symmetric in arguments related to nodal covariates and node-level summary statistics.  When split conformal prediction with node splitting is considered, this condition is not necessary, and asymmetric models may also be used; these are  considered in more detail in Section \ref{subsec:row-column-conformal}.          

\section{Conformal Prediction for Sampled Elements}
\label{sec-conformal-for-sampled-elements}
\subsection{Full Conformal Prediction}
We now state our main results on validity of conformal prediction for sampled elements.  

Let $s: \mathbb{R} \times \mathbb{R} \mapsto \mathbb{R}$ denote a nonconformity score and let $S_{ij} = s(Y_{ij}, \boldsymbol{\hat{\mu}}_{ij})$ denote the score for the pair $(i,j)$.   Suppose that $S = (S_{ij})_{1 \leq i,j \leq n}$ is a symmetric array.   Define the set of upper diagonal elements sampled by $M$ as $\mathcal{H}' = \{(i,j) \ | \ 1 \leq i < j \leq n, M_{ij} = 1 \}$, and let
\begin{align*}
 \Pi_{kl} = \frac{1}{|\mathcal{H}'|} \sum_{(i,j) \in \mathcal{H}'}
 \mathbbm{1}(S_{ij} \geq S_{kl}).
 \end{align*}
The above expression is closely related to the rank of $S_{kl}$ among elements with indices in $\mathcal{H}'$. A (full) conformal prediction set based on $\Pi_{kl}$ is given by
\begin{align*}
\widehat{C}_{kl} = \{ y_{kl} \ | \ \Pi_{kl} > \alpha  \}. 
\end{align*}
In practice, this prediction set is approximated by evaluating $\Pi_{kl}$ over a grid of possible values for $y_{kl}$. Since $\hat{\mu}$ is fitted using the observation corresponding to $\{k,l\}$, this step involves repeatedly fitting the model over the possible values of $y_{kl}$.   Our first main result establishes the finite-sample validity of full conformal prediction in our setting.

\begin{theorem}[Finite-Sample Validity of Full Conformal Prediction]
\label{theorem-super-uniform-array}
Suppose that Assumptions \refeq{assumption-je1}, \refeq{assumption-network-covariates}, and \refeq{assumption-permutation-invariant} hold, and $(S_{ij})_{1 \leq i,j \leq n}$ is a symmetric array.  Then, for any $1 \leq k < l \leq n$ and any $t \in [0,1]$,
\begin{align*}
P\left( \Pi_{kl} \leq t \ \bigr\rvert \ M_{kl} = 1 \right) \leq t. 
\end{align*}
Consequently, for any $\alpha \in [0,1]$,
\begin{align*}
P\left( Y_{kl} \in \widehat{C}_{kl} \ \bigr\rvert \ M_{kl} = 1 \right) \geq 1-\alpha. 
\end{align*}
\end{theorem}
\begin{remark}
\label{remark-diagonal-elements}
With jointly exchangeable arrays, in many settings diagonal elements may not be of interest, or they may behave differently from the off-diagonal elements, so these elements are excluded from $\mathcal{H}'$.  If they are of interest, one could construct conformal prediction sets based only on these elements.    
\end{remark}


The guarantee of Theorem \ref{theorem-super-uniform-array} is perhaps most natural in the setting of Example \ref{ex-promising-interactions}, where some ``promising'' observations were sampled, and while the response value for the test point is unobserved, the test pair is also promising according to the same criteria.  Such a result is similar in spirit to that of \citet{10.1093/biomet/asae010}, in which conformal prediction is performed subject to a selection rule acting on both the calibration and potential test points, on i.i.d.\  observations; the exchangeable array setting creates several additional challenges and subtleties.  

In particular, our result above does not directly follow from standard theory for conformal prediction.  First, while $(S_{ij})_{1 \leq i,j \leq n}$ is jointly exchangeable, it is not fully exchangeable as a set of data points; that would require distributional invariance under arbitrary swaps of elements in the array.  The result requires new theory that we have developed for conformal prediction beyond exchangeability,  presented in Section \ref{sec-beyond-exchangeability}.  Moreover, in the array setting, it is typically not possible to condition on the outcome of the sampling mechanism. 

\subsection{Split Conformal Prediction}
\label{sec:split-conformal}
In practice, split conformal prediction is more commonly used than full conformal prediction due to computational considerations.  In the array setting, there are many possible ways to split the sampled elements.  Let $M_{1}$ and $M_{2}$ be binary matrices where an entry is $1$ if the corresponding interaction is present in the training set or in the calibration set, respectively.    We assume that the test point is an arbitrary element of $M_2$; our results will hold for any pair in the calibration set. 

In what follows, let $\mathcal{N} \subseteq \{1,\ldots n\}$ denote the set of nodes that have a nonzero probability of belonging to an edge in the graph corresponding to $M_{2}$. For split conformal prediction, we consider the following exchangeability assumption:
\begin{manualassumption}{1'}[Jointly Exchangeable Splitting Mechanism]
\label{assumption-je1-prime} 

Consider the subgroup of permutations \begin{align*}
\Sigma_{\mathcal{N}} = \left\{ \sigma:[n] \mapsto [n] \ | \  \sigma(i) = i  \ \forall \ i \not \in \mathcal{N} \right\}.
\end{align*}
Suppose that for all $\sigma \in \Sigma_{\mathcal{N}}$,
\begin{align*}
(V^\sigma, M_1^\sigma, M_{2}^{\sigma}) \stackrel{d}{=} (V, M_{1}, M_2). 
\end{align*}
and additionally $M \geq M_{1} + M_2$, 
where $\geq$ corresponds to the element-wise order.    
\end{manualassumption}
In Assumption \ref{assumption-je1-prime}, the training set and calibration set are disjoint; however, this is not necessary for the corresponding nonconformity scores to be jointly exchangeable.  The assumption ensures that the test point, for which the response value is unknown, is not in the training set.  If this were not the case, then one is essentially back to the setting of full conformal prediction.  Next, we provide a few examples of such splitting mechanisms.

\begin{example}[Node Splitting]
Suppose that $D_1 \cup D_2$ forms a partition of $\{1, \ldots, n\}$. Let
\begin{align*}
M_{1}(i,j) = M_{ij} \mathbbm{1}(i,j \in D_1), \quad M_{2}(i,j) = M_{ij} \mathbbm{1}(i,j \in D_2).  
\end{align*}
If Assumption \ref{assumption-je1} is satisfied, then Assumption \ref{assumption-je1-prime} also holds.  
 
\end{example}

\begin{example}[Edge Splitting]
\label{ex-5-edge-splitting}
Consider the array $(\gamma_{ij})_{1 \leq i,j \leq n}$, 
where $\gamma_{ij} = \gamma_{ji} \sim \mathrm{Bernoulli}(p)$ are independent of all other random variables.  Let
\begin{align*}
M_1 = (\gamma_{ij} M_{ij})_{1 \leq i,j \leq n}, \quad M_2 = ( (1-\gamma_{ij}) M_{ij})_{1 \leq i,j \leq n}.   
\end{align*}
If Assumption \ref{assumption-je1} is satisfied, then so is \ref{assumption-je1-prime}. This also holds if the sampling probabilities depend on $V$.   

\end{example}

\begin{example}[Selecting a Node Based on Summary Statistics]
\label{ex:column-selection}
Let $\hat{Z}_1, \ldots, \hat{Z}_n$ be some nodal summary statistics satisfying Assumption \ref{assumption-network-covariates}, such as node degree.  Let $\mathrm{rank}(\hat{Z}_r, \{\hat{Z}_1, \ldots, \hat{Z}_n\}) =k$ for some pre-specified $1 \leq k \leq n$.  Let:
\begin{align*}
M_2(i,j) = M_{ij}\mathbbm{1}(j = \mathcal{R}), \quad M_1 = M-M_2 - M_2^T.  
\end{align*}
If Assumption \ref{assumption-je1} is satisfied, so is Assumption \ref{assumption-je1-prime}.  
\end{example}
\begin{remark}
\label{remark-rank-non-unique}
If ties are present and ranks are not uniquely defined, one can consider ranks with respect to a lexicographic order with tie-breaking random variables, which we discuss in Appendix \ref{sec:appendix-1}.
\end{remark}

Before stating our main result for split conformal prediction, we introduce an analogue of a permutation-invariant model with splitting.  Let 
\begin{align}
\mathcal{H}_1 = \{(i,j) \in [n]^2 \ | \ M_1(i,j) = 1 \}, \quad \mathcal{H}_2 = \{(i,j) \in [n]^2 \ | \ M_2(i,j) = 1 \} ,   
\end{align} 
and define the following arrays, corresponding to training data and model inputs:
\begin{align}
\label{eq-arrays-model}
\mathcal{T}^{\mathrm{split}} = (V_{ij}, \mathcal{Z}_{ij})_{(i,j) \in \mathcal{H}_1}, \quad \mathcal{I}^{\mathrm{split}} =(\mathcal{X}_{ij}, \mathcal{Z}_{ij}, M_2(i,j))_{(i,j) \in [n]^2}.
\end{align}
Let $\boldsymbol{\hat{\mu}}^{\mathrm{split}} \in \mathbb{R}^{n \times n}$ denote the matrix of fitted values, with 
\begin{align*}
\boldsymbol{\hat{\mu}}_{ij}^{\mathrm{split}} = \hat{\mu}( \mathcal{I}_{ij}^{\mathrm{split}} ; \ \mathcal{T}^{\mathrm{split}}). 
\end{align*}
For split conformal prediction, we make the following assumption: 
\begin{manualassumption}{3'}[Permutation Invariant Model for Split Conformal Prediction]
\label{assumption-permutation-invariant-split}
For any permutation $\sigma:[n] \mapsto [n]$,
\begin{align*}
\boldsymbol{\hat{\mu}}^{\mathrm{split}, \sigma} =\bigl( \hat{\mu}(\mathcal{I}_{ij}^{\mathrm{split}, \sigma}; \mathcal{T}^{\mathrm{split},\sigma} )\bigr)_{1 \leq i,j \leq n}.
\end{align*}
\end{manualassumption}
Let $\mathcal{H}_2' = \{(i,j)  \ | \ 1 \leq i < j \leq n,  M_2(i,j) = 1 \}$ and define 
\begin{align*}
\Pi_{kl}^{\mathrm{split}} = \frac{1}{|\mathcal{H}_2'|}\sum_{(i,j) \in \mathcal{H}_2'} \mathbbm{1}(S_{ij} \geq S_{kl}),
\end{align*}
where, analogously to full conformal prediction, $S_{ij} = s(Y_{ij}, \boldsymbol{\hat{\mu}}_{ij}^{\mathrm{split}})$.  A prediction set based on $\Pi_{kl}^{\mathrm{split}}$ is given by
\begin{align*}
\widehat{C}_{kl}^{\mathrm{split}} = \left\{ y_{kl} \ | \ \Pi_{kl}^{\mathrm{split}} > \alpha  \right\}.
\end{align*}
Split conformal prediction is computationally efficient since the model is fit only once.   This prediction set can be equivalently expressed in terms of sample quantiles of the nonconformity scores in the calibration set; for the reader's convenience, we state this explicitly in Appendix \ref{subsec-sample-quantiles}. 
\begin{theorem}[Finite-Sample Validity of Split Conformal Prediction]
\label{theorem-split-conformal}
Suppose that Assumptions \refeq{assumption-je1-prime}, \refeq{assumption-network-covariates}, and \refeq{assumption-permutation-invariant-split} hold. Then, for any $1 \leq k < l \leq n$ and $t \in [0,1],$
\begin{align*}
P\left( \Pi_{kl}^{\mathrm{split}} \leq t \ \bigr\rvert \ M_{2}(k,l) = 1 \right) \leq t. 
\end{align*}
Consequently, for any $\alpha \in [0,1],$
\begin{align*}
P\left(Y_{kl} \in \widehat{C}_{kl}^{\mathrm{split}} \ \bigr\rvert \ M_{2}(k,l) = 1 \right) \geq 1-\alpha.
\end{align*}
\end{theorem}
The proof of this result is again based on new theory for conformal prediction beyond exchangeability; results stated in Section \ref{sec-beyond-exchangeability}
imply finite-sample validity for both full and split conformal prediction.

\subsection{Conformal Prediction Sets that Exploit Similarities Within Rows/Columns} 
\label{subsec:row-column-conformal}
While directly using all sampled nonconformity scores to compute prediction set is natural in many ways, in a certain sense this approach does not take full advantage of the structure of exchangeable arrays. The representation theorems of Aldous and Hoover suggest that elements in a given row or column share a common latent random variable, and therefore may be in some sense closer to each other than to an arbitrary element.  Moreover, for asymmetric arrays, the previously described approach does not exploit potential differences in row and column behavior. We now present an alternative that takes advantage of this additional structure. 

Suppose that the entry $(k,l)$ is pre-specified as the test point.  Let $\mathcal{H}_2(k,\cdot) = \{(k,j) \in \mathcal{H}_2, \ j \in [n]  \}$ and $\mathcal{H}_2(\cdot,l) = \{(j,l) \in \mathcal{H}_2, \ j \in [n]  \}$. Then, it makes sense to consider the following quantities:
\begin{align*}
 R_{kl} &= \frac{1}{|\mathcal{H}_2(k,\cdot)|}\sum_{(k,j) \in \mathcal{H}_2} \mathbbm{1}(\{S_{kj} \geq S_{kl} \}  \cap \left\{(k,l) \in \mathcal{H}_2 \right\}) ,  \\
  C_{kl} &= \frac{1}{|\mathcal{H}_2(\cdot,l)|}\sum_{(j,l) \in \mathcal{H}_2} \mathbbm{1}(\{S_{jl} \geq S_{kl} \}  \cap \left\{ (k,l) \in \mathcal{H}_2 \right\}).
\end{align*}
While these quantities exploit potential similarities of observations in a given column or row, they do not make use of all observations in the sample. Consider the array $\mathcal{U} = (R_{ij}, C_{ij})_{1 \leq i,j \leq n}$. While joint exchangeability of $\mathcal{U}$ implies that the row and column ranks are marginally super-uniform, it does not uniquely specify the joint distribution of these statistics.  We propose an approach based on recent work related to multivariate quantiles \citep{10.1214/20-AOS1996,10.1214/21-AOS2136,10.1214/16-AOS1450} to utilize information from both rows and columns to construct a prediction region.  Some recent work has considered conformal prediction regions \citep{klein2025multivariateconformalpredictionusing} and permutation tests \citep{hlavka2023multivariate} based on multivariate ranks; however, these authors make the stronger exchangeability assumption, whereas we only assume joint exchangeability. 

To define multivariate ranks, let $\mu$ be a reference measure.  Usually conditions are placed on $\mu$ to ensure that a population optimal transport problem has desirable properties, but this will not be crucial in our setting, where only a multivariate sample quantile is required. Let $\hat{F}_{\pm}$ denote an optimal transport map transporting the entries $(\mathcal{U}_{ij})_{i,j \in \mathcal{H}_2}$ to a low-discrepancy sequence for $\mu$ of the same cardinality; that is, 
\begin{align*}
\hat{F}_{\pm} = \argmin_{T: T(x) \sim \hat{\mu}} \int \| x - T(x) \| \ d\hat{\nu}(x),
\end{align*}
where
\begin{align*}
\hat{\nu} = \frac{1}{|\mathcal{H}_2|} \sum_{(i,j) \in \mathcal{H}_2}\mathcal{U}_{ij}, \quad \hat{\mu} = \frac{1}{|\mathcal{H}_2|} \sum_{(i,j) \in \mathcal{H}_2} a_{ij},
\end{align*}
and $(a_{ij})_{(i,j) \in \mathcal{H}_2}$ is a set of low discrepancy points from the reference measure $\mu$\footnote{A low-discrepancy sequence of $n$ points from a measure $\mu$ is a deterministic set of points whose empirical distribution approximates $\mu$ more accurately than a typical random sample of the same size; see \citet{niederreiter1992} 
for a formal treatment. The coverage guarantee of Theorem \ref{theorem-row-column} also holds if $(a_{ij})$ are drawn independently from $\mu$, though this introduces additional randomness into the procedure that is not present with the 
deterministic low-discrepancy construction.}.  Consider a nonconformity score of the form $S(u) = \| \hat{F}_{\pm}(u)\|$.  Define the quantity
\begin{align*}
\check{\Pi}_{kl} = \frac{1}{|\mathcal{H}_2|}\sum_{(i,j) \in \mathcal{H}_2}\mathbbm{1}(S(\mathcal{U}_{ij}) \geq S(\mathcal{U}_{kl})). 
\end{align*}
This calibration set slightly differs from the one in Section \ref{sec:split-conformal}, in that both upper- and lower-diagonal entries are included.  When the array is symmetric, this means double-counting the same nonconformity score, which does not affect the validity of conformal prediction, although one has to deal with the minor complication that the test point also appears twice.  When the array is asymmetric, including both potentially increases the effective sample size, which would reduce the size of the corresponding prediction sets.

In what follows, let
\begin{align*}
\widehat{C}_{kl}^{\mathrm{row-col}} = \{y_{kl}  \ | \ \check{\Pi}_{kl} > \alpha \}.
\end{align*}

\begin{theorem}
\label{theorem-row-column}
Suppose that Assumptions \refeq{assumption-je1-prime} and \refeq{assumption-network-covariates} hold, and either the sampling mechanism is deterministic node splitting or Assumption \ref{assumption-permutation-invariant-split} holds. Then, for any $(k,l)$ and $t \in [0,1]$,  
\begin{align*}
P\left( \check{\Pi}_{kl} \leq t \ \bigr\rvert \ M_{2}(k,l) = 1 \right) \leq t. 
\end{align*}

Consequently, for any $\alpha \in [0,1],$
\begin{align*}
P\left(Y_{kl} \in \widehat{C}_{kl}^{\mathrm{row-col}} \ \bigr\rvert \ M_{2}(k,l) = 1 \right) \geq 1-\alpha.
\end{align*}
\end{theorem}

The proof again makes use of general results beyond exchangeability in Section \ref{sec-beyond-exchangeability}.   While finite-sample validity does not depend on the choice of reference measure $\mu$, width does, and empirically, we have observed that the default choice in the literature, the uniform measure on the unit $\ell_2$ ball,  leads to poor performance.  However, row-column quantiles are all in the first quadrant, 
and thus in our simulation study we instead considered a reference measure that is supported in the first quadrant, the unit ball centered at $(1/2,1/2)$.   Empirical results in Section \ref{sec-ex2} show that this row-column approach to conformal prediction often gives smaller sets than standard split conformal prediction. A  formal comparison of width properties to standard split conformal prediction would require analyzing the joint distribution of optimal transport depth statistics for jointly exchangeable arrays, which is outside the scope of this paper. 

The computational cost of the row-column approach depends on the implementation. For different choices of $y_{kl}$, the optimal transport map $\hat{F}_{\pm}$ needs to be recomputed, leading to increased computational cost.  However, if split conformal prediction is used, only the rank map needs to be recomputed, and the model has to be fit just once.  Moreover, discrete optimal transport between samples of equal sizes is a linear assignment problem, for which efficient polynomial time algorithms exist, and an approximate solution can be computed much faster; see, e.g., \citet{10.5555/2999792.2999868}.  Finally, in standard implementations of full conformal prediction, it is desirable to choose as fine of a grid for $y$ as possible, but this is unnecessary here:  one only needs to consider values of $y$ that lead to jumps in the values of $R_{kl}$ or $C_{kl}$, and there are at most $2n$ such values.  Further, in the practically important case of binary $y_{kl}$, the optimal transport map only needs be computed twice.  Additional implementation details are discussed in Appendix \ref{sec-ex2-appendix}.

\subsection{Selective Conformal Prediction Conditional on Ranks of Network Summary Statistics}
\label{sec-selective-conformal}
In many link prediction problems, certain columns or rows in the matrix may be more likely to contain missing links.  Focusing on these rows or columns can lead to possible increase in the rate of detection of  missing links. In Example \ref{ex:column-selection}, we establish that selecting a column based on the rank of a summary statistic yields a jointly exchangeable splitting mechanism, and therefore Theorem \ref{theorem-split-conformal} applies.   However, in this case,  we can provide a stronger conditional validity guarantee.

Suppose $\mathcal{B} \subseteq \{1, \ldots, n\}$ and $\mathcal{M}_{2,r} = \mathcal{B}$ if $M_{2}(i,r) = 1$ for all $ i \in \mathcal{B}$, and $M_{2}(i,r) = 0 $ for all $ \ i \not\in \mathcal{B}$.  Further, let $\mathcal{R} =r$ if $\mathrm{rank}(\hat{Z}_r, \{\hat{Z}_1, \ldots, \hat{Z}_n\}) =k$ for some pre-specified $1 \leq k \leq n$. For ease of exposition, here we assume that ranks are well-defined; handling ties is discussed in Remark \ref{remark-rank-non-unique}. We have the following result:
\begin{theorem}
\label{thm:column-selective-cp}
Under the same conditions as Theorem \ref{theorem-row-column}, for any $i \in \mathcal{B}$ and any $t \in [0,1]$,
\begin{align*}
P\left(C_{ir} \leq t \ | \ \mathcal{R} =r, \mathcal{M}_{2,r} = \mathcal{B} \right) \leq t.
\end{align*}
\end{theorem}
The proof of this stronger result requires a different strategy from the one used to prove Theorem \ref{theorem-split-conformal}.  While inference after selection is generally a delicate topic, this result suggests that conformal prediction is valid even if the column of interest is adaptively chosen.  Moreover, in this case a form of mask-conditional validity is achievable, which is not typically the case for jointly exchangeable arrays.  Another instance where mask-conditional validity holds is when the selection rule specifies the sampled nodes, and a subarray involving these nodes is observed \citep{lunde2023validityconformalpredictionnetwork}.      

\section{Beyond Exchangeability}
\label{sec-beyond-exchangeability}
To prove finite-sample validity of conformal prediction for arrays, 
we develop general technical machinery for settings where symmetry 
is present but exchangeability does not hold. It turns out that 
finite-sample validity does not depend on the array structure 
itself, but rather on whether the class of invariances 
is rich enough to map any element to a fixed reference element. From this abstract viewpoint, an ``element'' may be an integer 
index, with the relevant invariance being under permutations or cyclic shifts of the index set, or an unordered pair $\{i,j\}$ of 
node indices, where the relevant invariance is under node relabelings 
that map any given pair to the test pair while preserving the 
joint distribution of the array, as is possible under joint 
exchangeability. We believe these results are of substantial 
independent interest beyond the array setting considered here.

We first present a result where all $n$ observations are used. In this case, our result is closely related to classical randomization inference, where invariance under a group yields finite-sample validity; see \citet{ritzwoller2025randomizationinferencetheoryapplications} for an overview. Our proof strategy for this case adapts an argument of \citet{10.1093/biomet/asaa079}, who consider the special case of cyclic permutations.  We then consider the case where the sampling mechanism itself satisfies an invariance property. For the latter case, we use a novel bijection argument that differs from previous work. We start with some additional notation.
Let $W = (W_1, \ldots, W_n)$ be a vector of real-valued random variables, and for a bijection\footnote{While such bijections are technically permutations of $[n]$, we use the more general term since we also consider restricted classes such as cyclic shifts, for which ``permutation'' would be misleading. } $g:[n] \mapsto [n]$, let
\begin{align*}
g W = (W_{g(1)}, \ldots, W_{g(n)}).
\end{align*}
Suppose that for some collection of functions $\mathcal{G}$, the following distributional equivalence holds:
\begin{align}
\label{eq-distributional-invariance}
gW \stackrel{d}{=} W \ \forall g \in \mathcal{G}.
\end{align}

The following proposition establishes that, if $\mathcal{G}$ contains a subset $\mathcal{G}_0$ consisting of $n$ functions that satisfy a certain property, then superuniformity holds.  Thus, it is possible that superuniformity holds even when $|\mathcal{G}| = n$, which allows a much wider class of invariance structures than exchangeability.  Note that superuniformity immediately implies finite-sample validity for conformal prediction sets when the nonconformity scores satisfy the invariance property.    

\begin{proposition}[Superuniformity Beyond Exchangeability]
\label{prop-super-uniform-transitive}
Let $\mathcal{G}$ satisfy (\refeq{eq-distributional-invariance}), and suppose that for some $i \in \{1, \ldots, n\}$, there exists $\mathcal{G}_0 = \{g_1, \ldots, g_n\} \subseteq \mathcal{G}$ such that for each $j \in [n]$, 
\begin{align}
\label{eq-invariance-condition}
g_j(i) = j.
\end{align}
Then, for any $i \in \{1, \ldots, n\}$ and any $t \in [0,1],$ 
\begin{align*}
P\left( \frac{1}{n} \sum_{k=1}^n \mathbbm{1}(W_k \geq W_i) \leq t \right) \leq t. 
\end{align*}
\end{proposition}
When $\mathcal{G}$ is a group action,  requiring (\refeq{eq-invariance-condition}) is equivalent to requiring that the group action is transitive.  Many group actions that are of interest in statistical applications, including the cyclic group and rotational invariance along the surface of a hypersphere, satisfy this condition. 

Proposition \ref{prop-super-uniform-transitive}  can be proved using an argument analogous to the one in Proposition 1 of \citet{10.1093/biomet/asaa079}, which studied  properties of rank statistics under the cyclic group,  also a transitive group action.  For completeness, a proof of Proposition \ref{prop-super-uniform-transitive} is provided in Appendix \ref{sec:appendix-1}. 

\begin{remark}
If $\mathcal{G}$ is a group action and the action of $\mathcal{G}$ on $\{1,\ldots,n\}$ is not transitive, 
an analogous argument applied within each orbit shows that the rank 
of $W_i$ within its own orbit is superuniform, with a correspondingly smaller reference set. \citet{dobriban2025symmpi} 
study a related randomization test based on orbits of a group action; their construction requires evaluating the statistic at every group-transformed configuration in the orbit, which is not 
generally possible when observations can leave the sample under group transformations (i.e., map to unobserved coordinates), as in the setting of Theorem \ref{theorem-super-unform-be}. A related orbit-decomposition idea underlies e-value-based tests of group invariance \citep{koning2026measuringevidenceexchangeabilitygroup}, which similarly presume the group acts on a fully observed sample and do not address this setting.
\end{remark}

Now, we consider a further generalization of this result to handle the case where the sample itself is a random subset of $\{1, \ldots, n\}$. Let:
\begin{align*}
\mathcal{G}^{\mathcal{N}} = \left\{ g|_\mathcal{N} \ s.t. \ g \in \mathcal{G}, \ g|_\mathcal{N} \text{ is a bijection}  \right\}.
\end{align*}
Further suppose that the following condition holds:
\begin{align}
\label{eq-invariance-L}
g L \stackrel{d}{=} L \quad \forall g \in \mathcal{G}^{\mathcal{N}}. 
\end{align}
Unlike the setting of Proposition \ref{prop-super-uniform-transitive}, the
collection $\mathcal{G}^{\mathcal{N}}$ need not be closed under composition, as
a composition of restrictions need not itself be the restriction of an element
of $\mathcal{G}$.  However, this will turn out to be immaterial since these functions generate a group of invariances.  

\begin{lemma}[Invariances Generate a Group]
\label{lemma-invariance-group}
Suppose that $\mathcal{G}^{\mathcal{N}}$ satisfies (\refeq{eq-invariance-L}),
and let $\langle \mathcal{G}^{\mathcal{N}} \rangle$ denote the collection of all finite compositions
$h = h_1 \circ \cdots \circ h_m$, where for each $\ell$ either
$h_\ell \in \mathcal{G}^{\mathcal{N}}$ or
$h_\ell^{-1} \in \mathcal{G}^{\mathcal{N}}$, together with the identity.  Then
$\langle \mathcal{G}^{\mathcal{N}} \rangle$ is a group of bijections of $\mathcal{N}$ containing
$\mathcal{G}^{\mathcal{N}}$, and $hL \stackrel{d}{=} L$ for every $h \in \langle \mathcal{G}^{\mathcal{N}} \rangle$.
\end{lemma}

We now have the following result:

\begin{theorem}[Superuniformity for Sampled Elements Under Distributional Invariance]
\label{theorem-super-unform-be}
Suppose that $\mathcal{G}^{\mathcal{N}}$ satisfies (\refeq{eq-invariance-L}).  Furthermore, suppose that $\langle \mathcal{G}^{\mathcal{N}}\rangle$ is transitive; that is, 
\begin{align}
\label{eq-transitivity-condition}
\text{for every } x, y \in \mathcal{N} \ \text{ there exists } h \in \langle \mathcal{G}^{\mathcal{N}}\rangle
\ \text{ with } h(x) = y ,
\end{align}
Then, for any $t \in [0,1],$ 
\begin{align*}
P\left( \frac{1}{|\mathcal{K}|} \sum_{k \in \mathcal{K}} \mathbbm{1}(W_k \geq W_i) \leq t \ \biggr\rvert \ i \in \mathcal{K}  \right) \leq t. 
\end{align*}

\end{theorem}
\begin{remark}[Comparison to Previous Conditions]
\label{remark-comparison-conditions}
Condition (\refeq{eq-transitivity-condition}) asks only that the available
invariances be rich enough to carry any index to any other, and is therefore
the same requirement as in the setting without a random sampling mechanism.
An earlier version of this manuscript established
Theorem \ref{theorem-super-unform-be} under a slightly stronger hypothesis, requiring a
subcollection $\mathcal{G}_0^{\mathcal{N}}$
with $g_j(j) = i$ for each $j$ and with $g_j(i)$ distinct.  We show in Appendix \ref{sec:appendix-1} that whenever
(\refeq{eq-transitivity-condition}) holds, such a subcollection can always be
extracted, so that the earlier argument applies unchanged under the weaker
hypothesis.
\end{remark}
 
We give two proofs.  The first proceeds by constructing an explicit
measure-preserving bijection between events, together with a lemma showing that
a subcollection satisfying an auxiliary injectivity condition can always be
extracted.  This generalizes an argument posted earlier to arXiv.  The second is more direct, and works with an equivariant functional
of the sampled rank and the membership indicator.  The second proof does not require the
distinct image condition, while the first does not require the sampled
rank to be equivariant, so it may generalize to other settings not considered here. 


In the case of jointly exchangeable arrays, one can represent the invariances as functions acting on off-diagonal ordered pairs or sets of cardinality two, depending on whether the array is directed or undirected.  Both classes are transitive, and therefore, the condition above is satisfied.   



A notable prior work  \citet{10.1214/23-AOS2276} also explores validity of conformal prediction beyond the exchangeable setting.  Our result is complementary:  they consider an underlying data vector $Z = (Z_1, \dots, Z_{n+1})$ and pairwise swaps of the test point, letting $Z^i$ be a vector in which the $i$-th and $(n+1)$th element are swapped.  For a weighted conformal prediction procedure, they derive an upper bound on the coverage gap of the form:
\begin{align}
\label{eq-coverage-gap}
1-\alpha - P( Y_{n+1} \in \widehat{C}_{n+1}) \leq  \sum_{i=1}
^n w_i d_{TV}(Z^{i}, Z).
\end{align}
If the weights are chosen appropriately, the coverage gap may be small.  However, pairwise swaps may not preserve the distribution well for certain invariance structures; they are one possible choice for the collection $\mathcal{G}_0$ introduced in Proposition \ref{prop-super-uniform-transitive}
 when $\mathcal{G}$ is the full permutation group.  For other invariance structures, one may want to consider different functions acting on the indices of $Z$. As a simple example, in the case of the cyclic group, one may consider $Z^{i} = (Z_{\tau(1)}, \ldots, Z_{\tau(n+1)})$, where $\tau$ shifts all observations by $n+1- [i \text{ mod } (n+1)]$, so that $\tau(i) = n+1$.  Our result implies that the coverage gap is $0$ when invariance under the cyclic group holds. 

\section{Array Conformal Prediction for Missing Elements}
\label{sec-weighted-conformal}
In the previous section, we considered an exchangeable sampling mechanism, and focused on constructing a prediction set for one of the sampled observations.  However, in certain settings, we may be interested in constructing a prediction set for a missing point based only on observed data.  In this case, it turns out that a jointly exchangeable missingness mechanism imposes a lot of structure and makes this a tractable problem, at least asymptotically.   Constructing a prediction set for a missing point is inherently a more difficult 
problem than constructing one for a sampled point, with the test point drawn from a potentially different distribution 
than the calibration points, requiring additional 
assumptions. Consequently, our guarantees 
in this section will be asymptotic rather than 
finite-sample, and additional regularity conditions 
on the missingness mechanism and graphon estimator 
will be needed, though these conditions are mild 
in practice.

For this problem, we restrict our attention to a model inspired by representation theorems for jointly exchangeable arrays. 
\begin{assumption}[Dyadic Regression Model with Graphon Missingness]
\label{assumption-regression-missingness}
Suppose that $(X_1, \xi_1, \zeta_1), \ldots, (X_{n}, \xi_{n}, \zeta_{n}), \stackrel{i.i.d.}{\sim} P$ and let $(\eta_{ij})_{1 \leq i < j \leq n}$ and $(\nu_{ij})_{1 \leq i < j \leq n}$ be collections of i.i.d. random variables that are independent of all other random variables.  Suppose that $Y_{ij} = Y_{ji}$ for all $i \neq j$, and for appropriately measurable $g$, $w$, and $i <j$,
\begin{align*}
Y_{ij} = g(X_i,\zeta_i, X_j, \zeta_j, \nu_{ij}), \quad M_{ij} = \mathbbm{1}( \eta_{ij} \leq  \rho_nw(\xi_i, \xi_j) \wedge 1),
\end{align*}
where $\int_0^1 \int_0^1 w(u,v) du dv =1$, $0 <c < w(x,y) < C < \infty$ and $\rho_n \rightarrow 0$.
\end{assumption}
We focus on the more interesting sparse case, where the sparsity parameter $\rho_n \rightarrow 0$. We also assume a lower and upper-bounded graphon, which is a standard assumption in the network literature. 

The definition of $Y_{ij}$ allows heteroscedasticity; for example, one may consider $E(\gamma_{ij}) = 0$, $E(\gamma_{ij}^2) = 1$  and choose $g$ so that
\begin{align*}
Y_{ij} = f(X_i, \zeta_i, X_j, \zeta_j) + \sigma(X_i, X_j) \nu_{ij}.
\end{align*}
There can also be dependence between the latent positions associated with $Y$ and $M$. Edge covariates are also allowed, so long as they are functions of node covariates. A mild but crucial assumption here is that once node-level information is taken into account, any additional randomness is independent of everything else.  

We now consider a setup in which the data analyst trains a model on entries satisfying $M_{ij} = 1$ and then constructs a prediction set for a missing test point (i.e. $M_{ij} = 0$). We study split conformal prediction with deterministic node splitting, which seems to require the fewest assumptions for unconditional validity. Throughout this section, we assume $n$ is even for notational convenience.  Suppose, for concreteness, that the nodes are split into $D_1 = \{1, \ldots, N \}$ and $D_2 = \{ N+1, \ldots, n \}$,  Further, let $\mathcal{H}_1 = \{\{i,j\} \ | \ M_{ij} =1, i,j \in D_1  \}$ denote the training set, $\mathcal{H}_2 = \{\{i,j\} \ | \ M_{ij} = 1, i,j \in D_2 \}$ denote the calibration set, and $\mathcal{K}_{kl} = \mathcal{H}_2 \cup \{k,l\}$.

For our unconditional validity results, we consider a model of the form $\tilde{Y} = \hat{h}(X_i,X_j)$, where $\hat{h}$ is trained on the training data.  We refer to this case as the ``without network covariates'' setting, which appears to require the fewest conditions.  For conditional validity, we will also consider the ``with network covariates" case, where $\hat{Y}_{ij} = \hat{h}(X_i, X_j, \mathcal{Z}_{ij})$, and the model is fit on the elements in $\mathcal{H}_1$. This setup allows for a wide range of network summary statistics, including values imputed through matrix completion.  In link prediction problems, when $M_{ij} = 0$, it is natural to  consider matrix completion in which the corresponding entry of the response matrix is set to $0$.

We consider two different cases for the response value $Y$: binary, corresponding to typical link prediction problems, and continuous, appropriate for problems involving weighted networks.  For both cases, suppose that some model outputs a prediction $\hat{Y}$ for $Y$.  It is then natural to ask for (asymptotic) conditional coverage given $\hat{Y}$; such a conditional guarantee would even allow the choice of the test point to depend on $\hat{Y}$. In modern regression problems involving many covariates, asymptotic conditional validity given $\hat{Y}$ is often far more attainable than conditional coverage given covariates, due to the curse of dimensionality. 

While the practically relevant guarantee involves conditioning 
on the fitted value $\hat{Y}_{ij}$, we establish conditional 
validity with respect to a population-level quantity 
$\tilde{Y}_{ij} = h(X_i,\xi_i,\zeta_i,X_j,\zeta_j,\xi_j)$ 
that $\hat{Y}_{ij}$ approximates, in a manner made precise in 
our theorem statements. Under consistency of $\hat{Y}_{ij}$, 
this implies an asymptotically valid guarantee for test points 
chosen based on $\hat{Y}_{ij}$; see Proposition 
\ref{theorem-asymptotic-conditional-validity-graphon}.

For binary $Y$, we consider a special case of the nonconformity score proposed by \citet{NEURIPS2020_244edd7e}. 
Let $\hat{\pi}_y(\tilde{y} )$ be an estimate of the quantity $\pi_y(\tilde{y}) = P(Y_{ij}=y \ | \ \tilde{y}, M_{ij}=0)$. Let $\pi = (\pi_0, \pi_1)$ and $\hat{\pi} = (\hat{\pi}_0, \hat{\pi}_1)$.  For simplicity, assume $\hat{\pi}_0(u) + \hat{\pi}_1(u) =1$ for all $u \in \mathbb{R}$. Define the nonconformity score by 
\begin{align*}
\tilde{s}(y, \hat{y};\hat{\pi}) = \begin{cases}
0 &  \hat{\pi}_y(\hat{y}) \geq 0.5 \\
1-\hat{\pi}_y(\hat{y}) & \hat{\pi}_y(\hat{y}) < 0.5. 
\end{cases}
\end{align*}
To rigorously establish asymptotic conditional validity in the binary case, we will consider the following modified nonconformity score:
\begin{align}
\label{eq-score-binary}
s(y, \hat{y};\hat{\pi})  =  \tilde{s}(y, \hat{y};\hat{\pi}) \vee (1-\alpha).
\end{align}
For continuous $Y$, we consider a nonconformity score based on an estimate of the conditional CDF studied by \citet{Chernozhukove2107794118}. For notational simplicity, let $\tilde{F}(y_{ij}  \ | \  \tilde{y}_{ij})$ denote the conditional CDF of $Y_{ij}$ given $\tilde{Y}_{ij}$ and $M_{ij} = 0$.  Let $\tilde{F}_n$ be an estimate of $\tilde{F}$.  Define the nonconformity score as
\begin{align}
\label{eq-score-continuous}
s(y,\hat{y},\tilde{F}_{n}) = \left| \frac{1}{2} - \tilde{F}_{n}(y \ | \ \hat{y})\right|.
\end{align}
Since our goal is to construct prediction intervals for points that are different from those in the calibration set, we consider a weighted conformal prediction procedure \citep{NEURIPS2019_8fb21ee7} that places a higher weight on observations in the calibration set that are more similar to the test point.  In our setting, graphon estimation (see, for example, \citep{d6cb28bb-5a7b-386d-8ced-b68d43954df3,gao2015rate,klopp-oracle-inequalities,repec:oup:biomet:v:104:y:2017:i:4:p:771-783.}) may be used to estimate the weights.  Let $\hat{P} = (\hat{p}_{ij})_{1 \leq i,j \leq n}$ be an estimator of the matrix $P = (\rho_n w(\xi_i,\xi_j))_{1 \leq i, j \leq n}$.

For any $k,l \in D_2$ such that $M_{kl} = 0$, define the quantity: 
\begin{align*}
\Pi_{kl}^{\mathrm{weighted}} =\frac{1}{{N \choose 2}(1-\hat{\rho}_n)} \ \sum_{\{i,j\} \in \mathcal{K}_{kl}} \frac{1-\hat{p}_{ij}}{\hat{p}_{ij}} \mathbbm{1}(S_{ij} \geq S_{kl} ),
\end{align*}
where $\hat{\rho}_n = \frac{1}{{N \choose 2}} \sum_{N+1 \leq i<j \leq n} M_{ij}$.   The prediction set for $\{k,l\}$ is given by:
\begin{align*}
\widehat{C}_{kl}^{\mathrm{weighted}} = \left\{ y_{kl}  \ | \   \Pi_{kl}^{\mathrm{weighted}} > \alpha \right\}. 
\end{align*}
As discussed in Section \ref{sec:related-work}, this procedure is related to \citet{gui2023conformalized}, 
but differs in key ways: we consider a jointly 
exchangeable rather than a deterministic matrix, 
establish validity for any individual missing entry 
rather than in expectation, and handle unknown 
missingness probabilities through graphon estimation. Finally, we also consider conditional validity, which allows guarantees for test points that are adaptively chosen based on fitted values of a model.      

The following theorems establish conditions under which asymptotic unconditional and conditional validity holds.  Our unconditional validity result below holds for any choice of nonconformity score. 

\begin{theorem}[Asymptotic Validity of Graphon-Weighted Conformal Prediction]
\label{theorem-graphon-missigness-unconditional}
Suppose that Assumption \ref{assumption-regression-missingness} holds and that $\hat{Y}_{ij} = \hat{h}(X_i,X_j)$.  Moreover, 
\begin{enumerate}
\item[(a)] (Lower-bounded Graphon and Graphon Estimator)  $w(x,y) \geq c > 0$ and $\hat{p}_{ij}/\rho_n \geq c > 0$ for some $c > 0$;  
\item[(b)] (Sparsity level) $\rho_n = \omega(\sqrt{\log n}/n)$; 
\item[(c)] (Graphon Estimation)
\begin{align*}
\frac{1}{n^2 \rho_n^3} E[\|\hat{P} - P \|_F^2] = o(1).
\end{align*} 

Then, for any $k,l \in D_2$ satisfying $k<l$,
\begin{align*}
\liminf_{n \rightarrow \infty} P(Y_{kl} \in \widehat{C}_{kl}^\mathrm{weighted} \ | \ M_{kl} = 0 ) \geq 1-\alpha.
\end{align*}
\end{enumerate}
\end{theorem}

\begin{remark}
The lower bound requirement $\hat{p}_{ij}/\rho_n \ge c$ is a condition imposed strictly on the estimator.   Nonparametric estimators can output vanishing or negative edge probability estimates, and replacing them with $\max(\hat{p}_{ij}, c  \hat{\rho}_n)$, where $\hat{\rho}_n$ is the standard empirical edge density, solves this problem in practice and can be shown to have a negligible impact on the estimation rate under our assumed conditions. 
\end{remark}

We now state an asymptotic conditional validity result, which requires additional assumptions on consistency.  To the best of our knowledge, there are no conditional validity results for weighted conformal prediction in the literature.  In the continuous case, our proofs utilize recently developed empirical process results for exchangeable arrays due to \citet{10.1214/20-AOS1981}.  Our result is also novel in the sense that conditional validity is attained even when the test point is out of sample.     
\begin{theorem}[Asymptotic Conditional Validity of Graphon-Weighted Conformal Prediction] \label{theorem-asymptotic-conditional-validity-graphon}
Suppose that Assumption \ref{assumption-regression-missingness} and conditions $(a) - (c)$ of Theorem \ref{theorem-graphon-missigness-unconditional} hold.  Suppose that there exists a collection of random variables $(\tilde{Y}_{kl})_{1 \leq i,j \leq n}$ such that $|\hat{Y}_{kl} - \tilde{Y}_{kl}| = o_P(1)$ and $\tilde{Y}_{kl} = h(X_i, \xi_i, \zeta_i, X_j, \xi_j, \zeta_j)$ for some measurable $h$.  Suppose one of the four conditions below holds, as appropriate for the setting at hand:    
\begin{enumerate}
 \item[(d1)] Discrete response, without network covariates: 
 \begin{align*}
   E|\hat{\pi}_1(\hat{Y}_{kl}) - \pi_1(\widetilde{Y}_{kl})| = o(1) ;  
\end{align*}
\item[(d2)] Discrete response, with network covariates: 
 \begin{align*}
   E|\hat{\pi}_1(\hat{Y}_{kl}) - \pi_1(\widetilde{Y}_{kl})| = o(\rho_n); 
\end{align*}
 \item[(e1)]  Continuous response, without network covariates:
\begin{align*}
  \left|\tilde{F}_{n}(Y_{kl}\ | \ \hat{Y}_{kl}) - \tilde{F}(Y_{kl} \ | \  \widetilde{Y}_{kl})\right| = o(1) ;  
 \end{align*}
\item[(e2)]  Continuous response, with network covariates: 
\begin{align*}
  \left|\tilde{F}_{n}(Y_{kl}\ | \ \hat{Y}_{kl}) - \tilde{F}(Y_{kl} \ | \  \widetilde{Y}_{kl})\right| = o(\rho_n).   
 \end{align*}
 \end{enumerate}
    
    Then, if the nonconformity scores (\refeq{eq-score-binary}) or (\refeq{eq-score-continuous}) are used in the discrete or continuous cases, respectively, for any $\delta > 0$,
    \begin{align*}
     P(Y_{kl} \in \widehat{C}_{kl}^\mathrm{weighted} \ | \  \widetilde{Y}_{kl}, M_{kl} = 0 ) \geq  1-\alpha - R_n,   
    \end{align*}
     where $R_n \geq 0$ and $P(R_n > \epsilon) \rightarrow 0$ for any $\epsilon >0$.
\end{theorem}

The theorem  also holds  under condition (d2) or (e2) when edge splitting is used instead of node splitting. 

\begin{remark}
Explicit convergence rates for Theorems \ref{theorem-graphon-missigness-unconditional}
 and \ref{theorem-asymptotic-conditional-validity-graphon}
involve several factors: the graphon 
estimation rate in condition (c), the empirical 
process approximation for jointly exchangeable arrays, 
and the IPW correction for the distributional shift 
between observed and missing entries, making a 
unified rate statement difficult without restricting 
the generality of the framework. For specific graphon 
estimators and smoothness classes, rates can be 
derived from condition (c) combined with the oracle 
inequalities of \citet{klopp-oracle-inequalities}. 
We also note that explicit convergence rates for 
asymptotic conditional validity results are not 
currently available in the broader conformal 
prediction literature, even under simpler data 
generating mechanisms.
\end{remark}

\begin{remark}
 The stronger conditions in the ``with network covariates" case arise from the potential dependence between the network summary statistics computed on the training set and the calibration scores, which introduces an additional factor of $\rho_n^{-1}$ in the required convergence rate. This dependence vanishes when node splitting is used without full network covariates, since the training and calibration sets involve disjoint node sets. In practice, the dependence between training and calibration sets induced by network statistics is likely to be weaker than this worst-case bound.  

\end{remark}

Theorem \ref{theorem-asymptotic-conditional-validity-graphon} is for the asymptotic regime when the number of observations used to fit $\hat{Y}(\cdot)$ is growing; however, if $\hat{Y}$ is trained (or pre-trained) on a fixed number of points, then convergence of $\hat{Y}$ to a limiting model is not required.  Note that $\widetilde{Y}$ is not assumed to be a consistent estimator of any quantity.  However, in general, standard classifiers and conditional CDF estimators do not necessarily achieve the $L_1$ consistency property required by conditions (d) and (e); inverse probability-weighted (IPW) estimators are often required to achieve consistency.

For a simple example, consider kernel-based estimation of the conditional probability $P(Y=1 \ | \ \tilde{y}, M_{ij} = 0)$ for a binary response; the continuous case is analogous.  An estimator of $\pi_1(\widetilde{y})$ is given by:
\begin{align}
\label{eq-kernel-ipw}
\hat{\pi}_1(\widetilde{y}) = \sum_{\{i,j\} \in \mathcal{H}_1 } \kappa(\hat{Y}_{ij}, \tilde{y}) Y_{ij},  \quad
\kappa(\hat{y}_{ij}, \widetilde{y}) = 
\frac{K\left(\frac{|\hat{Y}_{ij} - \tilde{y}|}{h}\right)\frac{1-\hat{p}_{ij}}{\hat{p}_{ij}}}{ \sum_{\{i,j\} \in \mathcal{H}_1} K\left(\frac{|\hat{Y}_{ij} - \tilde{y}|}{h}\right)\frac{1-\hat{p}_{ij}}{\hat{p}_{ij}} }
\end{align}
for an appropriate choice of kernel function $K:\mathbb{R}^+ \mapsto \mathbb{R}^+$.  Furthermore, let $M' = (M_{ij}')_{1 \leq i,j \leq n}$ be a coupled version of $M$, where $M_{ij}' = \mathbbm{1}(\eta_{ij}' \leq \rho_n w(\xi_i, \xi_j))$ and $(\eta_{ij}')_{1 \leq i < j \leq n}$ is a collection of i.i.d.\ $ \mathrm{Uniform}[0,1]$ random variables independent of all other random variables.  Let $\mathcal{H}_1' = \{\{i,j\} \ : \ M_{ij}' = 0, i,j \in D_1\}$.  Define another kernel regression estimator by 
\begin{align}
\label{eq-kernel-approx}
\tilde{\pi}_1(\widetilde{y}) = \sum_{\{i,j\} \in \mathcal{H}_1' } \kappa(\tilde{Y}_{ij}, \tilde{y}) Y_{ij},  \quad
\kappa(\tilde{Y}_{ij}, \widetilde{y}) = 
\frac{K\left(\frac{|\tilde{Y}_{ij} - \tilde{y}|}{h}\right)}{ \sum_{\{i,j\} \in \mathcal{H}_1'} K\left(\frac{|\tilde{Y}_{ij} - \tilde{y}|}{h}\right) }.
\end{align}

We then have the following result. 
\begin{proposition}
\label{prop-kernel-ipw}
Suppose that $K:\mathbb{R}^+ \mapsto \mathbb{R}^+$ is a bounded and Lipschitz kernel function. Suppose that one of the following conditions hold:
\begin{enumerate}
    \item[(a)] (Without Network Covariates) \\
    $\max_{i,j \in D_1}\frac{E|\hat{Y}_{ij} - \tilde{Y}_{ij}|}{h} = o(1)$ or $\max_{i,j \in D_1} \frac{|\hat{Y}_{ij} - \tilde{Y}_{ij}|}{h} = o_P(1)$
    \item[(b)] (With Network Covariates) \\ 
    $\max_{i,j \in D_1}\frac{E|\hat{Y}_{ij} - \tilde{Y}_{ij}|}{h} = o(\rho_n)$ or $\max_{i,j \in D_1} \frac{|\hat{Y}_{ij} - \tilde{Y}_{ij}|}{h} = o_P(\rho_n)$.
\end{enumerate}
Then, for any $k,l \in D_2$ satisfying $k<l$,

\begin{align*}
\hat{\pi}_1(\hat{Y}_{kl}) - \tilde{\pi}_1(\hat{Y}_{kl}) = o_P(1).
\end{align*}
\end{proposition}
To establish conditional validity, it may often be easier to approximate an inverse-probability-weighted estimator in this manner before appealing to consistency results. Properties of kernel regression for nonparametric dyadic regression have been studied by \citet{GRAHAM2024105336}.  

We now discuss the graphon estimation condition $(c)$.  In many cases, this assumption leads to a stronger condition on the sparsity level than condition (b).  In particular, for stochastic block models and Lipschitz graphons, the results of \citet{klopp-oracle-inequalities} imply that $\rho_n = \omega(1/\sqrt{n})$ is required in a minimax sense.        

It may not be immediately clear what conditional validity with respect to $\widetilde{Y}$ implies about coverage properties of confidence intervals in which test points are chosen adaptively based on values of $\hat{Y}$ rather than $\widetilde{Y}$.  Suppose that a data analyst chooses a test point to analyze if $T_{kl} =1$, where $T_{kl} \sim \mathrm{Bernoulli}(f(\hat{Y}_{kl}))$.  This setup includes the case where $T_{kl} = \mathbbm{1}(\hat{Y}_{kl} > c)$, which was discussed in the introduction of the paper.  For concreteness, suppose that $T_{kl} = \mathbbm{1}(\gamma_{kl} \leq f(\hat{Y}_{kl}))$, where $(\gamma_{kl})_{1\leq k,l \leq n}$ are mutually independent $\mathrm{Uniform}[0,1]$ random variables independent of all other random variables. We then have the following result. 

\begin{proposition}[Selective Coverage of Conformal 
Prediction Sets]\label{prop-selective-coverage} 
Suppose that the conditions in Theorem 
\ref{theorem-asymptotic-conditional-validity-graphon} 
are satisfied and $\widetilde{Y}_{ij}$ is continuously 
distributed with bounded density for all 
$1 \leq i < j \leq n$. 
Further suppose that $f$ is piecewise Lipschitz with 
a finite number of jump discontinuities, and there exists $A \subseteq \mathbb{R}$ 
such that $P(\widetilde{Y}_{kl} \in A) \geq \delta_0 >0$ 
and $f(y) \geq \gamma > 0$ for all $y \in A$.
Then, for any $k,l \in D_2$ satisfying $k<l$ and any $\delta > 0$, 
for $n$ large enough, 
\begin{align*}
P( Y_{kl} \in \widehat{C}_{kl}^\mathrm{weighted} \ | \ 
M_{2}(k,l)=0, T_{kl} = 1)\geq 1 -\alpha -\delta.
\end{align*}
Moreover, for $n$ large enough, with probability 
tending to $1$, 
\begin{align*}
\frac{\sum_{N+1 \leq k < l \leq n} \mathbbm{1}(\{Y_{kl} \in \widehat{C}_{kl}^\mathrm{weighted} \} \cap \{M_{2}(k,l) = 0\} \cap \{T_{kl}=1\} )}{ \sum_{N+1 \leq k < l \leq n} \mathbbm{1}(\{M_{2}(k,l) = 0\} \cap \{T_{kl}=1\}) } \geq 1-\alpha -R_n,
\end{align*}
where $P(R_n > \delta) \rightarrow 0$. 
\end{proposition}

Although our approach is asymptotic in nature, it is very robust to the choice of selection rule.  In fact, the selection rule does not need to be known or estimated, which differs from existing methods such as
\citet{10.1093/biomet/asae010}.  Therefore, our approach has fewer ``researcher degrees of freedom,'' which is an important practical consideration in selective inference problems.  Our approach also makes full use of the calibration set and does not discard points that are not chosen by the selection rule, leading to potentially shorter widths.  However, it still does not guarantee validity for certain selection rules such as the sample maximum for $\hat{Y}$, for which the approach of \citet{10.1093/biomet/asae010} also fails.

Proposition \ref{prop-selective-coverage} provides an average coverage guarantee over ``interesting'' points, which is desirable but differs from false discovery rate (FDR) control.  FDR control for $p$-values from conformal prediction was previously studied by \citet{JMLR:v24:22-1176}.  In the jointly exchangeable array setting, it is not immediately obvious whether conformal $p$-values satisfy positive regression dependence, a condition often needed to establish validity of strategies involving the Benjamini-Hochberg procedure.  We leave the exploration of FDR control in the array setting to future work.

\section{Experiments}
\subsection{Finite-Sample Validity for Sampled Elements}
\label{simulations-sampled-elements}
In this section, we study empirical performance of the conformal prediction methods proposed in Section \ref{sec-conformal-for-sampled-elements}. 
First, we compare different splitting mechanisms from the examples in Section \ref{sec:split-conformal}, and then evaluate the row-column approach proposed in Section \ref{subsec:row-column-conformal}. 

\subsubsection{Comparing Splitting Mechanisms}
\label{sec-ex1}
The full description of the data-generating process for the experiments is given in Appendix \ref{sec-ex1-appendix};  here we describe the most important pieces.    For this set of comparisons, we consider the model

\begin{align}
\label{eq-exp-compare-split}
    Y_{ij} &= \text{sigmoid}\!\left(X_{i1}X_{j1}+X_{i2}X_{j2} - E(X_{i1}X_{j1}+X_{i2}X_{j2})\right) \nonumber 
    \\
    &\quad +\, c_1(\xi_{i1}\xi_{j1}+\xi_{i2}\xi_{j2})+c_2\epsilon_{ij},
\end{align}
where $c_1=0.9$, $c_2 = 0.25$  control the relative contributions of the latent signal and noise, and variables  $\epsilon_{ij}$ are drawn i.i.d.\ $ \mathcal{N}(0, 1)$. The responses are always centered  and scaled to have a standard deviation of $4$ to ensure fair comparisons. The node covariates $[X_1, X_2, X_3, X_4]$ follow a multivariate normal distribution, and the latent positions $\xi_{1}$, $\xi_2$ are constructed by normalizing and rescaling $X_3$, $X_4$ to lie on an ellipse segment within the unit square. 
The observed entries have a complex dependence because latent positions influence both the response $Y_{ij}$ and the missingness mechanism, with the masking matrix entries drawn independently as 
\begin{align*}
    M_{ij}\sim \text{Bernoulli}(\rho_n (\xi_{i1}\xi_{j1}+\xi_{i2}\xi_{j2} )),
\end{align*}
where $\rho_n$ controls the sparsity level. Both $Y$ and $M$ are symmetric with diagonals set to $0$ to represent an undirected graph without self-loops. 

To compare different splitting methods, we fit both a linear model and a random forest model $\hat{\mu}(\cdot)$ with node covariates $X_{1i}, X_{1j}, X_{2i}$ and $X_{2j}$ on observed entries ($M_{ij}=1$).  For both node splitting and edge splitting, we consider only upper-triangular elements.   The test point is pre-specified, without loss of generality, as $Y_{n-1, n}$. For the selected column method, the test column is chosen as the column with the highest degree after removing the test row; additional details are provided in Appendix \ref{sec-ex1-appendix}. Conformal prediction intervals for the test point are constructed using the nonconformity score $S_{ij}= s(Y_{ij}, \boldsymbol{\hat{\mu}}_{ij}) = |Y_{ij} - \boldsymbol{\hat{\mu}}_{ij}|$ for a given model $\hat{\mu}$ on the calibration set. We present results for graph size $n=200$, significance level $\alpha =0.1$, sparsity level $\rho_n$ = 1, and number of iterations $500$ in Table \ref{tab-ex1-compare-splitting-method}. 

\begin{table}[ht]
\centering
\begin{tabular}{ |c|l|c|c| }
\hline
 Method & Model & Coverage & Width \\
 \hline
 \multirow{2}{*}{Node Splitting}
     & Linear  & 0.894 & 9.30 \\
     & RF & 0.882 & 7.68 \\
 \hline
 \multirow{2}{*}{Edge Splitting}
     & Linear & 0.900 & 9.20 \\
     & RF & 0.888 & 7.15 \\
 \hline
 \multirow{2}{*}{Selected Column}
     & Linear & 0.894 & 9.07 \\
     & RF & 0.914 & 6.87 \\
 \hline
\end{tabular}
\caption{Comparison of coverage and width of node, edge, and selected column splitting methods for both linear model and random forest model.}
\label{tab-ex1-compare-splitting-method}
\end{table}

The coverage of the conformal prediction intervals is close to the nominal level of $0.9$ for all splitting methods, with all values falling within the binomial confidence interval $(0.874, 0.926)$. The random forest model consistently produces narrower intervals than the linear model as expected, since it better captures the nonlinear relationship between the node covariates and the response. Nevertheless, both models achieve coverage close to the nominal level, which is consistent with the finite-sample validity guarantee provided by Theorem \ref{theorem-split-conformal}. Comparing widths across different splitting methods, edge splitting tends to produce narrower intervals than node splitting, a pattern we observe across other response functions as well. The width of the selected column method depends on the choice of column; on average, it produces the shortest intervals in this setting.

\subsubsection{Comparing the Row-Column Approach}
\label{sec-ex2}
We now evaluate the row-column approach with node splitting proposed in Section \ref{subsec:row-column-conformal} and compare it with the standard node splitting method. We consider three variants of the data-generating process. The first, symmetric and homoscedastic case, is identical to the setting for Section \ref{simulations-sampled-elements},
while in the remaining two cases, $Y_{ij}$ and $M_{ij}$ are no longer constrained to be symmetric. The second case is for directed networks, which introduces asymmetry in all pairs $i \neq j$ of $Y$, with constants  $c_1 = 1.1$, $c_2=0.22$ so that relative contributions of the latent signal and noise remain comparable after standardization: 
\begin{align*}
Y_{ij} &= \text{sigmoid}\!\left(X_{i1}X_{j1}^2+X_{i2}X_{j2}^2 - E(X_{i1}X_{j1}^2+X_{i2}X_{j2}^2)\right) \\
&\quad +\, c_1(\xi_{i1}\xi_{j1}^2+\xi_{i2}\xi_{j2}^2)+c_2\epsilon_{ij}.
\end{align*}
The third case uses the same model structure as the first, but introduces heteroscedastic noise amplified by the node covariates, 
\begin{align*}
Y_{ij} &= \text{sigmoid}\!\left(X_{i1}X_{j1}+X_{i2}X_{j2} - E(X_{i1}X_{j1}+X_{i2}X_{j2})\right) \\
&\quad +\, c_1(\xi_{i1}\xi_{j1}+\xi_{i2}\xi_{j2})+c_2(X_{i1}+X_{j1})\epsilon_{ij}.
\end{align*}
Here, we set $c_1=0.9$ and $c_2=0.1$ to again keep the three cases comparable. For the symmetric case, we still consider only the upper-triangular entries with $i<j$, while for the other two all off-diagonal entries are included.

For standard node splitting  and a prespecified test point $(k,l)$, the conformal prediction interval is constructed from the absolute residuals on the calibration set, as described in Section \ref{sec:split-conformal}. For the row-column approach, the row and column ranks $R_{ij}$ and $C_{ij}$ are computed for each entry in the calibration set, using entries sharing the same row or column as reference. The resulting rank pairs are then mapped to a reference distribution supported on the unit circle translated to the first quadrant via optimal transport as described in Section \ref{subsec:row-column-conformal}.

Since the responses are continuous, it is reasonable to assume that the nonconformity scores are distinct almost surely, and the rank pairs can only change when $S_{kl}$ crosses one of the scores of entries sharing a row or column with the test point. To construct the prediction set, we scan through these candidates in descending order to identify the value of $S_{kl}$ at which $\check{\Pi}_{kl} > \alpha$ first occurs, recomputing the optimal transport map at each step. A detailed description of this procedure is provided in Appendix \ref{sec-ex2-appendix}. We present results for $n=200$, $\alpha =0.1$, $\rho_n = 0.8$, and number of iterations $500$ in Table \ref{tab-ex2-compare-row-column-node}.

\begin{table}[ht]
\centering
\begin{tabular}{ |c|c|c|c|c| }
 \hline
 Case & Method & Model & Coverage & Width \\
 \hline
 \multirow{4}{*}{1} 
     & \multirow{2}{*}{Standard} & Linear & 0.896 & 9.30 \\
     &                      & RF & 0.880 & 7.77 \\
 \cline{2-5}
     & \multirow{2}{*}{Row-Column} & Linear & 0.894 & 9.38 \\
     &                              & RF & 0.876 & 7.80 \\
 \hline
 \multirow{4}{*}{2} 
     & \multirow{2}{*}{Standard} & Linear & 0.896 & 9.75 \\
     &                      & RF & 0.896 & 7.57 \\
 \cline{2-5}
     & \multirow{2}{*}{Row-Column} & Linear & 0.892 & 9.68 \\
     &                              & RF & 0.906 & 7.44 \\
 \hline
\multirow{4}{*}{3} 
     & \multirow{2}{*}{Standard} & Linear & 0.904 & 9.05 \\
     &                      & RF & 0.902 & 7.43 \\
 \cline{2-5}
     & \multirow{2}{*}{Row-Column} & Linear & 0.898 & 8.88 \\
     &                              & RF & 0.892 & 6.94 \\
 \hline
\end{tabular}
\caption{Comparison of standard conformal prediction sets with row-column approach under node splitting.  The underlying model ranges from (1) basic (symmetric and homoscedastic), (2) asymmetric and (3) heteroscedastic.}
\label{tab-ex2-compare-row-column-node}
\end{table}

The coverage of both methods is close to the nominal level of $0.9$ across all three cases. The row-column approach produces narrower intervals than the standard node splitting method in the asymmetric and heteroscedastic cases, and the two are similar in the basic case. This is consistent with our intuition that entries sharing a row or column with the test point are more informative for calibration, as suggested by the Aldous-Hoover theorem. A similar pattern is observed when the row-column approach is combined with edge splitting; see Appendix \ref{sec-ex2-appendix} for further details.

\subsection{Asymptotic Validity for Missing Elements}
\label{sec-ex3}
In this section, we consider the setting where the test point is a missing entry and evaluate the graphon-weighted conformal prediction procedure proposed in Section \ref{sec-weighted-conformal}. Our empirical focus is primarily on evaluating the finite-sample performance under network dependence. We focus on marginal properties;  under the conditional validity framework, one must simultaneously optimize over parameters for the conditional regression estimators $\hat{Y}$ and the IPW-weighted conditional CDFs, and we leave optimal data-driven tuning parameter selection for this regime to future work.

We consider two data-generating processes that differ in the dimension of the latent covariates used for the response and the masking matrix. For the first model, the response is given by
\begin{align*}
    Y_{ij} 
    &= \text{sigmoid}\!\left(X_{i1}X_{j1}+X_{i2}X_{j2} - E(X_{i1}X_{j1}+X_{i2}X_{j2})\right) \\
    &\quad +\, c_1\xi_{i}\xi_{j}+c_2\epsilon_{ij},
\end{align*}
where $c_1$ and $c_2$ are as defined in Section \ref{sec-ex1}. The masking matrix is given by
\begin{align*}
    M_{ij}\sim \text{Bernoulli}(\rho_n \omega (\xi_i, \xi_j) ).
\end{align*} 
For the first setting, we use a graphon model with one-dimensional latent positions 
$\xi_i$,  $\omega (\xi_i, \xi_j)  = 1-|\xi_i - \xi_j|$, and for the second, we use the same two-dimensional inner product model as in Section \ref{sec-ex1}, $\omega (\xi_i, \xi_j)  = \xi_{i1}\xi_{j1} + \xi_{i2} \xi_{j2}$. For both models, the response $Y_{ij}$ and the masking matrix $M_{ij}$ share a complex dependency structure through the latent variables, and the conditional distributions of  $Y_{ij} \ | \ M_{ij}=0$ and $Y_{ij} \ | \ M_{ij}=1$ differ. 

As discussed in Section \ref{sec-weighted-conformal}, constructing the weighted conformal prediction interval for the missing elements requires estimation of the 
probability matrix $P$. We consider two estimation methods to solve this task. The 
first is the Universal Singular Value Thresholding (USVT) method  \cite{d6cb28bb-5a7b-386d-8ced-b68d43954df3}, which estimates $P$ 
by setting the singular values of the masking matrix below a pre-defined universal threshold to zero. While the original method assumes independent 
entries, the authors note that it extends to dependent entries with 
minor adjustments to the threshold. In our experiments, we set the threshold to be $3.03\sqrt{n}$, which is universal across all settings. The second graphon estimation method we include is Neighborhood Smoothing (NS)~\cite{repec:oup:biomet:v:104:y:2017:i:4:p:771-783.}, which estimates each entry of $P$ by averaging neighboring entries of the masking matrix. 
\begin{figure}[ht]
\centering
    \includegraphics[width=0.8\linewidth]{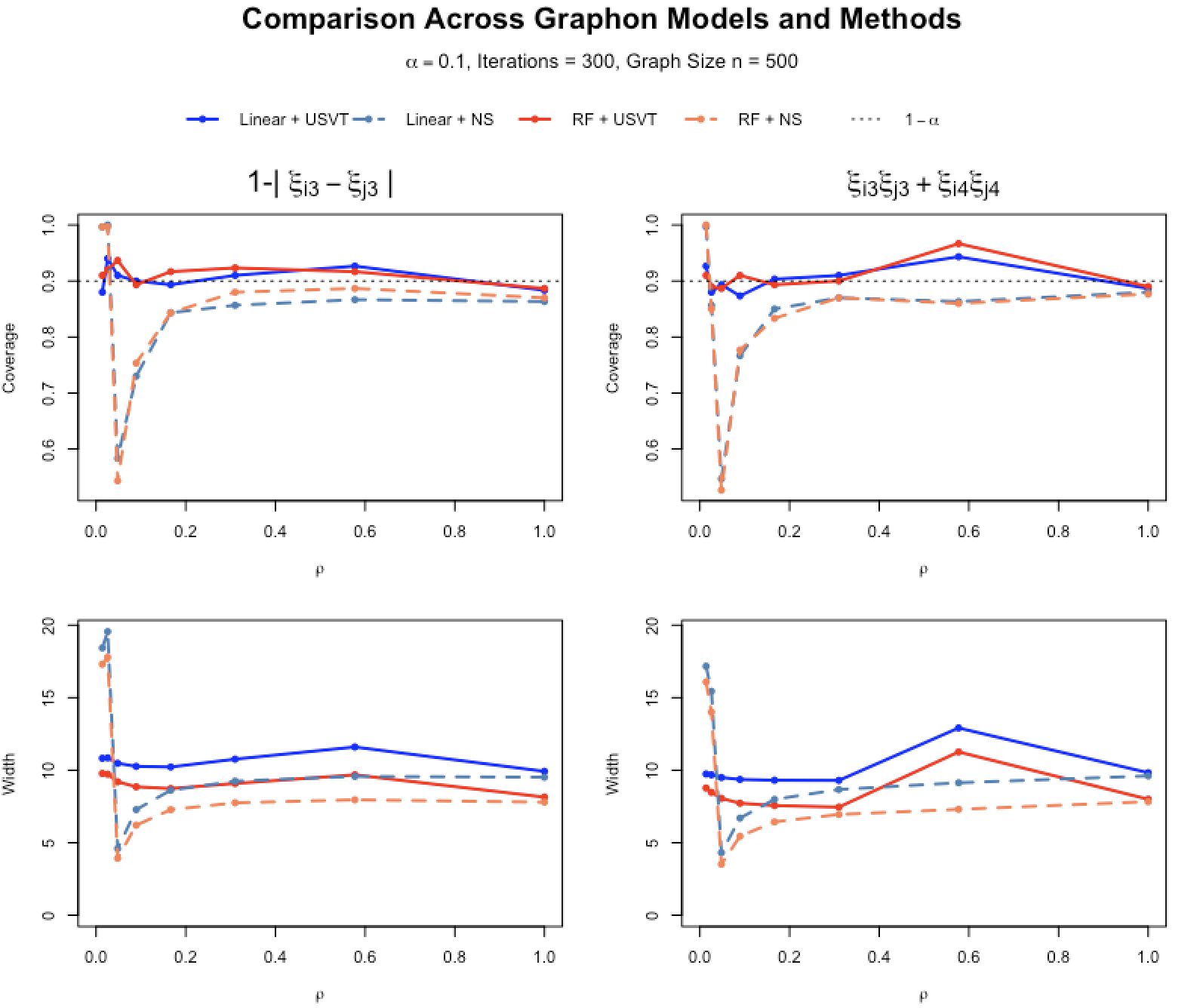}
    \caption{Top row: coverage of the weighted conformal prediction 
    interval for a missing entry across two graphon models as the 
    sparsity parameter varies. Bottom row: corresponding width of 
    the prediction intervals. Results are shown for linear and RF fits and  USVT and 
    NS graphon estimation methods.}
    \label{fig-ex3-missing}
\end{figure}

Results for $n=500$, $\alpha = 0.1$, and $300$ iterations are presented in Figure \ref{fig-ex3-missing}, where coverage and width are plotted as functions of the sparsity level $\rho_n$, ranging from $2n^{-0.8}$ to $\min\{2n^{-0.1}, 1\}$. For $\rho_n \geq 2n^{-0.3}$, coverage is close to the nominal level across all methods; however, for sparser graphs, the USVT graphon estimation performs better, possibly because it provides more accurate estimates of the probability matrices for sparser graphs. 
A precise theoretical 
understanding of how the properties of graphon estimators 
interact with the validity of weighted conformal procedures is 
an interesting direction for future work. 

On the other hand, for moderate to dense graphs ($\rho_n \geq n^{-0.7}$), the NS method consistently produces narrower prediction intervals than the USVT method.
Finally, the random forest model produces narrower conformal prediction intervals than the linear model, as in the previous simulation, and results are similar for the two graphon models across all settings. 

\subsection{Example: Predicting Common Citations}
\label{sec-ex4}
We illustrate our split conformal prediction procedure on the Cora dataset \citep{mccallum2000cora}, which contains data on $2708$ machine learning publications and $10556$ citations between papers. Each paper is classified into one of $7$ categories and is associated with a $1433$-dimensional binary feature vector, indicating the presence or absence of specific words. We consider the task of predicting the joint citation 
count between pairs of papers, which is of 
practical interest in applications such as paper 
recommendation systems, where papers with many common citations are likely to be closely related. 

In order to evaluate the performance of our method on ground truth, we generate a sampling mechanism in which each pair is selected with a probability that depends on node-level features through a symmetric function, creating a missingness mechanism that depends on the data. Selection based on node features is a realistic scenario: for instance, in proteomics recorded interactions are strongly biased toward well-studied proteins \citep{blumenthal2024emergence}, and in citation networks, whether a pair of papers is observed may depend on properties such as venue and full-text availability.  Our goal is to construct prediction sets for pairs sampled according to this mechanism ($M_{ij}=1$) but withheld from the calibration set, allowing us to evaluate empirical coverage against ground truth.

Let $A$ denote the (weighted) adjacency matrix representing citations between papers. We consider the response matrix 
$Y = A A^T$, that is, $Y_{ij}$ is the number of papers cited by both paper $i$ and paper $j$.   To mimic the ``study bias'' (the common practice of prioritizing experiments on pairs believed likely to interact), we generate the missingness mask by preferentially observing pairs that a graph neural network predicts are likely to interact. This yields a complex, feature-dependent missingness mechanism that is unknown to the calibration procedure.

To this end, we estimate a probability matrix $\hat{P}$ from $A$ and the node features using a graph attention network, then threshold at $0.5$ to obtain
\begin{align*}
    M_{ij} = \mathbbm{1}(\hat{p}_{ij} > 0.5).
\end{align*}
The resulting masking matrix has a density of $0.12$. Details of the graph attention network architecture \cite{brody2022how} and training procedure are provided in Appendix \ref{sec-ex4-appendix}. We consider three splitting methods: node splitting, edge splitting, and selected column splitting. For node splitting, nodes are partitioned with a $6/3/1$ ratio, and the observed edges within each subgraph serve as the training, calibration, and test sets. For edge splitting, the observed edges are randomly partitioned with the same $6/3/1$ ratio. For the selected column method, the column with the highest degree is selected, and its entries are split into calibration and test sets with a $9/1$ ratio, with all remaining entries used for training.

We fit a random forest model on the training set to predict $Y_{ij}$ for observed pairs ($M_{ij}=1$) using node-level features constructed from $X_{i}$ and $X_{j}$, where $X$ is the word list of each node, and the class labels of each node. These features include similarity measures (Jaccard similarity, Cosine similarity, Hadamard product sum) of $X_{i}$ and $X_{j}$, word counts (number of words in $X_{i}$ and $X_{j}$, number of common words, and number of 
distinct words between $X_{i}$ and $X_{j}$), and class information. Conformal prediction intervals are constructed using the absolute residual nonconformity score, and coverage is computed by averaging over all test points. We set $\alpha = 0.05$ and present results over $100$ iterations in Table \ref{tab-ex3-real-world}. The coverage on the test set is $0.950$ for node splitting and edge splitting, and $0.949$ for selected column method, which is consistent with the nominal level of $0.95$. Among the three splitting methods, the selected column method produces narrower intervals, which may be attributed to both the larger training set for the fitted random forest model, and the similarity between calibration and test entries that share a column.  While the masking matrix in this example was constructed from a fitted model rather than a known generative process, it provides a complex setting, and yet all of our methods perform well.   

\begin{table}[ht]
\centering
\begin{tabular}{ |C{3cm}|C{2cm}|C{2cm}|C{2cm}| }
 \hline
Splitting & Model & Coverage & Width \\
 \hline
 \multirow{1}{*}{Node}
     & RF & 0.950 & 1.62 \\
 \hline
  \multirow{1}{*}{Edge}
     & RF & 0.950 & 1.59 \\
 \hline
  \multirow{1}{*}{Selected Column}
     & RF & 0.949 & 0.18 \\
 \hline
\end{tabular}
\caption{Results for the Cora dataset. Reported coverage and width averaged over test nodes.}
\label{tab-ex3-real-world}
\end{table}

\section{Discussion}

In this paper, we have developed a framework for conformal prediction 
in dyadic regression problems under complex missingness mechanisms. 
Our contributions include both a general theoretical 
framework for conformal prediction beyond exchangeability, and 
specific methodology for array prediction problems under joint 
exchangeability.

The key theoretical contribution is  superuniformity of 
conformal prediction under distributional invariance conditions 
weaker than exchangeability; Theorem \ref{theorem-super-unform-be} handles 
the case where the sample itself is a random subset of the index 
set, a setting not covered by existing theory. The proof 
introduces a novel bijection argument that works directly with 
the joint distribution of nonconformity scores and inclusion indicators, circumventing standard rank-identity arguments which fail when observations can leave the sample under group transformations. We believe these results are of substantial independent interest beyond the array prediction setting considered here.

Several directions for future work are worth highlighting. First, establishing simultaneous coverage guarantees for groups of missing entries under the jointly exchangeable framework would extend the individual coverage results of Section \ref{sec-conformal-for-sampled-elements}
 in a practically useful direction. Second, the beyond-exchangeability framework of Section \ref{sec-beyond-exchangeability} may find applications in other settings where group invariance holds but full exchangeability does not, including time series with cyclic stationarity, spatial data with rotational invariance, and tensor-valued observations.  Finally, extending the weighted conformal results of Section \ref{sec-weighted-conformal} to handle 
unbounded graphons and heavier-tailed missingness mechanisms  would broaden the applicability of the framework.

\subsection*{Acknowledgments}
R. Lunde gratefully acknowledges support from NSF grant DMS-2515923. E. Levina and J. Zhu's research was partially supported by NSF grants DMS-2210439 and DMS-2610168.  This project was started while the first author was a postdoctoral fellow at the University of Michigan. The authors acknowledge the use of Claude Opus 5 and Claude Fable 5 for reviewing the pre-submission manuscript and for suggesting an alternative proof strategy for Theorem 5.1. This suggestion resulted in the removal of one condition without changing the conclusion relative to the version previously uploaded to arXiv. AI was used for proofreading in previous versions but not for mathematical proofs until the current version. The new proof was independently and fully verified by the authors.


 \bibliographystyle{apalike}
 \bibliography{mybib.bib}

@book{cameron1999permutation,
  author    = {Cameron, Peter J.},
  title     = {Permutation Groups},
  series    = {London Mathematical Society Student Texts},
  number    = {45},
  publisher = {Cambridge University Press},
  year      = {1999}
}

@article{blumenthal2024emergence,
  author  = {Blumenthal, David B. and Lucchetta, Marta and Kleist, Linda and Fekete, S{\'a}ndor P. and List, Markus and Schaefer, Martin H.},
  title   = {Emergence of power law distributions in protein--protein interaction networks through study bias},
  journal = {eLife},
  volume  = {13},
  pages   = {e99951},
  year    = {2024},
  doi     = {10.7554/eLife.99951}
}

@article{gao2015rate,
  author    = {Gao, Chao and Lu, Yu and Zhou, Harrison H.},
  title     = {Rate-optimal graphon estimation},
  journal   = {The Annals of Statistics},
  volume    = {43},
  number    = {6},
  pages     = {2624--2652},
  publisher = {Institute of Mathematical Statistics},
  year      = {2015},
  doi       = {10.1214/15-AOS1354}
}

@book{niederreiter1992,
  author    = {Niederreiter, Harald},
  title     = {Random Number Generation and Quasi-{M}onte {C}arlo Methods},
  series    = {{CBMS}-{NSF} Regional Conference Series in Applied Mathematics},
  volume    = {63},
  publisher = {Society for Industrial and Applied Mathematics ({SIAM})},
  address   = {Philadelphia, PA},
  year      = {1992}
}

@misc{koning2026measuringevidenceexchangeabilitygroup,
  author        = {Nick W. Koning},
  title         = {Measuring Evidence against Exchangeability and Group Invariance with E-values},
  year          = {2026},
  eprint        = {2310.01153},
  archiveprefix = {arXiv},
  primaryclass  = {math.ST},
  howpublished  = {Preprint, arXiv:2310.01153},
  url           = {https://arxiv.org/abs/2310.01153}
}

@misc{huang2026dataaugmentedbootstrapunifying,
  author        = {Kevin Han Huang},
  title         = {Data augmented bootstrap: Unifying confidence interval construction by approximate invariance},
  year          = {2026},
  eprint        = {2606.09049},
  archiveprefix = {arXiv},
  primaryclass  = {stat.ME},
  howpublished  = {Preprint, arXiv:2606.09049},
  url           = {https://arxiv.org/abs/2606.09049}
}

@misc{conditional_symmpi_2026,
  author        = {Yichen Shen and Mengxin Yu},
  title         = {Conditional Predictive Inference for General Structured Data with Group Symmetries},
  year          = {2026},
  eprint        = {2605.17934},
  archiveprefix = {arXiv},
  primaryclass  = {stat.ME},
  howpublished  = {Preprint, arXiv:2605.17934},
  url           = {https://arxiv.org/abs/2605.17934}
}

@article{unifying-conformal-2025,
  author    = {Barber, Rina Foygel and Tibshirani, Ryan J.},
  title     = {Unifying different theories of conformal prediction},
  journal   = {Electronic Journal of Statistics},
  volume    = {20},
  number    = {1},
  pages     = {1428--1474},
  publisher = {The Institute of Mathematical Statistics and the Bernoulli Society},
  year      = {2026},
  doi       = {10.1214/26-EJS2513}
}

@article{mccallum2000cora,
  author  = {McCallum, Andrew Kachites and Nigam, Kamal and Rennie, Jason and Seymore, Kristie},
  title   = {Automating the construction of internet portals with machine learning},
  journal = {Information Retrieval},
  volume  = {3},
  number  = {2},
  pages   = {127--163},
  year    = {2000}
}

@misc{paul2026probability,
  author = {Manit Paul and Arun Kumar Kuchibhotla},
  title  = {On a Probability Inequality for Order Statistics with Applications to Bootstrap, Conformal Prediction, and more},
  year   = {2026}
}

@article{Liang28042026,
  author    = {Ziyi Liang and Tianmin Xie and Xin Tong and Matteo Sesia},
  title     = {Structured Conformal Inference for Matrix Completion with Applications to Group Recommender Systems},
  journal   = {Journal of the American Statistical Association},
  volume    = {0},
  number    = {ja},
  pages     = {1--25},
  publisher = {Taylor \& Francis},
  year      = {2026},
  doi       = {10.1080/01621459.2026.2658287}
}

@misc{du2025conformallinkpredictionfalse,
  author        = {Wenqin Du and Wanteng Ma and Dong Xia and Yuan Zhang and Wen Zhou},
  title         = {Conformal Link Prediction with False Discovery Rate Control},
  year          = {2025},
  eprint        = {2507.07025},
  archiveprefix = {arXiv},
  primaryclass  = {stat.ME},
  howpublished  = {Preprint, arXiv:2507.07025},
  url           = {https://arxiv.org/abs/2507.07025}
}

@misc{klein2025multivariateconformalpredictionusing,
  author        = {Michal Klein and Louis Bethune and Eugene Ndiaye and Marco Cuturi},
  title         = {Multivariate Conformal Prediction using Optimal Transport},
  year          = {2025},
  eprint        = {2502.03609},
  archiveprefix = {arXiv},
  primaryclass  = {stat.ML},
  howpublished  = {Preprint, arXiv:2502.03609},
  url           = {https://arxiv.org/abs/2502.03609}
}

@inproceedings{10.5555/2999792.2999868,
  author    = {Cuturi, Marco},
  title     = {Sinkhorn distances: lightspeed computation of optimal transport},
  booktitle = {Advances in Neural Information Processing Systems},
  volume    = {26},
  pages     = {2292--2300},
  publisher = {Curran Associates, Inc.},
  year      = {2013}
}

@article{10.1214/20-AOS1996,
  author    = {Marc Hallin and Eustasio del Barrio and Juan Cuesta-Albertos and Carlos Matr{\'a}n},
  title     = {Distribution and quantile functions, ranks and signs in dimension d: A measure transportation approach},
  journal   = {The Annals of Statistics},
  volume    = {49},
  number    = {2},
  pages     = {1139--1165},
  publisher = {Institute of Mathematical Statistics},
  year      = {2021},
  doi       = {10.1214/20-AOS1996}
}

@misc{ritzwoller2025randomizationinferencetheoryapplications,
  author        = {David M. Ritzwoller and Joseph P. Romano and Azeem M. Shaikh},
  title         = {Randomization Inference: Theory and Applications},
  year          = {2025},
  eprint        = {2406.09521},
  archiveprefix = {arXiv},
  primaryclass  = {math.ST},
  howpublished  = {Preprint, arXiv:2406.09521},
  url           = {https://arxiv.org/abs/2406.09521}
}

@article{GAUCHER2021299,
  author  = {Solenne Gaucher and Olga Klopp},
  title   = {Maximum likelihood estimation of sparse networks with missing observations},
  journal = {Journal of Statistical Planning and Inference},
  volume  = {215},
  pages   = {299--329},
  year    = {2021},
  doi     = {10.1016/j.jspi.2021.04.003}
}

@article{9762725,
  author  = {Bhattacharya, Sohom and Chatterjee, Sourav},
  title   = {Matrix Completion With Data-Dependent Missingness Probabilities},
  journal = {IEEE Transactions on Information Theory},
  volume  = {68},
  number  = {10},
  pages   = {6762--6773},
  year    = {2022},
  doi     = {10.1109/TIT.2022.3170244}
}

@article{10.1145/3012704,
  author    = {Mart\'{\i}nez, V\'{\i}ctor and Berzal, Fernando and Cubero, Juan-Carlos},
  title     = {A Survey of Link Prediction in Complex Networks},
  journal   = {ACM Comput. Surv.},
  volume    = {49},
  number    = {4},
  publisher = {Association for Computing Machinery},
  address   = {New York, NY, USA},
  year      = {2016},
  doi       = {10.1145/3012704}
}

@article{GRAHAM2024105336,
  author  = {Bryan S. Graham and Fengshi Niu and James L. Powell},
  title   = {Kernel density estimation for undirected dyadic data},
  journal = {Journal of Econometrics},
  volume  = {240},
  number  = {2},
  pages   = {105336},
  year    = {2024},
  doi     = {10.1016/j.jeconom.2022.06.011}
}

@article{MarrsFosdick2023,
  author  = {Marrs, Frank W. and Fosdick, Bailey K.},
  title   = {Regression of binary network data with exchangeable latent errors},
  journal = {Network Science},
  volume  = {11},
  number  = {3},
  pages   = {502--535},
  year    = {2023},
  doi     = {10.1017/nws.2023.12}
}

@article{https://doi.org/10.1002/tie.5060050113,
  author  = {Tinbergen, Jan},
  title   = {Shaping the world economy},
  journal = {The International Executive},
  volume  = {5},
  number  = {1},
  pages   = {27--30},
  year    = {1963},
  doi     = {10.1002/tie.5060050113}
}

@article{apicella-hunter-gatherer,
  author  = {Apicella, Coren L. and Marlowe, Frank W. and Fowler, James H. and Christakis, Nicholas A.},
  title   = {Social networks and cooperation in hunter-gatherers},
  journal = {Nature},
  volume  = {481},
  number  = {7382},
  pages   = {497--501},
  year    = {2012},
  doi     = {10.1038/nature10736}
}

@article{Oneal_Russett_1999,
  author  = {Oneal, John R. and Russett, Bruce},
  title   = {{The Kantian} Peace: The {Pacific} Benefits of Democracy, Interdependence, and International Organizations, 1885--1992},
  journal = {World Politics},
  volume  = {52},
  number  = {1},
  pages   = {1--37},
  year    = {1999},
  doi     = {10.1017/S0043887100020013}
}

@inproceedings{pmlr-v202-zaffran23a,
  author    = {Zaffran, Margaux and Dieuleveut, Aymeric and Josse, Julie and Romano, Yaniv},
  title     = {Conformal prediction with missing values},
  booktitle = {Proceedings of the 40th International Conference on Machine Learning},
  series    = {Proceedings of Machine Learning Research},
  volume    = {202},
  pages     = {40578--40604},
  publisher = {PMLR},
  year      = {2023}
}

@inproceedings{pmlr-v204-luo23a,
  author    = {Luo, Rui and Nettasinghe, Buddhika and Krishnamurthy, Vikram},
  title     = {Anomalous edge detection in edge exchangeable social network models},
  booktitle = {Proceedings of the Twelfth Symposium on Conformal and Probabilistic Prediction with Applications},
  series    = {Proceedings of Machine Learning Research},
  volume    = {204},
  pages     = {287--310},
  publisher = {PMLR},
  year      = {2023}
}

@article{marandon-2024,
  author  = {Marandon, Ariane},
  title   = {Conformal link prediction for false discovery rate control},
  journal = {TEST},
  volume  = {33},
  number  = {4},
  pages   = {1062--1083},
  year    = {2024},
  doi     = {10.1007/s11749-024-00934-w}
}

@misc{shao2023distributionfreematrixpredictionarbitrary,
  author        = {Meijia Shao and Yuan Zhang},
  title         = {Distribution-Free Matrix Prediction Under Arbitrary Missing Pattern},
  year          = {2023},
  eprint        = {2305.11640},
  archiveprefix = {arXiv},
  primaryclass  = {stat.ME},
  howpublished  = {Preprint, arXiv:2305.11640},
  url           = {https://arxiv.org/abs/2305.11640}
}

@article{10.1214/23-AOS2276,
  author    = {Rina Foygel Barber and Emmanuel J. Cand{\`e}s and Aaditya Ramdas and Ryan J. Tibshirani},
  title     = {Conformal prediction beyond exchangeability},
  journal   = {The Annals of Statistics},
  volume    = {51},
  number    = {2},
  pages     = {816--845},
  publisher = {Institute of Mathematical Statistics},
  year      = {2023},
  doi       = {10.1214/23-AOS2276}
}

@article{JMLR:v24:22-1176,
  author  = {Ying Jin and Emmanuel J. Cand{\`e}s},
  title   = {Selection by Prediction with Conformal p-values},
  journal = {Journal of Machine Learning Research},
  volume  = {24},
  number  = {244},
  pages   = {1--41},
  year    = {2023}
}

@article{dobriban2025symmpi,
  author  = {Dobriban, Edgar and Yu, Mengxin},
  title   = {{SymmPI}: Predictive inference for data with group symmetries},
  journal = {Journal of the Royal Statistical Society Series B: Statistical Methodology},
  volume  = {87},
  number  = {5},
  pages   = {1353--1387},
  year    = {2025},
  doi     = {10.1093/jrsssb/qkaf022}
}

@article{10.1093/biomet/asaa079,
  author  = {Lei, Lihua and Bickel, Peter J.},
  title   = {An assumption-free exact test for fixed-design linear models with exchangeable errors},
  journal = {Biometrika},
  volume  = {108},
  number  = {2},
  pages   = {397--412},
  year    = {2021},
  doi     = {10.1093/biomet/asaa079}
}

@article{doi:10.1080/10618600.2017.1286243,
  author    = {Yunpeng Zhao and Yun-Jhong Wu and Elizaveta Levina and Ji Zhu},
  title     = {Link Prediction for Partially Observed Networks},
  journal   = {Journal of Computational and Graphical Statistics},
  volume    = {26},
  number    = {3},
  pages     = {725--733},
  publisher = {Taylor \& Francis},
  year      = {2017}
}

@article{10.1214/16-AOS1450,
  author    = {Victor Chernozhukov and Alfred Galichon and Marc Hallin and Marc Henry},
  title     = {Monge--Kantorovich depth, quantiles, ranks and signs},
  journal   = {The Annals of Statistics},
  volume    = {45},
  number    = {1},
  pages     = {223--256},
  publisher = {Institute of Mathematical Statistics},
  year      = {2017},
  doi       = {10.1214/16-AOS1450}
}

@article{10.1214/21-AOS2136,
  author    = {Promit Ghosal and Bodhisattva Sen},
  title     = {Multivariate ranks and quantiles using optimal transport: Consistency, rates and nonparametric testing},
  journal   = {The Annals of Statistics},
  volume    = {50},
  number    = {2},
  pages     = {1012--1037},
  publisher = {Institute of Mathematical Statistics},
  year      = {2022},
  doi       = {10.1214/21-AOS2136}
}

@inproceedings{gui2023conformalized,
  author    = {Gui, Yu and Barber, Rina Foygel and Ma, Cong},
  title     = {Conformalized matrix completion},
  booktitle = {Advances in Neural Information Processing Systems},
  volume    = {36},
  pages     = {4820--4844},
  publisher = {Curran Associates, Inc.},
  year      = {2023}
}

@inproceedings{huang2023uncertainty,
  author    = {Huang, Kexin and Jin, Ying and Cand{\`e}s, Emmanuel J. and Leskovec, Jure},
  title     = {Uncertainty quantification over graph with conformalized graph neural networks},
  booktitle = {Advances in Neural Information Processing Systems},
  volume    = {36},
  pages     = {26699--26721},
  publisher = {Curran Associates, Inc.},
  year      = {2023}
}

@inproceedings{Zargarbashi-gnn,
  author    = {Zargarbashi, Soroush H. and Antonelli, Simone and Bojchevski, Aleksandar},
  title     = {Conformal prediction sets for graph neural networks},
  booktitle = {Proceedings of the 40th International Conference on Machine Learning},
  series    = {Proceedings of Machine Learning Research},
  volume    = {202},
  pages     = {12292--12318},
  publisher = {PMLR},
  year      = {2023}
}

@article{lunde2023conformal,
  author    = {Robert Lunde and Elizaveta Levina and Ji Zhu},
  title     = {Conformal Prediction for Network-Assisted Regression},
  journal   = {Journal of the American Statistical Association},
  volume    = {120},
  number    = {551},
  pages     = {1633--1644},
  publisher = {Taylor \& Francis},
  year      = {2025},
  doi       = {10.1080/01621459.2025.2506198}
}

@article{doi:10.1080/01621459.2017.1341413,
  author    = {Harry Crane and Walter Dempsey},
  title     = {Edge Exchangeable Models for Interaction Networks},
  journal   = {Journal of the American Statistical Association},
  volume    = {113},
  number    = {523},
  pages     = {1311--1326},
  publisher = {Taylor \& Francis},
  year      = {2018}
}

@article{Chernozhukove2107794118,
  author       = {Chernozhukov, Victor and W{\"u}thrich, Kaspar and Zhu, Yinchu},
  title        = {Distributional conformal prediction},
  journal      = {Proceedings of the National Academy of Sciences},
  volume       = {118},
  number       = {48},
  publisher    = {National Academy of Sciences},
  year         = {2021},
  doi          = {10.1073/pnas.2107794118},
  elocation-id = {e2107794118}
}

@incollection{GRAHAM202023,
  author    = {Bryan S. Graham},
  title     = {Chapter 2 - {Dyadic Regression}},
  booktitle = {The Econometric Analysis of Network Data},
  pages     = {23--40},
  publisher = {Academic Press},
  year      = {2020},
  doi       = {10.1016/B978-0-12-811771-2.00008-0}
}

@book{vovk-algorithmic-world,
  author    = {Vladimir Vovk and Alexander Gammerman and Glenn Shafer},
  title     = {Algorithmic Learning in a Random World},
  publisher = {Springer},
  address   = {New York},
  year      = {2005}
}

@misc{angelopoulos2021gentle,
  author        = {Anastasios N. Angelopoulos and Stephen Bates},
  title         = {A Gentle Introduction to Conformal Prediction and Distribution-Free Uncertainty Quantification},
  year          = {2021},
  eprint        = {2107.07511},
  archiveprefix = {arXiv},
  primaryclass  = {cs.LG},
  howpublished  = {Preprint, arXiv:2107.07511},
  url           = {https://arxiv.org/abs/2107.07511}
}

@inproceedings{NEURIPS2020_244edd7e,
  author    = {Romano, Yaniv and Sesia, Matteo and Cand{\`e}s, Emmanuel J.},
  title     = {Classification with valid and adaptive coverage},
  booktitle = {Advances in Neural Information Processing Systems},
  volume    = {33},
  pages     = {3581--3591},
  publisher = {Curran Associates, Inc.},
  year      = {2020}
}

@article{commenges-transformations,
  author  = {Daniel Commenges},
  title   = {Transformations which preserve exchangeability and application to permutation tests},
  journal = {Journal of Nonparametric Statistics},
  volume  = {15},
  number  = {2},
  pages   = {171--185},
  year    = {2003}
}

@article{proteinlinkprediction,
  author  = {Yanjun Qi and Ziv Bar-Joseph and Judith Klein-Seetharaman},
  title   = {Evaluation of different biological data and computational classification methods for use in protein interaction prediction},
  journal = {Proteins},
  volume  = {63},
  number  = {3},
  pages   = {490--500},
  year    = {2006}
}

@article{drug-link-prediction,
  author  = {Abbas, Khushnood and Abbasi, Alireza and Dong, Shi and Niu, Ling and Yu, Laihang and Chen, Bolun and Cai, Shi-Min and Hasan, Qambar},
  title   = {Application of network link prediction in drug discovery},
  journal = {BMC Bioinformatics},
  volume  = {22},
  number  = {1},
  pages   = {187},
  year    = {2021},
  doi     = {10.1186/s12859-021-04082-y}
}

@article{10.1371/journal.pone.0154244,
  author    = {Berlusconi, Giulia and Calderoni, Francesco and Parolini, Nicola and Verani, Marco and Piccardi, Carlo},
  title     = {Link Prediction in Criminal Networks: A Tool for Criminal Intelligence Analysis},
  journal   = {PLOS ONE},
  volume    = {11},
  number    = {4},
  pages     = {1--21},
  publisher = {Public Library of Science},
  year      = {2016},
  doi       = {10.1371/journal.pone.0154244}
}

@inproceedings{NEURIPS2019_8fb21ee7,
  author    = {Tibshirani, Ryan J. and Foygel Barber, Rina and Cand{\`e}s, Emmanuel J. and Ramdas, Aaditya},
  title     = {Conformal prediction under covariate shift},
  booktitle = {Advances in Neural Information Processing Systems},
  volume    = {32},
  publisher = {Curran Associates, Inc.},
  year      = {2019}
}

@article{klopp-oracle-inequalities,
  author    = {Klopp, Olga and Tsybakov, Alexandre B. and Verzelen, Nicolas},
  title     = {Oracle inequalities for network models and sparse graphon estimation},
  journal   = {The Annals of Statistics},
  volume    = {45},
  number    = {1},
  pages     = {316--354},
  publisher = {Institute of Mathematical Statistics},
  year      = {2017}
}

@article{dean-verducci,
  author  = {Dean, Angela M. and Verducci, Joseph S.},
  title   = {Linear transformations that preserve majorization, {Schur} concavity, and exchangeability},
  journal = {Linear Algebra and its Applications},
  volume  = {127},
  pages   = {121--138},
  year    = {1990}
}

@misc{lunde2023validityconformalpredictionnetwork,
  author        = {Robert Lunde},
  title         = {On the Validity of Conformal Prediction for Network Data Under Non-Uniform Sampling},
  year          = {2023},
  eprint        = {2306.07252},
  archiveprefix = {arXiv},
  primaryclass  = {math.ST},
  howpublished  = {Preprint, arXiv:2306.07252},
  url           = {https://arxiv.org/abs/2306.07252}
}

@misc{kuchibhotla2021exchangeability,
  author        = {Arun Kumar Kuchibhotla},
  title         = {Exchangeability, Conformal Prediction, and Rank Tests},
  year          = {2021},
  eprint        = {2005.06095},
  archiveprefix = {arXiv},
  primaryclass  = {math.ST},
  howpublished  = {Preprint, arXiv:2005.06095},
  url           = {https://arxiv.org/abs/2005.06095}
}

@article{10.1214/20-AOS1981,
  author  = {Laurent Davezies and Xavier D'Haultf{\oe}uille and Yannick Guyonvarch},
  title   = {Empirical process results for exchangeable arrays},
  journal = {The Annals of Statistics},
  volume  = {49},
  number  = {2},
  pages   = {845--862},
  year    = {2021}
}

@article{hlavka2023multivariate,
  author  = {Hl{\'a}vka, Zden{\v e}k and Hlubinka, Daniel and Hudecov{\'a}, {\v S}{\'a}rka},
  title   = {Multivariate quantile-based permutation tests with application to functional data},
  journal = {Journal of Computational and Graphical Statistics},
  year    = {2025},
  doi     = {10.1080/10618600.2024.2444302},
  note    = {Published online; preprint at arXiv:2311.04017}
}

@article{10.1093/biomet/asae010,
  author  = {Bao, Yajie and Huo, Yuyang and Ren, Haojie and Zou, Changliang},
  title   = {Selective conformal inference with false coverage-statement rate control},
  journal = {Biometrika},
  volume  = {111},
  number  = {3},
  pages   = {727--742},
  year    = {2024}
}

@book{Fang1993NumbertheoreticMI,
  author    = {Fang, Kai-Tai and Wang, Yuan},
  title     = {Number-Theoretic Methods in Statistics},
  publisher = {Chapman \& Hall/CRC},
  address   = {London, UK},
  year      = {1994},
  isbn      = {978-0412465208}
}

@inproceedings{brody2022how,
  author    = {Shaked Brody and Uri Alon and Eran Yahav},
  title     = {How Attentive are Graph Attention Networks?},
  booktitle = {International Conference on Learning Representations},
  year      = {2022}
}

@article{d6cb28bb-5a7b-386d-8ced-b68d43954df3,
  author    = {Chatterjee, Sourav},
  title     = {Matrix estimation by universal singular value thresholding},
  journal   = {The Annals of Statistics},
  volume    = {43},
  number    = {1},
  pages     = {177--214},
  publisher = {Institute of Mathematical Statistics},
  year      = {2015}
}

@article{repec:oup:biomet:v:104:y:2017:i:4:p:771-783.,
  author  = {Yuan Zhang and Elizaveta Levina and Ji Zhu},
  title   = {Estimating network edge probabilities by neighbourhood smoothing},
  journal = {Biometrika},
  volume  = {104},
  number  = {4},
  pages   = {771--783},
  year    = {2017}
}

\appendix
\counterwithin{figure}{section}
\counterwithin{table}{section}
\section{Appendix}
\subsection{Proofs for Section \ref{sec-beyond-exchangeability}: Beyond Exchangeability
}
\label{sec:appendix-1}
Before we prove Proposition $ \ref{prop-super-uniform-transitive}$, we make an observation that simplifies arguments considerably.
\begin{proposition} 
\label{prop-lexicographic-order}
Let $(W_1, \ldots, W_n)$ be a vector of random variables and $\mathcal{G}$ be a collection of bijections $[n] \mapsto [n]$.  Suppose that $\gamma_1, \ldots, \gamma_n \sim \mathrm{Uniform}[0,1]$ are generated independently of all random variables and let $\widetilde{W}_i = (W_i,\gamma_i)$ for $i \in [n]$ and let $\widetilde{W} = (\widetilde{W}_1, \ldots, \widetilde{W}_n)$.  If
\begin{align*}
gW \stackrel{d}{=} W \ \forall \ g \in \mathcal{G}   
\end{align*}
then,
\begin{align*}
g \widetilde{W} \stackrel{d}{=} \widetilde{W}  \ \forall \ g \in \mathcal{G}.    
\end{align*}
Moreover, 
\begin{align*}
\frac{1}{n}\sum_{k=1}^n\mathbbm{1}(\widetilde{W}_k \succeq \widetilde{W}_i) \leq \frac{1}{n}\sum_{k=1}^n\mathbbm{1}(W_k \geq W_i), 
\end{align*}
where $\succeq$ corresponds to the lexicographic order.  
\end{proposition}
As a consequence of Proposition \ref{prop-lexicographic-order}, we have, for any $t \in \mathbb{R}$, 
\begin{align*}
P\left(  \frac{1}{n}\sum_{k=1}^n\mathbbm{1}(W_k \geq W_i)\leq t\right) \leq P\left( \frac{1}{n}\sum_{k=1}^n\mathbbm{1}(\widetilde{W}_k \succeq \widetilde{W}_i)  \leq t\right) 
\end{align*}
where $\widetilde{W}_1, \ldots,\widetilde{W}_n$ are distinct almost surely. In what follows, let $R_i = \frac{1}{n} \sum_{k=1}^n \mathbbm{1}(W_k \geq W_i)$ and $\widetilde{R}_i = \frac{1}{n} \sum_{k=1}^n \mathbbm{1}(\widetilde{W}_k \succeq \widetilde{W}_i)$. \\ 

\noindent \textit{Proof of Proposition \ref{prop-super-uniform-transitive}}.
 By the invariance assumption, we have that, for any $t \in \mathbb{R}$ and any $g_j \in \mathcal{G}_0$,
\begin{align*}
P(R_i \leq t) = P\left( \frac{1}{n}\sum_{k=1}^n \mathbbm{1}(W_k \geq W_i) \leq t \right)
&= P\left(\frac{1}{n} \sum_{k=1}^n \mathbbm{1}(W_{g_j(k)} \geq W_{g_j(i)}) \leq t  \right)
\\ &= P\left(\frac{1}{n} \sum_{k=1}^n \mathbbm{1}(W_{k} \geq W_{j}) \leq t \right)
\\ &= P(R_j \leq t).
\end{align*}
Therefore, $R_1, \ldots, R_n$ are identically distributed.  Moreover, for any $t \in [0,1]$, 
\begin{align*}
P(R_i \leq t) &= \frac{1}{n} \sum_{j=1}^n P(R_j \leq t)
\\ &\leq \frac{1}{n} E\left[\sum_{j=1}^n\mathbbm{1}(\widetilde{R}_j \leq t) \right]
\\ &\stackrel{(i)}{=}\frac{1}{n} \cdot n \cdot \frac{\lfloor nt \rfloor}{n}
\\ & \leq t,
\end{align*}
where in $(i)$ we used the fact that since $\tilde{R}_i$ are distinct almost surely, there are exactly $n \cdot \frac{\lfloor nt \rfloor}{n}$ elements that are less than or equal to $t$ almost surely. \qed \\ \\ 

We now give two proofs of Theorem \ref{theorem-super-unform-be}.  

\begin{proof}[First proof of Theorem \ref{theorem-super-unform-be}]
We will argue in three main steps.  First, we will prove a statement that requires two conditions, and then show that condition (b) is implied for a transitive group action. Then, we will argue that the condition in this proposition is equivalent to the one imposed in Theorem \ref{theorem-super-unform-be}.   

\paragraph{Step 1: Bijection argument}
We start with a bijection argument that requires a transitivity condition along with an identifiability condition. 

\begin{proposition}[Superuniformity for Sampled Elements Under Distributional Invariance]
\label{prop-super-unform-be}
Suppose that $\mathcal{G}^{\mathcal{N}}$ satisfies (\refeq{eq-invariance-L}), and let $i \in \mathcal{N}$. Suppose that there exists a collection of functions
\begin{align*}
\mathcal{G}_0^{\mathcal{N}} \subseteq \langle \mathcal{G}^{\mathcal{N}} \rangle   
\end{align*} 
such that:
\begin{enumerate}[label=(\alph*)]
\item
\label{eq:a}
$g_j(j) = i$
\item
\label{eq:b}
$g_j(i) \neq g_k(i)$ for $ j \neq k$.
\end{enumerate}
Then, for any $t \in [0,1],$ 
\begin{align*}
P\left( \frac{1}{|\mathcal{K}|} \sum_{k \in \mathcal{K}} \mathbbm{1}(W_k \geq W_i) \leq t \ \biggr\rvert \ i \in \mathcal{K}  \right) \leq t. 
\end{align*}
\end{proposition}
\begin{proof}

Under the stated conditions, we will first show that:
\begin{align}
\label{eq-rank-uniform}
\mathrm{rank}(\widetilde{W}_i, \{\widetilde{W}_k\}_{k \in \mathcal{K}}) \ \bigr\rvert \ i \in \mathcal{K}, \ |\mathcal{K}| = c  \sim \mathrm{Uniform}\{1, \ldots, c\}.
\end{align}

The case where $c \leq 1$ is trivial, so we focus on the case where $c >1$.  It suffices to show, for each $s \in \{1, \ldots, c-1\}$, 
\begin{align*}
& \ P(\{\mathrm{rank}(\widetilde{W}_i, \{\widetilde{W}_k\}_{k \in \mathcal{K}}) = s\} \cap  \{ \zeta_i =1 \} \cap \{|\mathcal{K}| = c\})  \\ 
= & \ P(\{\mathrm{rank}(\widetilde{W}_i, \{\widetilde{W}_k\}_{k \in \mathcal{K}}) = s+1\}  \cap  \{ \zeta_i =1 \} \cap \{|\mathcal{K}| = c \}). 
\end{align*}

To this end, for each $K \subseteq \mathcal{N}$ satisfying $|K| =c$ and $i \in K$, let $\mathcal{F}_s(K)$ be a collection of bijections $f:[c] \mapsto K$ ordering elements of $K$ such that $f(s) = i$. Let:
\begin{align*}
\mathcal{E}^{(s,c)} = \bigcup_{\substack{K: \ i \in K, \\ |K| = c}} \  \bigcup_{f \in \mathcal{F}_s(K)} \left\{\widetilde{W}_{f(1)} \prec \cdots \prec \widetilde{W}_{f(c)} \right\} \cap \{\mathcal{K} = K \}.
\end{align*}
Note that the sets in the union are disjoint.  Moreover, since the points are distinct almost surely, ties in the lexicographic order occur with probability $0$.

Define the collection:
\begin{align*}
\mathcal{F}_{s,c} = \bigcup_{\substack{K: \ i \in K, \\ |K| = c}} \mathcal{F}_s(K).
\end{align*}
We will construct a bijection: $\Psi: \mathcal{F}_{s,c}  \mapsto \mathcal{F}_{s+1,c}$ using elements of $\mathcal{G}^\mathcal{N}$, which will allow us to establish the desired property. 
To this end, suppose that $ \psi = \Psi f :[c] \mapsto \mathcal{N}$,  satisfies:
\begin{align*}
\psi(t) = \sum_{j \in \mathcal{N}}g_j(f(t)) \mathbbm{1}(f(s+1) = j).
\end{align*}
Observe that $\psi$ maps $f$ to values in $ K^{f} =  \{g_j(k) \ | \ j= f(s+1),  k \in K  \}$, where $K^f$ satisfies $|K^f| = c$ and $i = g_j(j) \in K^f$.  It is also clear that $\psi(s+1) = i$.
We now argue that it is onto.  Suppose $\psi \in \mathcal{F}_{s+1,c}$ and let $k = \psi(s)$. By assumption (b) and the pigeonhole principle (both domain and codomain of $j \mapsto g_j(i)$ are $\mathcal{N} \setminus \{i\}$), there exists $j \in \mathcal{N}$ such that $g_j(i) = k$.  Then, if we choose $f$ given by $f(t) = g_j^{-1}(\psi(t))$, then $f(s) = g_j^{-1}(k) = i$; therefore, $f \in \mathcal{F}_{s,c}$.  Moreover, $f(s+1) = g_j^{-1}(i) = j$; Therefore, $\Psi f = \psi$.  

Now, to show that it is one-to-one, suppose that $\psi_1 = \Psi f_1$, $\psi_2 = \Psi f_2$, and $\psi_1 = \psi_2$.  Then, it must be the case that $\psi_1(s) =  g_j(i)$, $\psi_2(s) =  g_k(i)$ for some $j,k \in \mathcal{N}$ and $\psi_1(s) = \psi_2(s)$.  Now, by assumption (b), $g_j(i) =g_k(i)$ implies $j=k$.  Furthermore, since $g_j$ is a bijection, it must also be the case that $f_1 = f_2$. Therefore, it follows that,
\begin{align*}
 & P\left(\mathcal{E}^{(s+1,c)} \right)
 \\ = \ &  P\left( \bigcup_{ h \in \mathcal{F}_{s+1,c}} \left[\left\{ \widetilde{W}_{h(1)} \prec \ldots \prec \widetilde{W}_{h(c) } \right\} \cap \biggl\{\mathcal{K} = \left\{h(1), \ldots, h(c) \right\}\biggr\}\right] \right)
  \\ = \  & P\left( \bigcup_{f \in \mathcal{F}_{s,c}} \left[\left\{ \widetilde{W}_{(\Psi f)(1)} \prec \ldots \prec \widetilde{W}_{(\Psi f)(c) } \right\} \cap \biggl\{\mathcal{K} = \left\{(\Psi f)(1), \ldots, (\Psi f)(c) \right\}\biggr\}\right] \right)
  \\ = \ &  \sum_{j \in \mathcal{N} \setminus \{i\}} \ \sum_{\substack{f \in \mathcal{F}_{s,c} \ | \  f(s+1)=j}} P\left( \left\{ \widetilde{W}_{g_j(f(1))} \prec \ldots \prec \widetilde{W}_{g_j(f(c)) } \right\} \cap \biggl\{\mathcal{K} =  \{ (g_j(f(1)),\ldots g_j(f(c))\} \biggr\} \right)
  \\ = \ &  P\left( \bigcup_{f \in \mathcal{F}_{s,c}} \left\{ \widetilde{W}_{f(1)} \prec \ldots \prec \widetilde{W}_{f(c) } \right\} \cap \biggl\{ \zeta_k =1 \  \forall \  k \in \{f(1), \ldots, f(c)\} \biggr\} \cap \biggl\{ \zeta_k =0 \  \forall \  k \not\in \{f(1), \ldots, f(c)\} \biggr\} \right) 
  \\ = \ & P\left(\mathcal{E}^{(s,c)} \right).
  \end{align*}

Thus, since $s \in \{1, \ldots, c-1\}$ was arbitrary, 
$P(\mathcal{E}^{(s,c)}) = P(\mathcal{E}^{(s+1,c)})$ for all $1 \leq s \leq c-1$, which establishes that 
$\mathrm{rank}(\widetilde{W}_i, \{\widetilde{W}_k\}_{k \in \mathcal{K}})$ 
is uniform on $\{1, \ldots, c\}$ conditional on 
$i \in \mathcal{K}$ and $|\mathcal{K}| = c$. 
Since $c \geq 1$ was also arbitrary, 
(\refeq{eq-rank-uniform}) follows.  Now,
  \begin{align*}
  & P\left( \frac{1}{|\mathcal{K}|} \sum_{k \in \mathcal{K}} \mathbbm{1}(W_k \geq W_i) \leq t \ \biggr\rvert \ i \in \mathcal{K}  \right)  \\ 
  \leq & \  \sum_{c=1}^n P\left( \frac{1}{|\mathcal{K}|} \sum_{k \in \mathcal{K}} \mathbbm{1}(\widetilde{W}_k \succeq \widetilde{W}_i) \leq t \ \biggr\rvert \ i \in \mathcal{K}, |\mathcal{K}| =c  \right) P\left(|\mathcal{K}| =c \ | \  i \in \mathcal{K} \right) \\
   \leq & \ t.
  \end{align*}
The claim follows.
\end{proof}

\paragraph{Step 2: Group Theory Arguments to Remove Condition (b)}
Our goal in this step is, given \ref{eq:a}, find a collection of $g_j \in \langle \mathcal{G}^{\mathcal{N}}\rangle$ for which $g_j(i)$ maps to distinct points.  We first prove Lemma \ref{lemma-invariance-group}, which establishes the basic property that any collection of invariances induces a group.  In what follows, for notational convenience, we will often use the notation $H = \langle \mathcal{G}^\mathcal{N}\rangle$.   

\begin{proof}[Proof of Lemma \ref{lemma-invariance-group}]
 For a bijection $h$ of $\mathcal{N}$, let $\varphi_h$ denote the map
$(x_k)_{k \in \mathcal{N}} \mapsto (x_{h(k)})_{k \in \mathcal{N}}$, so that
$hL = \varphi_h(L)$ and $\varphi_h \circ \varphi_g = \varphi_{g \circ h}$.

If $gL \stackrel{d}{=} L$ and
$hL \stackrel{d}{=} L$, then applying the map $\varphi_h$ to
both sides of $\varphi_g(L) \stackrel{d}{=} L$ gives
$(g \circ h) L \stackrel{d}{=} hL \stackrel{d}{=} L$. Thus, compositions also preserve invariance.  
 
For inversion, let $g$ be a bijection of $\mathcal{N}$ with
$gL \stackrel{d}{=} L$.  Applying the map
$\varphi_{g^{-1}}$ to both sides of $\varphi_g(L) \stackrel{d}{=} L$ gives
\begin{align*}
\varphi_{g^{-1}}\bigl( \varphi_g(L) \bigr) \stackrel{d}{=}
\varphi_{g^{-1}}(L),
\end{align*}
since equality in distribution is preserved by any measurable
transformation.  The left-hand side equals
$\varphi_{g \circ g^{-1}}(L) = \varphi_{\mathrm{id}}(L) = L$, and the
right-hand side is $g^{-1} L$.  Hence $g^{-1} L \stackrel{d}{=} L$.


That $H$ is a group is immediate from its definition: it contains the
identity, is closed under composition, and closed
under inversion. Moreover, each generator $h_\ell$ preserves the law of $L$,
either by (\refeq{eq-invariance-L}) or by the inversion fact above, so
induction on $m$ using the composition fact gives
$hL \stackrel{d}{=} L$ for every $h \in H$.

\end{proof}

 
We now introduce some group theory needed to state the construction of a collection of functions that satisfies condition (b).  These definitions may be found in standard textbooks on permutation groups such as \citet{cameron1999permutation}, although the construction considered here is not standard.  

Define a relation on $\mathcal{N} \times \mathcal{N}$ by
\begin{align*}
(x,y) \sim (x',y') \quad \Longleftrightarrow \quad
h(x) = x' \ \text{ and } \ h(y) = y' \ \text{ for some } h \in H .
\end{align*}
Since $H$ is a group by Lemma \ref{lemma-invariance-group}, $\sim$ is an
equivalence relation: reflexivity follows from the identity, symmetry from
closure under inversion, and transitivity from closure under composition.  We
refer to these equivalence classes as orbitals, which are orbits of the action $h \cdot (x,y) = (h(x), h(y)) \in \mathcal{N}^2$. We write $\Lambda$ for a
generic orbital.  For an orbital $\Lambda$ and $x \in \mathcal{N}$, define the suborbital:
\begin{align*}
\Gamma_\Lambda(x) = \{ k \in \mathcal{N} : (x,k) \in \Lambda \} .
\end{align*}


Let $\Lambda^{*} = \{ (y,x) : (x,y) \in \Lambda \}$ denote the transposed
orbital.  It is clear that transposed orbitals also form an equivalence class (i.e. are also orbitals), and we have the pairing $(x,y) \in \Lambda \iff (y,x) \in \Lambda^*$.  It is also immediate that $(\Lambda^*)^* = \Lambda$. 

In what follows, for $j \in \cN \setminus \{i\}$ let
\begin{align*}
  S_j = \{ h(i) : h \in H, \ h(j) = i \}.
\end{align*}
We establish some properties that will allow us to construct a function that satisfies $(a)$ and $(b)$.   The lemma we prove below establishes a connection between the class of invariances that satisfy $(a)$ for a particular choice of $j$ and a transposed orbital. 
\begin{lemma}
  \label{lem:Sj}
  Let $j \in \cN \setminus \{i\}$ and let $\Lambda$ be the orbital containing
  $(i,j)$.  Then $S_j = \Gamma_{\Lambda^{*}}(i)$.
\end{lemma}

\begin{proof}
  Let $\Delta$ be the orbital containing $(j,i)$; by definition of
  transposition, $\Delta = \Lambda^{*}$.

  If $h \in H$ satisfies $h(j) = i$ and $h(i) = k$, then
  $h \cdot (j,i) = (i,k)$, so $(i,k) \in \Delta$, i.e.\
  $k \in \Gamma_{\Delta}(i)$.  Conversely, if $k \in \Gamma_{\Delta}(i)$ then
  $(i,k) \in \Delta$ and $(j,i) \in \Delta$ lie in a common orbit, so there
  is $h \in H$ with $h \cdot (j,i) = (i,k)$, that is $h(j) = i$ and
  $h(i) = k$; hence $k \in S_j$.  Therefore
  $S_j = \Gamma_{\Delta}(i) = \Gamma_{\Lambda^{*}}(i)$.
\end{proof}

We now establish a lemma that will justify the existence of a bijection between $\Gamma_{\Lambda}(i)$ and $\Gamma_{\Lambda^*}(i)$.  This bijection will be used to construct a map from $\mathcal{N} \setminus \{i\}$ to itself that respects the constraint $(a)$. In our application, the output of the function may be interpreted as the value of $g_j(i)$; if it is a bijection, then each function maps to a unique point, thereby satisfying $(b)$.    
\begin{lemma}
  \label{lem:sizes}
  Let $H$ be a transitive group of bijections of the finite set $\cN$. For every orbital $\Lambda$ and $x,y \in \mathcal{N}$, $|\Gamma_\Lambda(x)| = |\Gamma_{\Lambda^{*}}(y)|$.
\end{lemma}

\begin{proof}
 Given $x, x' \in \cN$, transitivity provides $h \in H$ with
  $h(x) = x'$.  If $y \in \Gamma_{\Lambda}(x)$ then $(x,y) \in \Lambda$, so
  $h \cdot (x,y) = (x', h(y)) \in \Lambda$ and $h(y) \in \Gamma_{\Lambda}(x')$.  Thus
  $y \mapsto h(y)$ maps $\Gamma_{\Lambda}(x)$ into $\Gamma_{\Lambda}(x')$, and applying the
  same argument to $h^{-1}$ gives a map $\Gamma_{\Lambda}(x') \to \Gamma_{\Lambda}(x)$ inverse
  to it.  Hence, $|\Gamma_{\Lambda}(x)| = |\Gamma_{\Lambda}(x')|$.

  Since $\Lambda$ is the disjoint union over $x \in \cN$ of
  $\{x\} \times \Gamma_{\Lambda}(x)$, it follows that
  $|\Lambda| = |\cN| \cdot |\Gamma_\Lambda(x)|$, and likewise
  $|\Lambda^{*}| = |\cN| \cdot |\Gamma_{\Lambda^{*}}(y)|$.  Transposition is a
  bijection $\Lambda \to \Lambda^{*}$, so $|\Lambda| = |\Lambda^{*}|$, and
  dividing by $|\cN|$ gives the claim.
\end{proof}

We now establish the existence of the bijection that implies $(b)$ holds from condition $(a)$ alone. 
\begin{lemma}
  \label{lem:sdr}
  Let $H$ be a transitive group of bijections of the finite set $\cN$ and let
  $i \in \cN$.  Then there is a bijection
  $\phi : \cN \setminus \{i\} \to \cN \setminus \{i\}$ with
  $\phi(j) \in S_j$ for every $j$.
\end{lemma}

\begin{proof}
  By Lemma \ref{lem:Sj}, for any orbital $\Lambda$, every $j \in \Gamma_\Lambda(i)$ has the same admissible
  set $\Gamma_{\Lambda^{*}}(i) = S_j$, and by Lemma \ref{lem:sizes} these two sets have the same cardinality.  Choose any bijection
  $\phi_\Lambda : \Gamma_\Lambda(i) \to \Gamma_{\Lambda^{*}}(i)$ and define
  $\phi(j) = \phi_\Lambda(j)$ for $j \in \Gamma_\Lambda(i)$.

   The sets $\Gamma_\Lambda(i)$ over the equivalence classes of off-diagonal $\Lambda$ partition $\cN \setminus \{i\}$ (it can be readily verified that the diagonal orbital leads to $\Gamma_{\Lambda}(i) = \{i\}$), so $\phi$ is
  well defined on $\cN \setminus \{i\}$, and $\phi(j) \in S_j$ by
  construction.  For injectivity, note that as $\Lambda$ ranges over the
  off-diagonal orbitals so does $\Lambda^{*}$. Since
  each $\phi_\Lambda$ is injective and the images are disjoint, $\phi$ is
  injective, and being an injection of a finite set to itself, a bijection.
\end{proof}

\paragraph{Step 3: Putting the Pieces Together}
Now, we invoke arguments developed in Step 2 to argue that condition \ref{eq:b} can be dropped.  
\begin{proposition}[Equivalence of Conditions]
Let $\mathcal{G}^\mathcal{N}$ be a collection of bijections $\mathcal{N} \mapsto \mathcal{N}$. Then, the existence of $\mathcal{G}_0^{\mathcal{N}} \subseteq \langle \mathcal{G}^{\mathcal{N}}   \rangle$ satisfying conditions \ref{eq:a} and \ref{eq:b} of Proposition \ref{prop-super-unform-be}  for some $i \in \mathcal{N}$ is equivalent to the existence of $\mathcal{G}_0^{\mathcal{N}} \subseteq \langle \mathcal{G}^{\mathcal{N}}   \rangle $ satisfying:
\begin{enumerate}[label=(\alph*)]
\setcounter{enumi}{2} 
    \item \label{eq:c} $g_j(i) =j$  for some $i \in \mathcal{N}$ and each $j \in \mathcal{N} \setminus \{i\}$.   
\end{enumerate} 

\end{proposition}
\begin{proof}
\ref{eq:a} and \ref{eq:b} $\implies$ \ref{eq:c} is immediate since \ref{eq:a} alone implies $\langle \mathcal{G}^{\mathcal{N}} \rangle$ is transitive.  We now argue $\impliedby$. By transitivity implied by \ref{eq:c}, we have \ref{eq:c} $\implies$ \ref{eq:a}.  The conclusion follows if \ref{eq:a} $\implies$ \ref{eq:b}.  By Lemma \ref{lem:sdr}, there exists a collection of bijections $\mathcal{G}_{0}^\mathcal{N} = \{g_j \ | \  g_j \in \langle  \mathcal{G}^\mathcal{N} \rangle, j \in \mathcal{N} \setminus \{i\} \}$ satisfying $g_j(j) = i$, $g_j(i) = \phi(j)$, each mapping to unique values. The conclusion follows.    
\end{proof}
\end{proof}

\begin{proof}[Second proof of Theorem \ref{theorem-super-unform-be}]
In what follows, let:
\begin{align*}
  \widetilde{R}_j =  \frac{1}{|\mathcal{K}|} \sum_{k \in \mathcal{K}} \mathbbm{1}(\{\widetilde{W}_k \succeq \widetilde{W}_j\} \cap \{j \in \mathcal{K} \} ).
\end{align*}
  By Proposition \ref{prop-lexicographic-order}, it suffices to bound
  $P(\widetilde{R}_i \leq t \mid i \in \cK)$. For $j \in \cN$, define
\begin{align*}
  \Phi_j(\widetilde{L}) = \bigl( \widetilde{R}_j , \ \mathbbm{1}( j \in \cK) \bigr), \quad \Phi = (\Phi_j(\widetilde{L}))_{j \in \mathcal{N}}.
\end{align*}
Since $\langle \mathcal{G}^{\mathcal{N}} \rangle$ is transitive by assumption, we can find functions $g_j \in \langle \mathcal{G}^{\mathcal{N}} \rangle$ satisfying $g_j(i) = j$ for any $j \in \mathcal{N} \setminus \{i\}$.  Applying these functions, it can readily be seen that:    
\begin{align*}
(\Phi_{g_j(j)}(\widetilde{L}))_{j \in \mathcal{N}} = \Phi((\widetilde{L}_{g_j(j)})_{j \in \mathcal{N}}).
\end{align*}
Thus, by Proposition \ref{theorem-exchangeability-transformations}, it follows that, for any $j \in \mathcal{N}$,
\begin{align*}
\Phi_i(\widetilde{L}) \stackrel{d}{=} \Phi_j(\widetilde{L}). 
\end{align*}
 
  Examining the second coordinate of $\Phi$, we have that $P(j \in \cK)$ does not depend on
  $j \in \cN$.    Moreover, by assumption, $P(j \in \cK) >0$.  Furthermore,  
  \begin{align}
    \label{eq:joint}
    P\bigl( \widetilde{R}_j \leq t, \ j \in \cK \bigr)
      = P\bigl( \widetilde{R}_i \leq t, \ i \in \cK \bigr)
      \qquad \text{for all } j \in \cN .
  \end{align}
 
  Now, a simple counting argument yields
  \begin{align}
  \begin{split}
    \label{eq:count}
    \sum_{j \in \cN} \mathbbm{1}\bigl( \widetilde{R}_j \leq t, \ j \in \cK, |\mathcal{K}| = K \bigr)
      &= \#\{ j \in \cK : \widetilde{R}_j \leq t \} \mathbbm{1}(|\mathcal{K}| = K)
      \\&= \lfloor K t \rfloor \mathbbm{1}(|\mathcal{K}| = K) .
   \end{split}
  \end{align}
Taking expectations,
  \begin{align*}
    |\cN| \, P\bigl( \widetilde{R}_i \leq t, \ i \in \cK \bigr)
      &= \sum_{j \in \cN}
        P\bigl( \widetilde{R}_j \leq t, \ j \in \cK \bigr)
     \\&= E\left[\sum_{K = 1}^{|\mathcal{N}|}\sum_{j \in \cN}\mathbbm{1}
        \bigl( \widetilde{R}_j \leq t, \ j \in \cK, |\mathcal{K}| = K \bigr)\right] 
       \\&= E\left[\sum_{K = 1}^{|\mathcal{N}|} \lfloor Kt\rfloor \mathbbm{1}(|\mathcal{K}| = K)  \right]  
      \\ &\leq t \, E|\cK|
      \\&= t \sum_{j \in \cN} P(j \in \cK)
      = t \, |\cN| P(i \in \mathcal{K}) .
  \end{align*}
  The result follows dividing both sides by $|\mathcal{N}| P(i \in \mathcal{K})$ and invoking the definition of conditional probability. 
\end{proof}

\subsection{Connecting superuniformity to quantiles}
\label{subsec-sample-quantiles}
In the conformal prediction literature, it is also common to state prediction sets in terms of sample quantiles.  In this section, we show that this is also possible under the weaker superuniform condition even when the sample size is random.

Recall that the quantile function $Q(p)$ associated with a distribution $F$ is given by:
\begin{align*}
Q(p; F) = \inf_{x \in \mathbb{R}}\{ x \ | \ F(x) \geq p  \}.
\end{align*}
Let $\hat{q}_{p}$ denote a notion of an empirical quantile associated with data points $x_1,\ldots, x_n$:
\begin{align*}
\hat{q}_{p} = Q(p; F_n),
\end{align*}
where $F_n$ is the empirical CDF:
\begin{align*}
F_n(x) = \frac{1}{n} \sum_{i=1}^n \mathbbm{1}(x_i \leq x).
\end{align*}
Moreover, define:
\begin{align*}
G_n(x) = \frac{1}{n} \sum_{i=1}^n \mathbbm{1}(x_i < x).
\end{align*}
In what follows, we say that $x_i$ is among the $d$ smallest elements of $\{x_1, \ldots, x_n\}$ if $\sum_{j=1}^n \mathbbm{1}(x_j < x_i) \leq d-1$.  
The following property, stated without proof in \citet{NEURIPS2019_8fb21ee7}, will be useful for establishing the relationship between the superuniform property and empirical quantiles.  We provide a proof below.

\begin{lemma}
\label{lemma-quantile}
Let $x_1,\ldots, x_n \in \mathbb{R}$.  Then, for any $i \in [n]$ and $\alpha \in [0,1]$, 
\begin{align*}
x_i \leq \hat{q}_{1-\alpha} \iff  x_i \text{ among the } \lceil n(1-\alpha) \rceil \text{ smallest elements of } x_1, \ldots, x_n.
\end{align*}
\end{lemma}
\begin{proof}
We start by showing that $x_i \text{ among the } \lceil n(1-\alpha) \rceil \text{ smallest elements} \implies x_i \leq q_{1-\alpha}$.
If $x_i$ is among the $\lceil n(1-\alpha) \rceil$ smallest, then,
\begin{align*}
G_n(x_i) \leq \frac{\lceil n(1-\alpha) \rceil-1}{n} < 1-\alpha.   
\end{align*}
Thus, for any $x<x_i$, $F_n(x) < 1-\alpha$ and it follows that $x_i \leq \hat{q}_{1-\alpha}$.

Now, to show the other direction, we will prove the contrapositive.  Suppose that $x_i$ is not among the $\lceil n(1-\alpha) \rceil$ smallest.  Then, we have that:  
\begin{align*}
F_n(x_i) \geq \frac{\lceil n(1-\alpha)\rceil+1}{n} > 1-\alpha.
\end{align*}
Let $x^*$ be a maximizer of $F_n(x)$ over the $\lceil n(1-\alpha) \rceil$ smallest elements.  By assumption, $x^*<x_i$.  Furthermore,
\begin{align*}
F_n(x^*) \geq \frac{\lceil n(1-\alpha)\rceil}{n} \geq 1-\alpha.
\end{align*}
Therefore, under this hypothesis, $x_i > \hat{q}_{1-\alpha}$.   The result follows. 
\end{proof}

In what follows, let $(W_1, \ldots, W_n)$ be scalar-valued random variables, and let $\mathcal{K}$ be a potentially random subset of $[n]$. For some $i \in [n]$, define: 
\begin{align*}
\Gamma_i = \frac{1}{|\mathcal{K}|} \sum_{k \in \mathcal{K}}\mathbbm{1}(W_k \geq W_i).
\end{align*}
Furthermore, let $\hat{q}_{1-\alpha}^{\mathcal{K}}$ be the $1-\alpha$ quantile with respect to the empirical distribution of $(W_k)_{k \in \mathcal{K}}$.  We have the following result:

\begin{proposition}
\label{prop-quantile-ci}
Suppose that $W_1, \ldots, W_n$ are random variables such that for any $t \in [0,1]$, $P(\Gamma_i \leq t \ | \ i \in \mathcal{K}) \leq t$ for some $i \in [n]$.  Then, for any $\alpha \in [0,1],$
\begin{align*}
P(W_i \leq \hat{q}_{1-\alpha}^{\mathcal{K}} \ | \ i \in \mathcal{K}) \geq 1-\alpha.
\end{align*}
\end{proposition}
\begin{proof}
Let $ A_i = \{ \omega \in \Omega \ | \ i \in \mathcal{K}\}$. By the deterministic relationship established in Lemma \ref{lemma-quantile}, it follows that, for all $\omega \in A_i$, 
\begin{align*}
W_{i} \leq \hat{q}_{1-\alpha}^{\mathcal{K}} \iff W_{i} \text{ among } \lceil |\mathcal{K}|(1-\alpha) \rceil \text{ smallest elements of } (W_k)_{k \in \mathcal{K}}.
\end{align*}
Note that, for all  $\omega \in A_i$,
\begin{align*}
\Gamma_i > \alpha  \iff \text{ at most } \lceil  |\mathcal{K}|(1-\alpha) \rceil -1  \text{ elements less than } W_{i}.
\end{align*}
Therefore, for all  $\omega \in A_i$,
\begin{align*}
\Gamma_i > \alpha \implies W_i \leq \hat{q}_{1-\alpha}^{ \mathcal{K}}.
\end{align*}
Consequently,
\begin{align*}
P(W_{i} \leq \hat{q}_{1-\alpha}^{\mathcal{K}} \ | \ i \in \mathcal{K}) \geq P(\Gamma_i > \alpha \ | \ i \in \mathcal{K}) \geq 1-\alpha.
\end{align*}
\end{proof}
With split conformal prediction, one often considers an empirical quantile computed on the calibration points, which excludes the test point.  To account for this exclusion, the level of the empirical quantile is adjusted.  Let $\mathcal{K}$ be a random subset of $[n]$ that contains the test index $i$ and $W_{cal} = (W_k)_{k \in \mathcal{K}\setminus\{i\}}$ denote the calibration data, which excludes $W_i$ and $N = |\mathcal{K}|-1$.   

Let $\hat{F}_n^{cal}$ and $\hat{q}_{\tau}^{cal}$ denote empirical CDFs and level $\tau$ empirical quantiles of $W_{cal}$, respectively, where $\tau = \left(1 + 1/N \right)(1- \alpha)$.

 We have the following proposition:
\begin{proposition}
Suppose that for any $t \in [0,1]$, $P(\Gamma_i \leq t \ | \ i \in \mathcal{K}) \leq t$.  Then, for any $\alpha \in [0,1],$
\begin{align*}
P(W_i \leq \hat{q}_{\tau}^{cal} \ | \ i \in \mathcal{K}) \geq 1-\alpha.
\end{align*}
\end{proposition}

\begin{proof}
As before, let $A_i = \{\omega \in \Omega \ | \ i \in \mathcal{K} \}$. Due to the discrete nature of $\hat{F}_n^{cal}$, $\hat{q}_\tau^{cal} \in \{-\infty\} \cup \{W_k\}_{k \in \mathcal{K} \setminus \{i\} }$.  Thus, for all $\omega \in A_i$, 
\begin{align*}
W_i \leq \hat{q}_\tau^{cal} &\iff W_i \leq \text{the } \lceil N\tau \rceil \text{th smallest element of } W_{cal} 
\\ & \iff  W_i \text{ among the } \lceil N \tau \rceil  \text{ smallest elements of } (W_k)_{k \in \mathcal{K}}
\\ & \iff W_i \leq \hat{q}_{1-\alpha}^{\mathcal{K}},
\end{align*}
where the last line follows from the fact that $N\tau = |\mathcal{K}|(1-\alpha)$. 
Due to analogous reasoning used in the proof of Proposition \ref{prop-quantile-ci}, it now follows that, for all $\omega \in A_i$,
\begin{align*}
\Gamma_i > \alpha \implies   W_i \leq  \hat{q}_\tau^{cal}.  
\end{align*}
Consequently,
\begin{align*}
P( W_i \leq  \hat{q}_\tau^{cal} \ | \ i \in \mathcal{K}) \geq P(\Gamma_i > \alpha \ | \ i \in \mathcal{K}) \geq 1-\alpha.
\end{align*}
\end{proof}

\subsection{Proofs for Section \ref{sec-problem-setup}: Problem Setup and Section \ref{sec-conformal-for-sampled-elements}: Array Conformal Prediction for Sampled Elements}

We start by stating a proposition related to preservation of distributional invariance under appropriate transformations, which plays a key role in our proofs below.  This proposition, stated in \citet{lunde2023conformal}, is an adaptation of a result due to \citet{dean-verducci} and \citet{commenges-transformations} stated in the review article of \citet{kuchibhotla2021exchangeability}.     
\begin{proposition}
\label{theorem-exchangeability-transformations}
 Let $X$ be a random variable taking values in $\mathcal{X}$ and suppose that $Y = H(X)$ for some $H: \mathcal{X} \mapsto \mathcal{Y}$.  Further suppose that for some collection of functions $\mathcal{F}$ $\mathcal{X} \mapsto \mathcal{X}$,
 \begin{align}
 \label{eq-invariance-assumption}
  f(X) \stackrel{d}{=} X \ \ \forall \ f \in \mathcal{F}.  
 \end{align}
Furthermore, let $\mathcal{G}$ be a collection of functions $\mathcal{Y} \mapsto \mathcal{Y}$ and suppose that for any $g \in \mathcal{G}$, there exists a corresponding $f \in \mathcal{F}$ such that,
\begin{align}
\label{eq-transformation-assumption}
g(H(X)) = H(f(X)) \ \  a.s.    
\end{align}
Then,
\begin{align*}
g(Y) \stackrel{d}{=} Y \ \ \forall \ g \in \mathcal{G}.  
\end{align*}
\end{proposition}

\begin{proposition}
\label{prop-hyperedge-exchangeability}
The data-generating process in Example \ref{ex-hyperedge-sampling}  is jointly exchangeable.
\end{proposition}
\begin{proof}
Let $\mathcal{X}$ denote the set of all subsets of $[n]$.  For any permutation function $\sigma:[n] \mapsto [n]$, define the function $f^\sigma: \mathcal{X} \mapsto \mathcal{X}$ as:
\begin{align*}
f^\sigma(A) = \{ \sigma(x) \ | \  x \in A   \}
\end{align*}
for $A \subseteq [n]$.  Furthemore, let $\mathcal{H} = (\mathcal{H}_1, \ldots, \mathcal{H}_m)$.  By our assumption (\refeq{eq-hyper-permutation}),
$f^{\sigma} \mathcal{H}  \stackrel{d}{=} \mathcal{H}$ for any $\sigma$.  In the notation of Proposition \ref{theorem-exchangeability-transformations}, let $H$ denote the function that constructs a weighted adjacency matrix $M \in [n] \times [n]$ from $\mathcal{H}$ and $g$ be a function that transforms $M$ by applying the same permutation $\sigma$ to the row and column indices of $M$.  Clearly, for any $f^\sigma$, applying the same permutation to the row and column indices of $H(\mathcal{H})$ results in a matrix that is equal to $H(f^\sigma\mathcal{H})$.  Therefore, the result follows from Proposition \ref{theorem-exchangeability-transformations}.

\end{proof}

In what follows, we say that an array $U$ is jointly $\Sigma_{\mathcal{N}}$-exchangeable if for any $\sigma \in \Sigma_{\mathcal{N}}$,
\begin{align*}
U^\sigma \stackrel{d}{=} U.
\end{align*}
Below, we provide a proof of Theorem \ref{theorem-split-conformal}; the proof of Theorem \ref{theorem-super-uniform-array} is analogous. \\

\noindent \textit{Proof of Theorem \ref{theorem-split-conformal}}.  We start by verifying that the array $(S,M_1, M_2)$ is jointly $\Sigma_{\mathcal{N}}$-exchangeable.

To this end, we will first show that $(V,\mathcal{Z}, M_1, M_2)$ is jointly $\Sigma_{\mathcal{N}}$-exchangeable.    Since $(V,M_1,M_2)$ is jointly $\Sigma_{\mathcal{N}}$-exchangeable by Assumption \ref{assumption-je1-prime}, by Proposition \ref{theorem-exchangeability-transformations}, it suffices to show that permuting the rows and columns of $(V,M_1,M_2)$ with $\sigma \in \Sigma_{\mathcal{N}}$ permutes $(V,\mathcal{Z}, M_1, M_2)$ accordingly.  This is guaranteed by Assumption \ref{assumption-network-covariates}.

Now, since permuting the rows and columns of $(V,\mathcal{Z}, M_1, M_2)$ permutes $(Y, \boldsymbol{\hat{\mu}}, M_1, M_2)$ accordingly by Assumption \ref{assumption-permutation-invariant}, $(Y, \boldsymbol{\hat{\mu}}, M_2)$ is jointly $\Sigma_{\mathcal{N}}$-exchangeable by Proposition \ref{theorem-exchangeability-transformations} again.  Now, since the same nonconformity score is applied to all observations, it follows that $(S,M_2)$ is jointly $\Sigma_{\mathcal{N}}$-exchangeable.

Now, we verify the conditions of Theorem \ref{theorem-super-unform-be}.  Let $\mathcal{N}$ be the collection of nodes such that $P(\zeta_i= 1) >0$. For each $\{i,j\} \subseteq \mathcal{N}^2$ with $k \neq l$, let $g_{\{i,j\}}$ be the function:
\begin{align*}
g_{\{i,j\}}(\{x,y\}) = \{ \sigma_{\{i,j\}}(x), \sigma_{\{i,j\}}(y)\},
\end{align*}
where $\sigma_{\{i,j\}}$ is the permutation that swaps $i$ and $k$ and $j$ and $l$, respectively, and keeps all other elements fixed.  If $|\{i,j\} \cap \{k, l \}| = 1$, then $\sigma_{\{i,j\}}$ swaps distinct elements  $\sigma_{\{i,j\}}(\{k,l\}) = \{i,j\}$, the group invariance is transitive.  
The result follows.  \qed    
\bigskip

In the above proof, we considered the case where the array is symmetric, for which it is more natural to represent an index of an off-diagonal array element using sets rather than ordered pairs.  When the array is potentially asymmetric, the ordered pair representation is more appropriate. In this case, the class of invariances is also transitive, and thus, superuniformity continues to hold.  
In what follows, for  a set $\mathcal{N} \subseteq [n]$, let $\widetilde{\mathcal{N}}^2 = \{(i,j) \in \mathcal{N}^2 \ | \ i \neq j \}$ denote the set of off-diagonal elements of $\mathcal{N}^2$.  The following proposition is an immediate consequence of transitivity of the group of interest.    

\begin{proposition}
\label{prop-ordered-pair-case}
Let $\Sigma^{*}_{\mathcal{N}}$ denote the group of permutations of
$\mathcal{N}$, and for $\sigma \in \Sigma^{*}_{\mathcal{N}}$ let
$g_\sigma : \widetilde{\mathcal{N}}^2 \mapsto \widetilde{\mathcal{N}}^2$ be
given by $g_\sigma((i,j)) = (\sigma(i), \sigma(j))$.  Then the collection
\begin{align*}
\mathcal{G} = \left\{ g_\sigma \ | \ \sigma \in \Sigma^{*}_{\mathcal{N}}
\right\}
\end{align*}
acts transitively on $\widetilde{\mathcal{N}}^2$.  Consequently, if the
distributional invariance condition (\refeq{eq-invariance-L}) holds for
$\mathcal{G}$, then the conditions of
Theorem \ref{theorem-super-unform-be} are satisfied.
\end{proposition}
 
\begin{proof}
Let $(i,j)$ and $(k,l)$ be elements of $\widetilde{\mathcal{N}}^2$.  Since
$i \neq j$ and $k \neq l$, there exists $\sigma \in \Sigma^{*}_{\mathcal{N}}$
with $\sigma(i) = k$ and $\sigma(j) = l$, and the corresponding $g_\sigma$
satisfies $g_\sigma((i,j)) = (k,l)$.  Hence $\mathcal{G}$ acts transitively on
$\widetilde{\mathcal{N}}^2$, and so does the group it generates; the
conclusion follows from Theorem \ref{theorem-super-unform-be}.
\end{proof}

\noindent \textit{Proof of Theorem \ref{theorem-row-column}.}  
By similar reasoning to the proof of Theorem \ref{theorem-super-uniform-array} above, $(V, \mathcal{Z}, M_1, M_2)$ is jointly exchangeable.  Now, notice that permuting the rows and columns of this array results in permuting the rows and columns of $(\mathcal{U}, M_1, M_2 )$ accordingly.  Therefore, this array is also jointly exchangeable.

Now consider the array $(S(\mathcal{U}_{ij}), M_1(i,j), M_2(i,j) )_{1 \leq i,j \leq n}$, where $S(\mathcal{U}_{ij}) =0 $ if $(i,j) \not \in \mathcal{H}_2$.  By construction, $S(\mathcal{U}_{ij}) = \|\hat{F}_{\pm}(\mathcal{U}_{ij}) \|$, where:
\begin{align*}
\hat{F}_{\pm} = \argmin_{T: T(x) \sim \hat{\mu}} \int \| x - T(x) \| \ d\hat{\nu}(x),
\end{align*}
\begin{align*}
\hat{\nu} = \frac{1}{|\mathcal{H}_2|} \sum_{(i,j) \in \mathcal{H}_2}\mathcal{U}_{ij}, \quad \hat{\mu} = \frac{1}{|\mathcal{H}_2|} \sum_{(i,j) \in \mathcal{H}_2} a_{ij},
\end{align*}
and $(a_{ij})_{(i,j) \in \mathcal{H}_2}$ is a set of low discrepancy points from the reference measure $\mu$.  Notice that the loss function:
\begin{align*}
\frac{1}{|\mathcal{H}_2|}\sum_{(i,j) \in \mathcal{H}_2} \| \mathcal{U}_{ij} - T(\mathcal{U}_{ij}) \|
\end{align*}
is invariant under relabeling of the pairs and depends only on order statistics.  Therefore, relabeling $\mathcal{U}$ corresponds to relabeling $(S(\mathcal{U}), M_1, M_2)$ accordingly.  Thus, by Proposition \ref{theorem-exchangeability-transformations}, $(S(\mathcal{U}), M_2)$ is jointly $\Sigma_\mathcal{N}^*$ exchangeable.  The result now follows from Proposition \ref{prop-ordered-pair-case}. \qed     
\bigskip

\noindent \textit{Proof of Theorem \ref{thm:column-selective-cp}.}
It suffices to show that $(S_{ir})_{(i,j) \in \mathcal{H}_2(\cdot, r)}$ is conditionally exchangeable. Note that under the conditions of Theorem \ref{theorem-row-column}, $(S_{ij}, \mathcal{Z}_{ij}, M_{ij})_{1 \leq i,j \leq n}$ is jointly exchangeable. Now, observe that:
\begin{align*}
\ & \ P\left(\bigcap_{j \in \mathcal{B}} \{S_{r\sigma(j)} \leq t_j\}  \ \biggr \rvert \ \mathcal{R} = r, M_{2,r} = \mathcal{B} \right)
 \\ = &  \sum_{ \tau(m) = r } \frac{ P\left(\bigcap_{j \in \mathcal{B}} \{S_{r\sigma(j)} \leq t_j\} \cap \{\hat{Z}_{\tau(1)} < \cdots < \hat{Z}_{\tau(n)} \}  \cap \{M_{2,r} =\mathcal{B}  \}\right) }{P(\{\mathcal{R} = r\}  \cap \{M_{2,r} = \mathcal{B}\}  )} 
 \\ \stackrel{(i)}{=} &  \sum_{ \tau(m) = r } \frac{ P\left(\bigcap_{j \neq r} \{S_{\sigma(r)\sigma(j)} \leq t_j\} \cap \{\hat{Z}_{\sigma(\tau(1))} < \cdots < \hat{Z}_{\sigma(\tau(n))} \} \cap \{M_{2,r}^\sigma = \mathcal{B}\}\right)}{P(\{\mathcal{R} = r\}  \cap \{M_{2,r} = \mathcal{B} \}  )}
  \\  \stackrel{(ii)}{=} &  \sum_{ \tau(m) = r } \frac{ P\left(\bigcap_{j \neq r} \{S_{rj} \leq t_j\} \cap \{\hat{Z}_{\tau(1)} < \cdots < \hat{Z}_{\tau(n)}\} \cap \{M_{2,r} = \mathcal{B}\} \right)}{P(\{\mathcal{R} = r\}  \cap \{M_{2,r} = \mathcal{B}\}  )} 
 \\ = & \  P\left(\bigcap_{j \neq r} \{S_{rj} \leq t_j\}  \ \biggr \rvert \ \mathcal{R} = r, M_{2,r} = \mathcal{B} \right),
\end{align*}
where $(i)$ follows from the fact  the set of all possible orderings of $\hat{Z}_1, \ldots \hat{Z}_n$ satisfying $\mathrm{rank}(\hat{Z}_r, \{\hat{Z}_1, \ldots, \hat{Z}_{n}\}) = m$  remains unchanged when applying $\sigma$ and $\{M_{2r}^\sigma = \mathcal{B}\} = \{M_{2r} = \mathcal{B}\}$.  $(ii)$ follows from the fact that $((S_{ij})_{1 \leq i, j \leq n}, \mathcal{Z}, M_2)$ is jointly exchangeable.  As discussed in Remark \ref{remark-rank-non-unique}, when ranks are not uniquely defined, the above argument goes through without modification with respect to the lexicographic rank with auxiliary i.i.d. $\mathrm{Uniform}[0,1]$ random variables. The result follows.  \qed

\subsection{Proofs for Section \ref{sec-weighted-conformal}: Array Conformal Prediction for Missing Elements}

In the proof below, we will make use of a coupled version of $M$, denoted $M' = (M_{ij})_{1 \leq i,j \leq n}$.  To this end, suppose that for $1 \leq i < j \leq n$:
\begin{align*}
M_{ij}' = \mathbbm{1}(\eta_{ij}' \leq w(\xi_i, \xi_j)),
\end{align*}
where $\eta_{kl}' = \eta_{kl}$ and for all other $1 \leq i< j \leq n$, $\eta_{ij}' \stackrel{d}{=} \eta_{ij}$ and is independent of all other random variables.  Let $\mathcal{K}_{kl}' = \{\{i,j\} \ | \  M_{ij}' = 0, \ i,j \in D_2 \}$ and $\hat{\rho}_n' = |\mathcal{K}_{kl}'|/{N \choose 2} $.  Further define the following quantities:
\begin{align*}
\hat{G}(x) &=\frac{1}{{N \choose 2}(1- \hat{\rho}_n)} \sum_{\{i,j\} \in \mathcal{K}_{kl}} \frac{1-\hat{p}_{ij}}{\hat{p}_{ij}} \mathbbm{1}(S_{ij} \geq x),
\\ \tilde{G}(x) &=\frac{1}{|\mathcal{K}_{kl}' |} \sum_{\{i,j\} \in \mathcal{K}_{kl}'}  \mathbbm{1}(S_{ij} \geq x).
\end{align*}

The following lemma will be instrumental to proving unconditional validity:

\begin{lemma}
\label{lemma-unconditional-weighted}
Under the conditions of Theorem \ref{theorem-graphon-missigness-unconditional}, conditional on $M_{kl} = 0$, 
\begin{align}
\label{eq-unconditional-weighted}
\left|\hat{G}(S_{kl})- \tilde{G}(S_{kl})  \right| = o_P(1).
\end{align}
\end{lemma}

By triangle inequality, we have:
\begin{align*}
\left|\hat{G}(S_{kl})- \tilde{G}(S_{kl})  \right| &\leq 
 \underbrace{\frac{1-\rho_n}{1-\hat{\rho}_n }}_{A} \cdot \underbrace{\frac{2}{{N \choose 2}(1-\rho_n)} \sum_{\{i,j\} \in \mathcal{K}_{kl}} \left|\frac{1-\hat{p}_{ij}}{\hat{p}_{ij}} - \frac{1-p_{ij}}{p_{ij}}  \right|}_{B}
 \\ & \ +  \frac{1-\rho_n}{1-\rho_n } \cdot\underbrace{\left|\frac{1}{{N \choose 2}(1-\rho_n) } \sum_{\lfloor \frac{n}{2} \rfloor+1 \leq i < j \leq n } (1-p_{ij})\mathbbm{1}(S_{ij} \geq S_{kl}) \cdot \left\{\frac{\mathbbm{1}(M_{ij} = 1)}{p_{ij}} - 1 \right\}\right|}_{C}
 \\ & \ +  \frac{1-\rho_n}{1-\hat{\rho}_n } \cdot\underbrace{\left|\frac{1}{{N \choose 2}(1-\rho_n) } \sum_{\lfloor \frac{n}{2} \rfloor+1 \leq i < j \leq n} \mathbbm{1}(S_{ij} \geq S_{kl}) \cdot \left\{\mathbbm{1}(M_{ij}' = 0)- (1-p_{ij}) \right\}\right|}_{D}
 \\ & \ + \underbrace{\left| \frac{1-\hat{\rho}_n}{1-\hat{\rho}_n'} - 1 \right|}_{E} 
 \\ &= A +B+C+D + E, \text{ say.}
\end{align*}
For $A$, we have that, conditional on $M_{kl} = 1$ and $\xi_1,\ldots,\xi_n$,
\begin{align*}
{1-\hat{\rho}_n} = \frac{1}{{n \choose 2}} \sum_{1 \leq i < j \leq n, \ (i,j) \neq (k,l)}(1-M_{ij}) \stackrel{P}{\rightarrow} 1.
\end{align*}
Since $\rho_n \rightarrow 0$, $A \stackrel{P}{\rightarrow} 1$. For $B$, we have, by Cauchy-Schwarz inequality, 
\begin{align*}
 B  &\leq  \sqrt{\frac{\rho_n}{{N \choose 2} (1-\rho_n)^2  } \sum_{1 \leq i < j \leq n} \left| \frac{\hat{p}_{ij} - p_{ij}}{\hat{p}_{ij} p_{ij}}  \right|^2} \times \sqrt{ \frac{1}{{N \choose 2} \rho_n} \sum_{\lfloor \frac{n}{2} \rfloor +1 \leq i < j \leq n} M_2(i,j)}  \\ 
  & =  \sqrt{\frac{1}{n^2 \rho_n^3}\|\hat{P} - P  \|_F^2} \times O_P(1).
\end{align*}

Now, notice that, for $n$ large enough since $P(M_{kl} = 0) \rightarrow 1$, for any $\epsilon >0$,
\begin{align*}
P\left( \frac{1}{n^2 \rho_n^3}\|\hat{P} - P  \|_F^2 > \epsilon  \ \biggr\rvert \ M_{kl} = 0 \right) &\leq 2 P\left( \left\{ \frac{1}{n^2 \rho_n^3}\|\hat{P} - P  \|_F^2 > \epsilon \right\} \cap \{M_{kl} = 0\} \right)  
\\ & \leq 2 P\left( \frac{1}{n^2 \rho_n^3}\|\hat{P} - P  \|_F^2 > \epsilon \right) \rightarrow 0.
\end{align*}

Now, for $C$, conditioning on $(X_i, \eta_i, \zeta_i)_{1 \leq i \leq n}$ and $(\nu_{ij})_{1 \leq i,j \leq n}$, only $(M_{ij})_{1 \leq i,j \leq n}$ is random.  Therefore, by Hoeffding's inequality for weighted sums, it follows that, for some $K>0$,
\begin{align*}
P(C > \epsilon \ | \ M_{kl} = 0) \leq 2\exp\left( \frac{-Kn^2 \epsilon^2}{\rho_n^2} \right).
\end{align*}
This converges to $0$ under condition (b).  Now for $D$, note that $M_{ij}'$ still has expectation $(1-p_{ij})$; therefore, by reasoning similar to the bound for $C$, this term converges in probability to $0$.  

Conditional on $\xi_1, \ldots, \xi_n$, both the numerator and denominator concentrate with high probability around their expectation. In the sparse case, the expectations converge to $1$, and thus, $E$ must converge in probability to $0$ conditional on $M_{kl} = 0$ by standard continuity arguments.   \qed
\bigskip

\noindent \textit{Proof of Theorem \ref{theorem-graphon-missigness-unconditional}}. 
Define the array $\left(V_{ij}, M_{ij}\mathbbm{1}(i,j \in D_1 ), M_{ij}'\mathbbm{1}(i,j \in D_2 ) \right)_{1 \leq i,j \leq n}$.  This array satisfies condition \ref{assumption-je1-prime}.  Moreover, under node splitting, the permutation-invariant model condition is not required since $(Y_{ij}, \hat{\mu}_{ij},(1-M_{ij}') \mathbbm{1}(i,j \in D_2 )  )_{1 \leq i < j \leq n}$ is $\Sigma_{D_2}$ exchangeable. Therefore, by analogous reasoning used in the proof of Theorem \ref{theorem-split-conformal}, it follows that:
\begin{align*}
P(Y_{kl} \in \{ y \ | \ \tilde{G}(y) > \alpha\} \ | \ M_{kl} = 0) \geq 1-\alpha.
\end{align*}
Now, for any $\epsilon > 0$, for $n$ large enough, it follows that:
\begin{align*}
& \ P(Y_{kl} \in \widehat{C}_{kl}^\mathrm{weighted} \ | \ M_{kl} = 0) \\   \geq & \  P(Y_{kl} \in \{y  \ | \ \tilde{G}(y) - |\hat{G}(y) - \tilde{G}(y)  | > \alpha\} \ | \ M_{kl} = 0) 
\\ \geq & \ P(\tilde{G}(Y_{kl})  > \alpha +\epsilon/2 \ | \ M_{kl} = 0) - P(|\hat{G}(Y_{kl}) - \tilde{G}(Y_{kl})| > \epsilon/2 \ | \ M_{kl} = 0)
\\ \geq & \ 1-\alpha-\epsilon .
\end{align*}      
\qed

\bigskip

\noindent \textit{Proof of Theorem \ref{theorem-asymptotic-conditional-validity-graphon}}.
In what follows, we prove the binary and continuous cases separately. \bigskip

\noindent \textbf{Binary Case:} \\ 
In the binary case, let $S_{ij}^* = s(Y_{ij},\tilde{Y}_{ij}; \pi) \vee (1-\alpha)$, $\tilde{S}_{ij}^* = s(Y_{ij},\hat{Y}_{ij}; \hat{\pi})$,  and $\widetilde{S}_{ij}^* = s(Y_{ij},\tilde{Y}_{ij}; \pi) $. For an arbitrary $\delta >0$ define: 
\begin{align*}
\tilde{G}^*(x) = \frac{1}{|\mathcal{K}_{kl}' |}\sum_{ \{i,j \} \in \mathcal{K}_{kl}' } \mathbbm{1}(R_{ij}^* \geq x),
\end{align*}
where for  $\{i,j\} \neq \{k,l\}$:
\begin{align*}
R_{ij}^* = ((S_{ij}^* - \delta) \mathbbm{1}(\mathcal{A}_{ij}) + (1-\alpha)\mathbbm{1}(\mathcal{A}_{ij}^c), \quad \mathcal{A}_{ij} = \{S_{ij}^* > 1-\alpha + \delta\} \cap \{\widetilde{S}_{kl}^* > 1-\alpha - \delta\}
\end{align*}
Furthermore, let:
\begin{align*}
R_{kl}^* = S_{kl}^* +\delta \mathbbm{1}(\widetilde{S}_{kl}^* >1-\alpha-\delta ).
\end{align*}

By Lemma \ref{lemma-binary-conditional}, 
for both settings $(i)$ and $(ii)$, we have $\hat{G}(S_{kl}) - \tilde{G}^*(R_{kl}^*) = o_P(1)$. It follows that:

\begin{align*}
& \ P\left( Y_{kl} \in \widehat{C}_{kl}^\mathrm{weighted} \ | \ \tilde{Y}_{kl}, M_{kl}=0 \right) \\ 
= & \ P\left( Y_{kl} \in \left\{y   \ | \  \Pi_{kl}^\mathrm{weighted} > \alpha \right\} \ | \ \tilde{Y}_{kl}, M_{kl} =0  \right)  
\\ \geq & \ P\left( Y_{kl} \in \left\{y   \ | \  \tilde{G}^*(R_{kl}^*) > \alpha + \epsilon \right\} \ | \ \tilde{Y}_{kl}, M_{kl} = 0  \right) - R_n^{(1)}
\\ = &\  P\left(Y_{kl} \in \left\{ y \ | \  R_{kl}^* \leq \hat{q}_{1-\alpha-\epsilon}^{\mathcal{K}_0'} \right\} \ | \ \tilde{Y}_{kl}, M_{kl} = 0 \right) - R_n^{(1)}
\\ \geq  & \ P\left(Y_{kl} \in \left\{ y \ | \  S_{kl}^* + \delta \mathbbm{1}(\widetilde{S}_{kl}^* > 1-\alpha -\delta)  \leq 1-\alpha \right\}  \ | \ \tilde{Y}_{kl}, M_{kl} = 0 \right) - R_n^{(1)},
\end{align*}
where $P(R_n^{(1)} > \epsilon) \rightarrow 0$ by Markov's inequality.  In what follows, let $y_{(1)} = y_{(1)}(\tilde{Y}) = \arg \min_{ y \in \{0,1\}} 
\pi_y(\tilde{Y})$.  It now follows that:

\begin{align*}
&P\left( Y_{kl} \in \widehat{C}_{kl}^\mathrm{weighted} \ | \ \tilde{Y}_{kl}, M_{kl}=0 \right) 
\\ \geq &  1- P(\{Y_{kl} = y_{(1)}\} \cap \{\pi_{y_{(1)}}(\tilde{Y}_{kl}) <\min(0.5,\alpha+\delta)\} \ | \ \tilde{Y}_{kl}, M_{kl} = 0) -R_n^{(1)} 
\\ \geq & 1-\alpha-\delta-R_n^{(1)}.
\end{align*}
Since the choice of $\delta >0$ was arbitrary, the result follows.

\qed

\noindent \textbf{Continuous Case:}

 In what follows, let $S_{ij} =|1/2 - \hat{F}_{Y|\hat{Y},M_{ij =0}}(Y_{ij} \ | \  \hat{Y}_{ij})|$, $S_{ij}^* = |1/2 - F_{Y|\tilde{Y},M_{ij =0}}(Y_{ij} \ | \  \tilde{Y}_{ij})|$. Define the quantities:
 \begin{align*}
 \hat{F}(x) &= \frac{1}{{N \choose 2}(1- \hat{\rho}_n)} \sum_{\{i,j\} \in \mathcal{K}_{kl}} \frac{1-\hat{p}_{ij}}{\hat{p}_{ij}} \mathbbm{1}(S_{ij} \leq x),
 \\ F^*(x) &= P(S_{ij}^* \leq x \ | \  \{i,j\}\text{ missing})
 \end{align*}
 For notational convenience, let $\tilde{P}(\cdot) = P(\cdot \ | \ \{k,l\} \text{ missing}, \tilde{Y}_{kl})$. In the continuous case, we have:
 \begin{align*}
 \tilde{P}( \hat{F}(S_{kl}) \leq 1-\alpha) = \tilde{P}( \hat{G}(S_{kl}) > \alpha).
 \end{align*}
 Notice that:
\begin{align*}
\tilde{P}(Y_{kl} \in \hat{C}) &= \tilde{P}(\hat{F}(S_{kl}) \leq  1-\alpha) \\ 
& \geq  \tilde{P}\left( |\hat{F}(S_{kl}) - F^*(S_{kl}^*)| + |F^*(S_{kl}) - F^*(S_{kl}^*)| + F^*(S_{kl}^*)     \leq 1-\alpha \right) \\ 
 & \geq \tilde{P}\left( F^*(S_{kl}^*) \leq 1-\alpha - \epsilon \right) - \tilde{P}\left( |\hat{F}(S_{kl}) - F^*(S_{kl}^*)| > \epsilon/2 \right) - \tilde{P}( |F^*(S_{kl}) - F^*(S_{kl}^*)| >\epsilon/2)    \\ 
 &= I - II - III, \text{ say.}
\end{align*}
By the inverse probability transform, for $I$, we have that $F^*(S_{kl}^*) \sim \mathrm{Uniform}[0,1]$ conditional on $\tilde{Y}_{kl}$, $M_{kl}= 0$.  Therefore $I = 1-\alpha - \epsilon$.  Now, for $III$, we have, by inverse transform sampling again,
\begin{align*}
F^*(x) = P\left( \left| \frac{1}{2}- \tilde{F}(Y_{ij} \ | \ \tilde{Y}_{ij})\right| \leq x \right) =2x \text{ for } x \in[0,1/2].
\end{align*}
Therefore, $F^*(x)$ is $2$-Lipschitz.  It follows that \begin{align*}
III \leq  \tilde{P}( 2|S_{kl}-S_{kl}^*| >\epsilon/2) \stackrel{P}{\rightarrow} 0
\end{align*}  
by condition (e1) or (e2), as appropriate, and Markov's inequality. $II \stackrel{P}{\rightarrow} 0$ by Lemma \ref{lemma-continuous-conditional} and Markov's inequality. \qed

\begin{lemma}
\label{lemma-binary-conditional}
Suppose that the conditions of Theorem \ref{theorem-graphon-missigness-unconditional} hold.  Moreover, suppose that condition (d1) or (d2) of Theorem \ref{theorem-asymptotic-conditional-validity-graphon} holds, as appropriate for the without- or with-network-covariates case, respectively. Then, conditional on $M_{kl} = 0$,
\begin{align}
\label{eq-conditional-binary}
\hat{G}(S_{kl}) - \tilde{G}^*(R_{kl}^*) = o_P(1).
\end{align}
\end{lemma}
\begin{proof}
We consider the without and with network covariate cases separately. \\  \\ 
\noindent \textbf{(i) Without Network Covariates:} \\
Consider the bound:
\begin{align*}
\hat{G}(S_{kl}) - \tilde{G}^*(R_{kl}^*)  = \hat{G}(S_{kl}) - \tilde{G}(S_{kl}) + \tilde{G}(S_{kl}) - \tilde{G}^*(R_{kl}^*) .  
\end{align*}
By Lemma \ref{lemma-unconditional-weighted}, $\hat{G}(S_{kl}) - \tilde{G}(S_{kl}) = o_P(1)$.  For the latter term, we have:  
\begin{align*}
|\tilde{G}(S_{kl}) - \tilde{G}^*(R_{kl}^*)| & \leq \frac{1-\rho_n}{1-\rho_n'} \cdot \frac{1}{{N \choose 2}(1-\rho_n)}\sum_{\lfloor \frac{n}{2}\rfloor+1 \leq i < j \leq n } [\mathbbm{1}(|\widetilde{S}_{ij}-\widetilde{S}_{ij}^*| > \delta)+ \mathbbm{1}(|\widetilde{S}_{kl}-\widetilde{S}_{kl}^*| > \delta) ].  
\end{align*}
Conditional on $M_{kl} = 0$, this term converges to $0$ by Markov's inequality and the fact that $P(M_{kl} = 0) \rightarrow 1$.  Thus, (\refeq{eq-conditional-binary}) holds under assumption $(i)$. \\ \\
\textbf{(ii) With Network Covariates:} \\ 
Consider the decomposition:
\begin{align*}
& \hat{G}(S_{kl}) - \tilde{G}^*(R_{kl}^*) \\ 
\geq &  \underbrace{\frac{1-\rho_n}{1-\hat{\rho}_n}}_I \cdot \underbrace{\frac{1}{{N \choose 2}(1- \rho_n)} \sum_{\{i,j\} \in \mathcal{K}_{kl}} \frac{1-p_{ij}}{p_{ij}} \mathbbm{1}(R_{ij}^* \geq R_{kl}^*)}_{II}
\\ -& \frac{1-\rho_n}{1-\hat{\rho}_n} \cdot \underbrace{\frac{1}{{N \choose 2}(1- \rho_n)} \sum_{\{i,j\} \in \mathcal{K}_{kl}} \left|\frac{1-\hat{p}_{ij}}{\hat{p}_{ij}} - \frac{1-p_{ij}}{p_{ij}}  \right|}_{III}
\\ -& \frac{1-\rho_n}{1-\hat{\rho}_n} \cdot \underbrace{\frac{1}{{N \choose 2}(1- \rho_n)} \sum_{ \{i,j\} \in \mathcal{K}_{kl} } \frac{1 - \hat{p}_{ij}}{\hat{p}_{ij}} [\mathbbm{1}(|\widetilde{S}_{ij}-\widetilde{S}_{ij}^*| > \delta) + \mathbbm{1}(|\widetilde{S}_{kl}-\widetilde{S}_{kl}^*| > \delta)]}_{IV}.
\end{align*}

Conditional on $M_{kl} = 0$, $I \stackrel{P}{\rightarrow} 1$ and $III \stackrel{P}{\rightarrow} 0$ by the same reasoning used in the proof of Theorem \ref{theorem-graphon-missigness-unconditional} to bound terms $A$ and $B$, respectively.  For $IV$, we have:

 \hspace*{-2.5cm}\vbox{\begin{align*}
|IV| &\leq 2 III + \frac{1-\rho_n}{1-\hat{\rho}_n} \cdot \underbrace{\left|\frac{1}{{N \choose 2}(1- \rho_n)} \sum_{\lfloor \frac{n}{2}\rfloor+1 \leq i < j \leq n } (1-p_{ij})\left(\frac{M_{ij}}{p_{ij}} -1 \right) [\mathbbm{1}(|\widetilde{S}_{ij}-\widetilde{S}_{ij}^*| > \delta) +  \mathbbm{1}(|\widetilde{S}_{kl}-\widetilde{S}_{kl}^*| > \delta) ]\right|}_{IV_a} \\  &+ \underbrace{\frac{1}{{N \choose 2}(1-\rho_n)}\sum_{\lfloor \frac{n}{2}\rfloor+1 \leq i < j \leq n } [\mathbbm{1}(|\widetilde{S}_{ij}-\widetilde{S}_{ij}^*| > \delta)+ \mathbbm{1}(|\widetilde{S}_{kl}-\widetilde{S}_{kl}^*| > \delta) ]}_{IV_b}.
\end{align*}}

\hspace*{-2.5cm}\vbox{\begin{align*}
|IV_a| \leq \sqrt{\frac{1}{{N \choose 2}(1- \rho_n)} \sum_{\lfloor \frac{n}{2}\rfloor+1 \leq i < j \leq n } \left(\frac{M_{ij}}{p_{ij}} -1 \right)^2}  \sqrt{\frac{1}{{N \choose 2}(1- \rho_n)} \sum_{\lfloor \frac{n}{2}\rfloor+1 \leq i < j \leq n } \mathbbm{1}(|\widetilde{S}_{ij}-\widetilde{S}_{ij}^*| > \delta) + \mathbbm{1}(|\widetilde{S}_{kl}-\widetilde{S}_{kl}^*| > \delta)}
\end{align*}}
The term in the square root on the LHS is $o_P(\rho_n)$ conditional on $M_{kl} = 0$ by Markov's inequality and the fact that:
\begin{align*}
& \max_{\substack{\lfloor \frac{n}{2}\rfloor+1 \leq i < j \leq n \\ \{i,j\} \neq \{k,l\} }} E\left(\frac{M_{ij}}{p_{ij}} -1 \ \biggr\rvert \ M_{kl} = 0 \right)^2
\\ \leq & \max_{\substack{\lfloor \frac{n}{2}\rfloor+1 \leq i < j \leq n \\ \{i,j\} \neq \{k,l\} }} E\left[E\left(\frac{M_{ij}}{p_{ij}} -1 \ \biggr\rvert \ \xi_i,\xi_j \right)^2 \  \biggr\rvert \ M_{kl} = 0 \right] \\
\leq & \frac{\rho_nC(1-\rho_nC)}{c \rho_n^2}.
\end{align*}
The $\{k,l\}$ term is lower-order since $n \rho_n \rightarrow \infty$ by assumption.  Similarly, $IV_b \stackrel{P}{\rightarrow} 0 $ conditional on $M_{kl} =0$ by Markov's inequality.   

Now for $II$, we have:
\begin{align*}
|II| &\leq \underbrace{\left|\frac{2}{{N \choose 2}} \sum_{\lfloor \frac{n}{2}\rfloor+1 \leq i < j \leq n } (1-p_{ij})\left(\frac{M_{ij}}{p_{ij}} -1\right)\mathbbm{1}(R_{ij}^* \geq R_{kl}^*) \right|}_{II_a}
\\ & +  \underbrace{\left|\frac{2}{{N \choose 2}} \sum_{\lfloor \frac{n}{2}\rfloor+1 \leq i < j \leq n } \left(M_{ij}' -(1-p_{ij})\right)\mathbbm{1}(R_{ij}^* \geq R_{kl}^*) \right|}_{II_b}.
\end{align*}
Observe that conditional on $(\xi_1,\zeta_1, X_1), \ldots, (\xi_n,\zeta_n, X_n)$, $(R_{ij}^*, R_{kl}^*)$ are fixed, and $(M_{ij})_{1 \leq i < j \leq n}$ are mutually independent.  Therefore, by Hoeffding's inequality for weighted sums, $II_a \stackrel{P}{\rightarrow} 0$.  $II_b \stackrel{P}{\rightarrow} 0$ by analogous reasoning. 
\end{proof}
\begin{lemma}
\label{lemma-continuous-conditional}
Suppose that the conditions of Theorem \ref{theorem-graphon-missigness-unconditional} hold.  Moreover, suppose that condition (e1) or (e2) of Theorem \ref{theorem-asymptotic-conditional-validity-graphon} holds, as appropriate for the without- or with-network-covariates case, respectively. Then, conditional on $M_{kl} = 0$,
\begin{align*}
\left| \hat{F}(S_{kl}) - F^*(S_{kl}^*)   \right|  = o_P(1).
\end{align*}

\end{lemma}

\begin{proof}
We will prove an upper bound; the lower bound is analogous. We again consider separate bounds for cases not involving network covariates and cases involving network covariates.
In what follows let:
\begin{align*}
\tilde{F}(x) = \frac{1}{{N \choose 2}(1-\rho_n)} \sum_{\{i,j\} \in \mathcal{K}_{kl}' } \mathbbm{1}(S_{ij} \leq x), \quad \tilde{F}^*(x) = \frac{1}{{N \choose 2}(1-\rho_n)} \sum_{\{i,j\} \in \mathcal{K}_{kl}' } \mathbbm{1}(S_{ij}^* \leq x). 
\end{align*}
\\  
\noindent \textbf{(i) Without Network Covariates:} \\ 
Consider the decomposition:
\begin{align*}
\hat{F}(S_{kl}) - F^*(S_{kl}^*) = \hat{F}(S_{kl}) - \tilde{F}(S_{kl}) + \tilde{F}(S_{kl}) -  F^*(S_{kl}^*).  
\end{align*}
By arguments made in the proof of Lemma \ref{lemma-unconditional-weighted}, for any $\delta > 0$, the first term converges in probability to $0$ conditional on $M_{kl} = 0$.  For the second term, we have:  
\begin{align*}
\tilde{F}(S_{kl}) -  F^*(S_{kl}^*) \leq  & \ \underbrace{\tilde{F}^*(S_{kl}^*+ \delta) -  F^*(S_{kl}^* +\delta)}_{A} \\
& + \underbrace{F^*(S_{kl}^* +\delta)- F^*(S_{kl}^*)}_{B} \\ 
& + \underbrace{\frac{1-\rho_n}{1-\rho_n'} \cdot \frac{1}{{N \choose 2}(1-\rho_n)}\sum_{\lfloor \frac{n}{2}\rfloor+1 \leq i < j \leq n } [\mathbbm{1}(|S_{ij}-S_{ij}^*| > \delta)+ \mathbbm{1}(|S_{kl}-S_{kl}^*| > \delta) ]}_{C}.
\end{align*}
Conditional on $M_{kl} = 0$,  $C \stackrel{P}{\rightarrow} 0$ by arguments made in Lemma \ref{lemma-binary-conditional}. 

Now, for $A$, observe that $(S_{ij}^*, M_{ij})_{i,j \in \mathbb{N}}$ is a jointly exchangeable and diassociated array satisfying Assumption 1 of \citet{10.1214/20-AOS1981}.  Moreover, $\mathcal{F} = \{\mathbbm{1}( \{x\leq t\}\ \cap \{y=1 \}) \ \ | \ t \in \mathbb{R}\}$ is a VC class satisfying Assumption 3(i) of the above reference.  Therefore, it follows that:
\begin{align*}
A\leq \sup_{x \in \mathbb{R}} |\tilde{F}^*(x) -F^*(x)| \stackrel{P}{\rightarrow} 0.   
\end{align*}

For $B$, observe that, since $P(M_{kl} = 0) \rightarrow 1$, 
\begin{align*}
  B &\leq 2 \sup_{x \in \mathbb{R}}\{P( \{S_{ij}^* \leq x + \delta\} \cap \{M_{ij} = 0\}) -  P( \{S_{ij}^* \leq x\} \cap \{M_{ij} = 0\})\} \\ 
  & \leq   2\sup_{x \in \mathbb{R}} P( x \leq S_{ij}^* \leq x + \delta) \leq 4\delta.
\end{align*}
where the last inequality follows since $S_{ij}^* \sim \mathrm{Uniform}[0,1/2]$. Since $\delta $ can be chosen so that this term is arbitrarily small, $B \rightarrow 0$.  The result follows. \qed

\noindent \textbf{(ii) With Network Covariates:}
Consider a representation that is analogous to the binary case with network covariates: 
\begin{align*}
&\frac{1}{{N \choose 2}(1- \hat{\rho}_n)} \sum_{\{i,j\} \in \mathcal{K}_{kl}} \frac{1-\hat{p}_{ij}}{\hat{p}_{ij}} \mathbbm{1}(S_{ij} \leq x) - F^*(x) \\ 
 \leq & \underbrace{\frac{1-\rho_n}{1-\hat{\rho}_n}}_I \cdot \underbrace{\frac{1}{{N \choose 2}(1- \rho_n)} \sum_{\{i,j\} \in \mathcal{K}_{kl}} \frac{1-p_{ij}}{p_{ij}} \mathbbm{1}(S_{ij}^* \leq x+\delta) - F^*(x)}_{II}
\\ +& \frac{1-\rho_n}{1-\hat{\rho}_n} \cdot \underbrace{\frac{1}{{N \choose 2}(1- \rho_n)} \sum_{\{i,j\} \in \mathcal{K}_{kl}} \left|\frac{1-\hat{p}_{ij}}{\hat{p}_{ij}} - \frac{1-p_{ij}}{p_{ij}}  \right|}_{III}
\\ +& \frac{1-\rho_n}{1-\hat{\rho}_n} \cdot \underbrace{\frac{1}{{N \choose 2}(1- \rho_n)} \sum_{\{i,j\} \in \mathcal{K}_{kl}} \frac{1-\hat{p}_{ij}}{\hat{p}_{ij}}[\mathbbm{1}(|S_{ij}-S_{ij}^*| > \delta) + \mathbbm{1}(|S_{kl}-S_{kl}^*| > \delta) ] }_{IV}
\end{align*}
The terms $I, III$, and  $IV$, may be bounded using the same arguments used in the proof of the binary case of Theorem \ref{theorem-asymptotic-conditional-validity-graphon}.  For $II$, consider the bound:
\begin{align*}
II &\leq \underbrace{\frac{1}{{N \choose 2}(1- \rho_n)} \sum_{\lfloor \frac{n}{2}\rfloor+1 \leq i < j \leq n} (1-p_{ij})\left(\frac{M_{ij}}{p_{ij}}-1\right) \mathbbm{1}(S_{ij}^* \leq x+\delta)}_{II_a}
\\ & + \underbrace{\frac{1}{{N \choose 2}(1- \rho_n)} \sum_{\lfloor \frac{n}{2}\rfloor+1 \leq i < j \leq n} \left(1-p_{ij} - M_{ij}'\right) \mathbbm{1}(S_{ij}^* \leq x+\delta)}_{II_b}
\\ &+ \underbrace{\frac{1}{{N \choose 2}(1- \rho_n)} \sum_{\lfloor \frac{n}{2}\rfloor+1 \leq i < j \leq n} \mathbbm{1}(\{S_{ij}^* \leq x+\delta\} \cap \{M_{ij}' = 0\}) - F^*(x+\delta)}_{II_c}
\\ &+ \underbrace{F^*(x+\delta)- F^*(x)}_{II_d}.
\end{align*}
Observe that, conditional on $(E,X,\xi)$ and $(\eta_{ij})_{i,j \in D_1}$, the only random term is $ (M_{ij})_{i,j\in D_2}$.  Moreover, the vector $( \mathbbm{1}(S_{ij}^*\leq x ))_{ \lfloor \frac{n}{2}\rfloor+1 \leq i < j \leq n}$ indexed by $x \in \mathbb{R}$ takes at most ${N \choose 2}$ different values.  Thus, by a union bound and Hoeffding's inequality,
\begin{align*}
P\left(\sup_{x \in \mathbb{R}}|II_a| > \epsilon \ | \ M_{kl}=0\right) \leq {N \choose 2} \times \exp\left( \frac{-4{ N \choose 2} \left(\epsilon+\frac{1}{{N \choose 2}}\right)^2}{ \rho_n^2 c^2 }  \right) \rightarrow 0.
\end{align*}
$II_b \stackrel{P}{\rightarrow} 0$  conditional on $M_{kl} = 0$ by similar reasoning.  $II_c \stackrel{P}{\rightarrow} 0$ and $II_d\stackrel{P}{\rightarrow} 0$ by arguments made to bound the terms $A$ and $C$, respectively, in the ``without network covariate'' case.  The result follows.

\end{proof}

\bigskip

\noindent \textit{Proof of Proposition \ref{prop-selective-coverage}.}
Let $\tilde{T}_{kl} = \mathbbm{1}(\eta_{kl} \leq f(\tilde{Y}_{kl}))$.   We will first show that:
\begin{align}
\label{eq-bernoulli-diff}
P(T_{kl} \neq \tilde{T}_{kl}) \rightarrow 0.
\end{align}
Let $D_f$ denote the set of discontinuity points of $f$ and $m$ denote the number of such discontinuities.  Observe that, for any $\epsilon >0$,
\begin{align*}
& \ P(T_{kl} \neq \tilde{T}_{kl}) \\ 
\leq & \  2 \left[P([\tilde{Y}_{kl} - \epsilon, \tilde{Y}_{kl} + \epsilon] \cap D_f \neq \emptyset) + P(|\hat{Y}_{kl} - \tilde{Y}_{kl}| > \epsilon/L) + P(f(\tilde{Y}_{kl}) - \epsilon/L \leq \gamma_{kl} \leq f(\tilde{Y}_{kl}) + \epsilon/L) \right].
\end{align*}

For the first term, we have:
\begin{align*}
P([\tilde{Y}_{kl} - \epsilon, \tilde{Y}_{kl} + \epsilon] \cap D_f \neq \emptyset) &=  P\left(\bigcup_{x \in D_f} \tilde{Y}_{kl} \in [x-\epsilon, x+\epsilon]\right)
\\ &\leq 2Km\epsilon,
\end{align*}
where $K$ is supremum of the density of $\tilde{Y}_{kl}$.  The second term converges to $0$ since $\hat{Y}_{kl} - \tilde{Y}_{kl} = o_P(1)$ and the third term can be made arbitrarily small by anti-concentration properties of the uniform distribution. Thus, (\ref{eq-bernoulli-diff}) holds.

Now, observe that the term $P(\{M_{2}(k,l) = 0\} \cap \{\tilde{T}_{kl} =1 \})$ satisfies the following lower bound:
\begin{align*}
& P(\{M_{2}(k,l) = 0\} \cap \{\tilde{T}_{kl} =1 \})
\\ = & E[P(M_2(k,l) = 0 \ | \ \xi_k, \xi_l) \cdot P( T_{kl} = 1 \ | \ \tilde{Y}_{kl}) ]
\\ \geq & (1-\rho_n C) \gamma \delta_0. 
\end{align*}
Therefore, for $n$ large enough, it follows that:

\hspace*{-1cm}\vbox{\begin{align*}
& \left| P( Y_{kl} \in \widehat{C}_{kl}^\mathrm{weighted} \ | \ M_2(k,l) = 0, T_{kl} = 1) - P( Y_{kl} \in \widehat{C}_{kl}^{\mathrm{weighted}} \ | \ M_2(k,l) = 0, \tilde{T}_{kl} = 1)   \right| 
\\ \leq &\frac{8 P( \{\tilde{T}_{kl} = 1\} \triangle \{T_{kl} = 1\})}{\gamma^2 \delta_0^2} \rightarrow 0. 
\end{align*}}
where $\triangle$ is the symmetric set difference operator. Moreover, define the event:
\begin{align*}
\mathcal{A}_\delta = \{ P( \{Y_{kl} \in \widehat{C}_{kl}^{\mathrm{weighted}} \ | \ M_{2}(k,l) = 0, \tilde{Y}_{kl}) \geq 1-\alpha-\delta \}.
\end{align*}
By Theorem \ref{theorem-asymptotic-conditional-validity-graphon}, $P(A_\delta) \rightarrow 1$.   Therefore,  
\begin{align*}
& \ P( \{Y_{kl} \in \widehat{C}_{kl}^{\mathrm{weighted}} \} \cap \{M_2(k,l) = 0\} \cap \{T_{kl} = 1\}) \\
= & \ E[P( \{Y_{kl} \in \widehat{C}_{kl}^{\mathrm{weighted}} \ | \ M_{2}(k,l) = 0, \tilde{Y}_{kl}) P(M_{2}(k,l) = 0 \ | \ \tilde{Y}_{kl}) f(\tilde{Y}_{kl})] \\ 
\geq & \ (1-\alpha-\delta) E_{\mathcal{A}_\delta}[P(M_{2}(k,l) = 0 \ | \ \tilde{Y}_{kl}) f(\tilde{Y}_{kl})].
\end{align*}
Since $E_{\mathcal{A}_\delta}[P(M_{2}(k,l) = 0 \ | \ \tilde{Y}_{kl}) f(\tilde{Y}_{kl})] \rightarrow E[P(M_{2}(k,l) = 0 \ | \ \tilde{Y}_{kl}) f(\tilde{Y}_{kl})]$ by Dominated Convergence Theorem, it follows that, for $n$ large enough:
\begin{align*}
P( Y_{kl} \in \widehat{C}_{kl}^{\mathrm{weighted}} \ | \ M_2(k,l) = 0, \tilde{T}_{kl} = 1) & \geq (1-\alpha-\delta) \frac{E_{\mathcal{A}_\delta}[P(M_{2}(k,l) = 0 \ | \ \tilde{Y}_{kl}) f(\tilde{Y}_{kl})]}{E[P(M_{2}(k,l) = 0 \ | \ \tilde{Y}_{kl}) f(\tilde{Y}_{kl})]} \\ 
& \geq 1-\alpha- 2\delta.
\end{align*}
The first claim follows. 

For the second claim, first observe that, by Markov's inequality:
\begin{align*}
& P\left( \left|\frac{1}{{N \choose 2}}\sum_{\lfloor \frac{n}{2}\rfloor +1 \leq k < l \leq n} \mathbbm{1}( \{M_{2}(k,l) = 0\} \cap \{T_{kl}=1\} ) - \mathbbm{1}( \{M_{2}(k,l) = 0\} \cap \{\widetilde{T}_{kl}=1\} ) \right| > \epsilon \right) 
\\ \leq & \frac{P(T_{kl} \neq \tilde{T}_{kl})}{\epsilon} \rightarrow 0.
\end{align*}
Therefore, for the denominator, we have:
\begin{align*}
 &\frac{1}{{N \choose 2}}\sum_{\lfloor \frac{n}{2}\rfloor +1 \leq k < l \leq n} \mathbbm{1}( \{M_{2}(k,l) = 0\} \cap \{T_{kl}=1\} ) \\ = & \ \frac{1}{{N \choose 2}}\sum_{\lfloor \frac{n}{2}\rfloor +1 \leq k < l \leq n} \mathbbm{1}( \{M_{2}(k,l) = 0\} \cap \{\tilde{T}_{kl}=1\} )  + o_P(1)
 \\ = & \ E[P( M_{2}(k,l) = 0 \ | \  \tilde{Y}_{kl}) \cdot P(\tilde{T}_{kl}=1 \ | \ \tilde{Y}_{kl} )] + o_P(1),
\end{align*}
where the last line follows from the fact that $(M_{kl}, \tilde{Y}_{kl})_{1 \leq i,j \leq n}$ is a disassociated array, and consequently, Hoeffding's inequality for U-statistics may be applied since this term permits an analogous representation as a sum over permutations involving independent blocks.   

For the numerator, we prove the continuous case; the binary case can be proved using analogous reasoning. Using the notation in the proof of Theorem \ref{theorem-asymptotic-conditional-validity-graphon}, we have:
\begin{align*}
& \ \frac{1}{{N \choose 2}}\sum_{\lfloor \frac{n}{2}\rfloor +1 \leq k < l \leq n} \mathbbm{1}\left( \left\{Y_{kl} \in \widehat{C}_{kl}^{\mathrm{weighted}}\right\} \cap \{M_{2}(k,l) = 0\} \cap \{T_{kl}=1\} \right)
\\ \geq & \   \underbrace{\frac{1}{{N \choose 2}}\sum_{\lfloor \frac{n}{2}\rfloor +1 \leq k < l \leq n} \mathbbm{1}\left( \left\{ F^*(S_{kl}^*) \leq 1-\alpha - \delta \right\} \cap \{M_{2}(k,l) = 0\} \cap \{\tilde{T}_{kl}=1\} \right)}_{A} 
\\ & - \underbrace{\frac{1}{{N \choose 2}}\sum_{\lfloor \frac{n}{2}\rfloor +1 \leq k < l \leq n} \mathbbm{1}( |\hat{F}(S_{kl})-F^*(S_{kl}^*)| > \delta)}_B + o_P(1). 
\end{align*}
For $B$, observe that $P(|\hat{F}(S_{kl}) - F^*(S_{kl}^*)| > \delta) = o(1)$ by Lemma \ref{lemma-continuous-conditional}. Therefore, $B= o_P(1)$ by Markov's inequality.  For $A$, notice that the summands are again functions of a disassociated array, and concentrate around their expectation by a variant of Hoeffding's inequality for U-statistics.  Moreover, by the argument made in the proof of Theorem \ref{theorem-asymptotic-conditional-validity-graphon}, we have: 
\begin{align*}
P( F^*(S_{kl}^*) \leq 1-\alpha -\delta \ | \  M_{2}(k,l) = 0,  \tilde{Y}_{kl} ) = 1-\alpha-\delta. 
\end{align*}
Therefore,
\begin{align*}
A \geq  (1-\alpha-\delta) \cdot E[P(M_2(k,l) =0 \ | \ \tilde{Y}_{kl})\cdot P(\tilde{T}_{kl} \ | \ \tilde{Y}_{kl})] + o_P(1).
\end{align*}
The claim now follows after combining bounds for the numerator and denominator and using the fact that the denominator is bounded away from $0$.  
\qed

\bigskip

\noindent \textit{Proof of Proposition \ref{prop-kernel-ipw}}.
It suffices to show that:
\begin{align*}
\frac{\frac{1}{{N \choose 2}(1-\hat{\rho}_n)h } \sum_{\{i,j\} \in \mathcal{H}_{1}} \frac{1-\hat{p}_{ij}}{\hat{p}_{ij}} K\left(\frac{|\hat{Y}_{ij} - \tilde{y}|}{h} \right)}{\frac{1}{{N \choose 2}(1-\rho_n)h } \sum_{1 \leq i < j \leq \lfloor \frac{n}{2} \rfloor } K\left(\frac{|\tilde{Y}_{ij} - \tilde{y}|}{h} \right)} &= 1+o_P(1),
\end{align*}
\hspace*{-2.5cm}\vbox{\begin{align*}
\left| \frac{1}{{N \choose 2}(1-\hat{\rho}_n)h} \sum_{\{i,j\} \in \mathcal{H}_{1}} \frac{1-\hat{p}_{ij}}{\hat{p}_{ij}} K\left(\frac{|\hat{Y}_{ij} - \tilde{y}|}{h} \right)y_{ij} - \frac{1}{{N \choose 2}(1-\rho_n)h} \sum_{1 \leq i < j \leq \lfloor \frac{n}{2} \rfloor } K\left(\frac{|\tilde{Y}_{ij} - \tilde{y}|}{h} \right)y_{ij} \right|
&= o_P(1).
\end{align*}}
Both statements can be proven using arguments used to prove previous results. In what follows, let $L$ denote the Lipschitz constant associated with the kernel. 

\noindent \textbf{(i) Without Network Covariates:} \\ 
Consider the decomposition:
\begin{align*}
& \frac{1}{{N \choose 2}(1-\hat{\rho}_n)} \sum_{\{i,j\} \in \mathcal{H}_{1}} \frac{1-\hat{p}_{ij}}{\hat{p}_{ij}} K\left(\frac{|\hat{Y}_{ij} - \tilde{y}|}{h} \right)y_{ij}  \\
\leq & \frac{1-\rho_n'}{1-\hat{\rho}_n} \cdot \frac{1}{|\mathcal{H}_1' |} \sum_{\{i,j\} \in \mathcal{H}_{1}'} K\left(\frac{|\tilde{Y}_{ij} - \tilde{y}|}{h} \right)y_{ij} \\ 
 & + \frac{1-\rho_n}{1-\hat{\rho}_n} \cdot  \frac{B}{{N \choose 2}(1-\rho_n)} \sum_{\{i,j\} \in \mathcal{H}_{1}} \left|\frac{\hat{p}_{ij}-p_{ij}}{\hat{p}_{ij}p_{ij}}\right| \\
 \ & + \  \frac{1-\rho_n}{1-\hat{\rho}_n} \cdot \frac{1}{{N \choose 2}(1-\rho_n)} \sum_{1 \leq i<j \leq \lfloor \frac{n}{2} \rfloor} (1-p_{ij})\left(\frac{M_{ij}}{p_{ij}} -1\right) K\left(\frac{|\hat{Y}_{ij} - \tilde{y}|}{h} \right)y_{ij} \\
 \ & + \  \frac{1-\rho_n}{1-\hat{\rho}_n} \cdot \frac{1}{{N \choose 2}(1-\rho_n)} \sum_{1 \leq i<j \leq \lfloor \frac{n}{2} \rfloor} \{(1-p_{ij})-M_{ij}'\}K\left(\frac{|\hat{Y}_{ij} - \tilde{y}|}{h} \right)y_{ij}
\\ &+  \frac{1-\rho_n}{1-\hat{\rho}_n} \cdot \frac{1}{{N \choose 2}(1-\rho_n)} \sum_{1 \leq i<j \leq \lfloor \frac{n}{2} \rfloor} \frac{L\left| \hat{Y}_{ij} - \tilde{Y}_{ij} \right|}{h}.
\end{align*}
This decomposition is analogous to the one used in Lemma \ref{lemma-unconditional-weighted}; it is clear that the second part of the proposition holds repeating similar arguments.   

\noindent \textbf{(ii) With Network Covariates:} \\ 
Consider the decomposition:
\begin{align*}
& \frac{1}{{N \choose 2}(1-\hat{\rho}_n)} \sum_{\{i,j\} \in \mathcal{H}_{1}} \frac{1-\hat{p}_{ij}}{\hat{p}_{ij}} K\left(\frac{|\hat{Y}_{ij} - \tilde{y}|}{h} \right)y_{ij}  \\ 
\leq \ & \ \frac{1-\rho_n}{1-\hat{\rho}_n} \cdot \frac{1}{{N \choose 2}(1-\rho_n)} \sum_{\{i,j\} \in \mathcal{H}_{1}} \frac{1-p_{ij}}{p_{ij}} K\left(\frac{|\tilde{Y}_{ij} - \tilde{y}|}{h} \right)y_{ij} 
\\ & + \frac{1-\rho_n}{1-\hat{\rho}_n} \cdot \frac{1}{{N \choose 2}(1-\rho_n)} \sum_{\{i,j\} \in \mathcal{H}_{1}} \left|\frac{1-\hat{p}_{ij}}{\hat{p}_{ij}} - \frac{1-p_{ij}}{p_{ij}}\right|
\\ & + \frac{1-\rho_n}{1-\hat{\rho}_n} \cdot \frac{1}{{N \choose 2}(1-\rho_n)} \sum_{\{i,j\} \in \mathcal{H}_{1}} \frac{1-\hat{p}_{ij}}{\hat{p}_{ij}} \frac{L\left| \hat{Y}_{ij} - \tilde{Y}_{ij} \right|}{h}.
\end{align*}
This representation is analogous to the one used in the proofs of Lemmas \ref{lemma-binary-conditional} and Lemma \ref{lemma-continuous-conditional}. Repeating similar reasoning, the result follows.

\subsection{Additional Experimental Details}
\subsubsection{Comparing Splitting Mechanisms}
\label{sec-ex1-appendix}
In this section, we provide additional details for the comparison of node splitting, edge splitting, and selected column methods discussed in Section \ref{sec-ex1}. For this and the following section, the test element $Y_{ij}$ is observed.

The node covariates for each node $i$ are generated as:
\begin{align*}
\mathbf{X}_i = 
\begin{bmatrix}
X_{i1} \\ X_{i2} \\ X_{i3} \\ X_{i4}
\end{bmatrix}
\sim \mathcal{N} \left(\begin{bmatrix}
1 \\ 3 \\ 0 \\ 0
\end{bmatrix}, \begin{bmatrix}
    1 & 0.3 & 0.7 & 0.6 \\
    0.3 & 4 & 0.5 & 0.8 \\
    0.7 & 0.5 & 1 & 0.5 \\
    0.6 & 0.8 & 0.5 & 1
\end{bmatrix}
\right),
\end{align*}
where $X_{i1}$ and $X_{i2}$ are the observable covariates for fitting the model, and $X_{i3}$ and $X_{i4}$ are used to construct the latent positions:
\begin{align*}
\xi_{i1} = \alpha_3\frac{|X_{i3}|}{\|(X_{i3},X_{i4})\|_2},
\quad
\xi_{i2} = \alpha_4\frac{|X_{i4}|}{\|(X_{i3},X_{i4})\|_2},
\end{align*}
where $\alpha_3, \alpha_4 \sim U[0.8,1]$ are the feature amplifying factors. This transformation maps $[X_{i3},X_{i4}]$ to the first quadrant of the unit circle and rescales the two dimensions by $\alpha_3$ and $\alpha_4$, so that the latent positions $[\xi_{i1}, \xi_{i2}]$ lie on an ellipse restricted to the first quadrant, with semi-axes being $\alpha_3$ and $\alpha_4$. 



The response $Y_{ij}$ and the masking matrix $M_{ij}$ are as defined in Section \ref{sec-ex1}. The constants $c_1=0.9$ and $c_2 = 0.25$ in \eqref{eq-exp-compare-split} are chosen so that the observable sigmoid signal ($P_{sigmoid}$) and latent linear signal ($P_{linear}$) have approximately equal power, and the signal-to-noise ratio is approximately $20$, i.e. $\frac{P_{sigmoid}}{P_{linear}}\approx 1$ and $SNR=\frac{P_{sigmoid}+P_{linear}}{P_{noise}}\approx 20$, where the power of a component $S$ is defined as $P_{S} = \sum_{i=1,j=1}^{n} S_{ij}^2/n^2$. The responses are standardized and rescaled to a standard deviation of $4$ across all synthetic examples.

We now describe the splitting mechanisms in more detail. For node splitting, the nodes are partitioned into $D_1 = \{1, \ldots, \lfloor n/2 \rfloor\}$ and $D_2 = \{\lfloor n/2 \rfloor + 1, \ldots, n\}$. The training set consists of observed entries among pairs in $D_1$, and the calibration set consists of observed entries among pairs in $D_2$. The test point is $Y_{n-1,n}$, and we only look at the cases where the test point is observed. For edge splitting, the observed entries are randomly partitioned into two disjoint sets, with the test point again being $Y_{n-1,n}$. Since $Y$ and $M$ are symmetric, for both node splitting and edge splitting methods, we use only the upper-triangular entries with $i<j$.

For the selected column method, the calibration set consists of the observed entries in the column with the highest degree, computed after removing the corresponding row with the highest degree to avoid information leakage. The test point is the entry in the $(n-1)$-th row of the selected column, and we again only consider the cases where it is observed. The training set consists of all remaining observed entries, excluding the selected column and its corresponding row. For this method, we use the full array rather than restricting to upper-triangular entries, so that the test point exists in the array. 

\subsubsection{Row-Column Approach} 
\label{sec-ex2-appendix}
In this section, we provide additional details for the row-column approach discussed in Section \ref{sec-ex2}, together with its comparison with the standard node splitting and edge splitting methods. The data-generating process, including the basic, asymmetric, and heteroscedastic cases, is described in Section \ref{sec-ex2}.

We define two sets used in the calibration procedure. The reference set consists of all observed entries sharing a row or column with any entry in the calibration set, including the calibration set itself, and is used for computing the row and column ranks of the prediction error on the calibration set. The update set consists of all entries in the reference set that share a row or column with the test point; only the row and column ranks in the intersection of the update set and the calibration set are recomputed.

Suppose that the index of the test point is pre-specified as $(k,l)$. In the following, we describe the calibration procedure for the basic case. For the asymmetric and heteroscedastic cases, the calibration procedures are identical to the basic case, except that the training and calibration sets include all off-diagonal entries rather than only upper-triangular ones. 

To construct the prediction interval, we scan through candidate values for the absolute residual $S_{kl} = |Y_{kl} - \boldsymbol{\hat{\mu}}|$. Since only the row and column ranks of entries that share a row or column with the test point are affected by the value of $S_{kl}$, the rank pairs can only change when $S_{kl}$ crosses another score in the update set. The candidate set therefore consists of the absolute residuals of the $m$ entries in the update set, excluding the test element, in descending order:
\[ S_{i_1 j_1} >S_{i_2 j_2}  > \cdots >S_{i_{m} j_{m}}.\]
Starting from a value larger than $ S_{i_1 j_1} $, we set $S_{kl}$ to this candidate value, and compute the row and column ranks for each entry $S_{ij}$ in the calibration set:

\begin{align*}
 R_{ij} &= \frac{1}{|\mathcal{H}_2(i,\cdot)|}\sum_{(i,j') \in \mathcal{H}_2} \mathbbm{1}(\{S_{ij'} \geq S_{ij} \}  \cap \left\{(k,l) \in \mathcal{H}_2 \right\}) \\
  C_{ij} &= \frac{1}{|\mathcal{H}_2(\cdot,j)|}\sum_{(i',j) \in \mathcal{H}_2} \mathbbm{1}(\{S_{i'j} \geq S_{ij} \}  \cap \left\{ (k,l) \in \mathcal{H}_2 \right\}),
\end{align*}
where $\mathcal{H}_2$ denotes the full (unrestricted) calibration set, which also serves as the reference set. Let $\mathcal{H}'_2$ denote the restricted calibration set that only includes the upper diagonal elements. The row and column pairs $\mathcal{U} = (R_{ij}, C_{ij})_{i,j\in \mathcal{H}'_2}$ are then mapped to the grid $\mathcal{A} = (a_{ij})_{(i,j)\in\mathcal{H}'_2}$ on the unit circle translated to the first quadrant via optimal transport, where the target distribution is generated using the good lattice point method of \cite{Fang1993NumbertheoreticMI} with generating vector $(n;1,89)$. 

The nonconformity score of each entry in the calibration set is defined as $S(\mathcal{U}_{ij}) = \| a_{ij}\|_2$. Next, we examine whether the following quantity
\begin{align*}
\check{\Pi}_{kl}^{\mathrm{row-col}} = \frac{1}{|\mathcal{H}'_2|}\sum_{(i,j) \in \mathcal{H}'_2}\mathbbm{1}(S(\mathcal{U}_{ij}) \geq S(\mathcal{U}_{kl}))
\end{align*}
exceeds $\alpha$. If $\check{\Pi}_{kl} > \alpha$, the current candidate determines the prediction interval:
\begin{align*}
\widehat{C}^{\mathrm{row-col}}_{kl} = \{ y_{kl} \mid \check{\Pi}_{kl}^{\mathrm{row-col}}>\alpha \}.
\end{align*}
Otherwise, we proceed to the next candidate in descending order and repeat the procedure until $\check{\Pi}_{kl} > \alpha$ first occurs. Since $m\leq n-1$, at most $n-1$ candidates need to be examined. With node splitting, the update set contains at most $n$ entries. 

We note that the order of the candidate values is set to be descending, as scanning in descending order substantially reduces the number of candidates that need to be examined in practice. In a special case of link prediction, the edge response is binary and the conformal prediction set can only take one of the three possible values: $\{0\}$, $\{1\}$ or $\{0,1\}$. Therefore, at most three candidate values need to be examined, making the row-column approach particularly efficient for link prediction problems.

We now briefly describe the row-column approach with edge splitting. In the edge splitting method, the observed entries are randomly partitioned into two disjoint sets $\mathcal{H}_1$ and $\mathcal{H}_2$. The training set consists of observed entries $(i,j) \in \mathcal{H}_1$, and the calibration set consists of observed entries $(i,j) \in \mathcal{H}_2$. For the basic case, we again restrict to $i<j$, while all off-diagonal entries are included for the asymmetric and heteroscedastic cases. The reference set and update sets are defined in the same way as the node splitting case, excluding the entries from the training set. 

The calibration procedure is identical to the one described above, with the update set containing at most $2n-1$ entries.

We report the results for the row-column approach with edge splitting ($\rho_n = 0.4$) with $n=200$, $\alpha =0.1$ and $500$ iterations in Table \ref{tab-ex2-compare-row-column-edge}. The row-column approach produces narrower prediction intervals than the standard splitting methods in the asymmetric and heteroscedastic cases, while the two are similar in the basic case. The coverage of the edge splitting method is more variable, which is likely due to the sparsity of the array at $\rho_n = 0.4$.

\begin{table}[ht]
\centering
\begin{tabular}{ |c|c|c|c|c| }
 \hline
 Case & Method & Model & Coverage & Width \\
 \hline
 \multirow{4}{*}{1}
     & \multirow{2}{*}{Standard} & Linear & 0.906 & 9.21 \\
     &                           & RF     & 0.892 & 7.65 \\
 \cline{2-5}
     & \multirow{2}{*}{Row-Column} & Linear & 0.898 & 9.40 \\
     &                             & RF     & 0.938 & 7.67 \\
 \hline
 \multirow{4}{*}{2}
     & \multirow{2}{*}{Standard} & Linear & 0.916 & 9.68 \\
     &                           & RF     & 0.888 & 7.39 \\
 \cline{2-5}
     & \multirow{2}{*}{Row-Column} & Linear & 0.886 & 9.60 \\
     &                             & RF     & 0.874 & 7.23 \\
 \hline
 \multirow{4}{*}{3}
     & \multirow{2}{*}{Standard} & Linear & 0.906 & 8.98 \\
     &                           & RF     & 0.902 & 7.34 \\
 \cline{2-5}
     & \multirow{2}{*}{Row-Column} & Linear & 0.898 & 8.90 \\
     &                             & RF     & 0.858 & 7.02 \\
 \hline
\end{tabular}
\caption{Comparison of standard conformal prediction sets with the row-column approach under edge splitting ($\rho = 0.4$). The underlying model ranges from (1) basic (symmetric and homoscedastic), (2) asymmetric, and (3) heteroscedastic.}
\label{tab-ex2-compare-row-column-edge}
\end{table}

\subsubsection{Missing Elements}
\label{sec-ex3-appendix}
In this section, we provide additional details for the weighted conformal prediction of a missing element in the array, as discussed in Section \ref{sec-ex3}. The data-generating process is defined as in Section \ref{sec-ex3}, with node covariates generated as in Section \ref{sec-ex1-appendix}.

For the setting of the weighted conformal prediction with the node splitting method, the detailed splitting mechanism is as follows. The nodes are partitioned into $D_1 = \{1, \ldots, \lfloor n/2 \rfloor\}$ and $D_2 = \{\lfloor n/2 \rfloor + 1, \ldots, n\}$. The training set consists of observed entries with indices $(i,j)$, $i,j \in D_1$ and $i<j$, and the calibration set consists of observed entries among pairs in $D_2$. The test point is a missing entry $Y_{n-1,n}$ with $M_{n-1,n} = 0$. 


In addition, we present the averaged scaled Frobenius error of the USVT and NS methods together with the coverage and width for $n=500$, $\alpha = 0.1$, and $300$ iterations given sparsity level $\rho_n$ ranging from $2n^{-0.8}$ to $\min\{ 2n^{-0.1},1\}$ in Figure \ref{fig-ex3-dependency-innerprod-graphon}, and analyze the behavior of the two estimation methods. The estimated probability matrices from the USVT and NS methods are $\hat{P}^{USVT}$ and $\hat{P}^{NS}$, respectively. The scaled Frobenius error is defined as $\| P-\hat{P} \|_{F}/\| P \|_F$ and is averaged over $300$ iterations for each sparsity level $\rho_n$. The USVT method gives a more precise estimation for very sparse graphs, which coincides with the coverage of its conformal prediction interval being closer to the nominal level $0.9$. On the other hand, we also observe that for moderate to dense graphs $\rho_n \geq n^{-0.7}$, the NS method tends to produce narrower prediction intervals. 

\begin{figure}[ht]
\centering
    \includegraphics[width=0.8\linewidth]{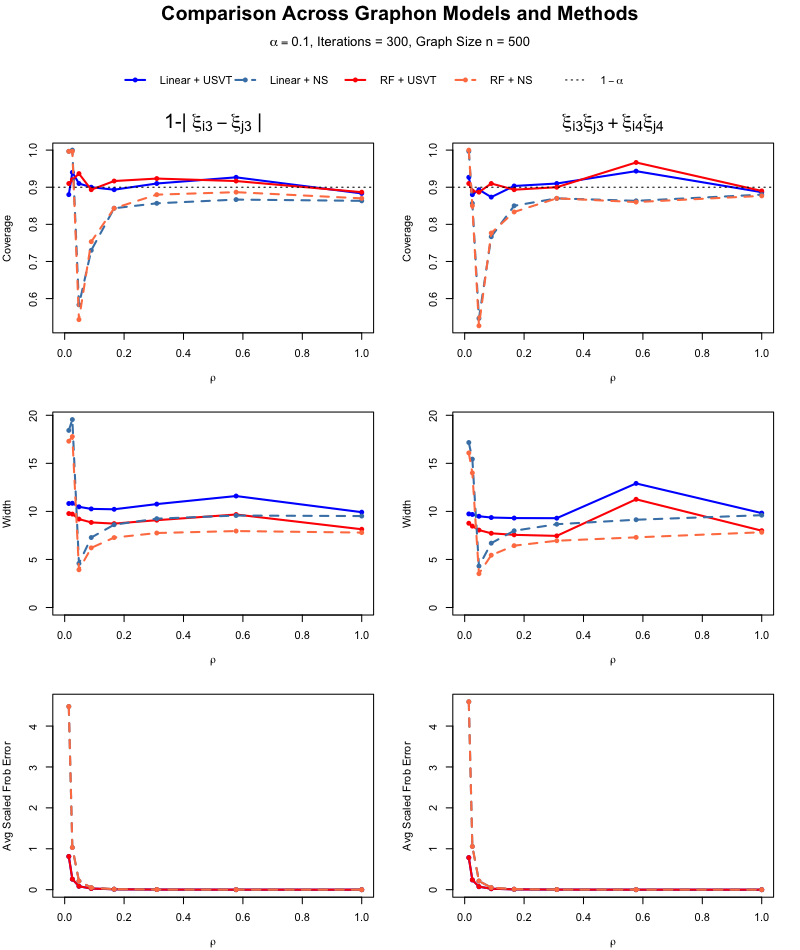}
    \caption{Top row: coverage. Middle row: width. Bottom row: averaged scaled Frobenius error $\| P-\hat{P} \|_{F}/\| P \|_F$ of the USVT and NS estimators across different sparsity levels and two graphon models.}
    \label{fig-ex3-dependency-innerprod-graphon}
\end{figure}

\subsubsection{Predicting Common Citations}
\label{sec-ex4-appendix}
In this section, we provide additional details for the Cora dataset experiment discussed in Section \ref{sec-ex4}. As described in the main text, the masking matrix $M$ is constructed from the output of a graph attention network (GATv2) \cite{brody2022how}. We now describe the architecture and training procedure used to obtain $M$.

The graph attention network takes the node feature vectors and the full citation graph $A$ as input, and outputs the predicted edge probabilities $\hat{p}_{ij}$ for all node pairs. The architecture consists of $2$ GATv2 layers with $64$ hidden dimensions, $4$ attention heads, and a dropout rate of $0.3$. The model is trained for $200$ epochs using the Adam optimizer with learning rate $0.001$ and weight decay $5 \times 10^{-4}$, achieving an AUC of $0.9650$ on the prediction $\hat{p}_{ij}$. The masking matrix $M$ is then obtained by thresholding $\hat{P}$ at $0.5$, i.e.,
\begin{align*}
    M_{ij} = \mathbbm{1} (\hat{p}_{ij} > 0.5),
\end{align*}
resulting in a density of $0.12$.

 \end{document}